\documentclass{amsart}
\title{Nakayama functor for monads on finite abelian categories}
\author[K.~Shimizu]{Kenichi Shimizu}
\email{kshimizu@shibaura-it.ac.jp}
\address{Department of Mathematical Sciences \\
  Shibaura Institute of Technology \\
  307 Fukasaku, Minuma-ku, Saitama-shi, Saitama 337-8570, Japan.}
\date{}

\usepackage{amsmath,amscd,amssymb,mathtools}
\usepackage{bbm}
\usepackage{dsfont}
\usepackage{accents}
\usepackage{hyperref}

\usepackage{tikz}
\usepackage{tikz-cd}
\usetikzlibrary{calc}

\usepackage{amsthm}
\numberwithin{equation}{section}
\newtheorem{counter}{}[section]
\theoremstyle{definition}
\newtheorem{definition}         [counter]{Definition}
\newtheorem{notation}           [counter]{Notation}

\theoremstyle{plain}
\newtheorem{lemma}              [counter]{Lemma}

\newtheorem{theorem}            [counter]{Theorem}
\newtheorem{corollary}          [counter]{Corollary}
\newtheorem{question}           [counter]{Question}

\newtheorem*{theorem*}          {Theorem}
\theoremstyle{remark}
\newtheorem{remark}             [counter]{Remark}
\newtheorem{example}            [counter]{Example}

\newcommand{\id}{\mathrm{id}}
\newcommand{\eval}{\mathrm{ev}}
\newcommand{\coev}{\mathrm{coev}}
\newcommand{\op}{\mathrm{op}}
\newcommand{\cop}{\mathrm{cop}}
\newcommand{\rev}{\mathrm{rev}}
\newcommand{\env}{\mathrm{env}}
\newcommand{\unitobj}{\mathbbm{1}}
\newcommand{\Nat}{\mathrm{Nat}}
\newcommand{\Hom}{\mathrm{Hom}}

\newcommand{\radj}{\mathrm{ra}}
\newcommand{\rradj}{\mathrm{rra}}
\newcommand{\rrradj}{\mathrm{rrra}}
\newcommand{\ladj}{\mathrm{la}}
\newcommand{\lladj}{\mathrm{lla}}
\newcommand{\llladj}{\mathrm{llla}}
\newcommand{\iHom}{\underline{\mathrm{Hom}}}

\newcommand{\bfk}{\Bbbk}
\newcommand{\lmod}[1]{{#1}\text{\rm -mod}}
\newcommand{\rmod}[1]{\text{\rm mod-}{#1}}
\newcommand{\lcomod}[1]{{#1}\text{\rm -comod}}
\newcommand{\bimod}[2]{{#1}\text{\rm -mod-}{#2}}
\newcommand{\copow}{\otimes}
\newcommand{\act}{\mathord{\mathrm{act}}}
\newcommand{\REX}{\mathrm{Rex}}

\newcommand{\Vect}{\mathbf{Vec}}
\newcommand{\canalg}{\mathds{A}}
\newcommand{\catactl}{\mathbin{\triangleright}}
\newcommand{\catactr}{\mathbin{\triangleleft}}
\newcommand{\catacte}{\mathbin{\tilde{\catactl}}}
\newcommand{\Nak}{\mathds{N}}  
\newcommand{\Free}{\mathds{F}} 
\newcommand{\Forg}{\mathds{U}} 
\newcommand{\EW}{\mathrm{EW}}  
\newcommand{\can}{\boldsymbol{\phi}}
\newcommand{\Radford}{\mathrm{R}}
\begin{document}

\begin{abstract}
  If $\mathcal{M}$ is a finite abelian category and $\mathbf{T}$ is a linear right exact monad on $\mathcal{M}$, then the category $\mathbf{T}\mbox{-mod}$ of $\mathbf{T}$-modules is a finite abelian category. We give an explicit formula of the Nakayama functor of $\mathbf{T}\mbox{-mod}$ under the assumption that the underlying functor of the monad $\mathbf{T}$ has a double left adjoint and a double right adjoint.
  As applications, we deduce formulas of the Nakayama functor of the center of a finite bimodule category and the dual of a finite tensor category.
  Some examples from the Hopf algebra theory are also discussed.
\end{abstract}

\maketitle

\tableofcontents

\section{Introduction}

Given a vector space $X$ over a field $\bfk$, we denote its dual space by $X^*$.
If $A$ is a finite-dimensional algebra over $\bfk$, then the dual space $A^*$ has a natural structure of an $A$-bimodule. Hence we have an endofunctor
\begin{equation}
  \label{eq:intro-Nakayama-A-mod}
  \Nak_A : \lmod{A} \to \lmod{A}, \quad M \mapsto A^* \otimes_A M
\end{equation}
on the category $\lmod{A}$ of finite-dimensional left $A$-modules. This functor, called the Nakayama functor, plays an important role in the representation theory of finite-dimensional algebras; see, {\it e.g.}, \cite{MR2894798}.

Recently, the Nakayama functor is attracted attention from a different viewpoint, that is, the theory of finite tensor categories and their modules.
We recall that a finite abelian category \cite{MR3242743} is a linear category that is equivalent to $\lmod{A}$ for some finite-dimensional algebra $A$.
Fuchs, Schaumann and Schweigert \cite{MR4042867} remarked that the Nakayama functor \eqref{eq:intro-Nakayama-A-mod} is expressed as the coend
\begin{equation*}
  \Nak_{A}(M) = \int^{X \in \lmod{A}} \Hom_A(M, X)^* \otimes_{\bfk} X
  \quad (M \in \lmod{A}).
\end{equation*}
Following this observation, one can define the Nakayama functor $\Nak_{\mathcal{M}} : \mathcal{M} \to \mathcal{M}$ for a finite abelian category $\mathcal{M}$ by
\begin{equation*}
  \Nak_{\mathcal{M}}(M) := \int^{X \in \mathcal{M}} \Hom_{\mathcal{M}}(M, X)^* \copow X
  \quad (M \in \mathcal{M}),
\end{equation*}
where $\copow$ is the copower (see \S\ref{subsec:fin-ab-cat}).

The point is that the functor $\Nak_{\mathcal{M}}$ is defined without referencing an algebra $A$ such that $\mathcal{M} \approx \lmod{A}$.
In studying finite tensor categories and their modules, we often consider finite abelian categories for which no finite-dimensional algebras realizing them are given explicitly.
The above abstractly redefined Nakayama functor is often useful in such a situation.
A key feature of the Nakayama functor is the following relation to adjoint functors:
Given a functor $F$, we denote by $F^{\ladj}$ and $F^{\radj}$ a left and a right adjoint of $F$, respectively, if it exists.
Let $G: \mathcal{M} \to \mathcal{N}$ be a linear functor between finite abelian categories such that a double right adjoint $G^{\rradj} := (G^{\radj})^{\radj}$ exists and is right exact. By the universal property of Nakayama functors as coends, we have an isomorphism
\begin{equation}
  \label{eq:intro-cano-iso}
  \can_G :  G^{\rradj} \circ \Nak_{\mathcal{M}} \to \Nak_{\mathcal{N}} \circ G
\end{equation}
of functors that is `natural' and `coherent' in a certain sense \cite{MR4042867}.

As pointed out in \cite{MR4042867}, some fundamental results on finite tensor categories and their modules are obtained by the isomorphism~\eqref{eq:intro-cano-iso}. For example, let $\mathcal{C}$ be a finite tensor category. Then the isomorphism \eqref{eq:intro-cano-iso} yields the formula
\begin{equation*}
  \Nak_{\mathcal{C}}(X)
  \cong \Nak_{\mathcal{C}}(\unitobj) \otimes X^{\vee\vee}
  \quad (X \in \mathcal{C})
\end{equation*}
of the Nakayama functor of $\mathcal{C}$, where $\otimes$ is the tensor product, $\unitobj$ is the unit object and $(-)^{\vee}$ is the left duality functor of $\mathcal{C}$. By this formula, we see that $\mathcal{C} \approx \lmod{A}$ for some symmetric Frobenius algebra $A$ if and only if $\mathcal{C}$ is unimodular ({\it i.e.}, $\Nak_{\mathcal{C}}(\unitobj) \cong \unitobj$) and the double dual functor of $\mathcal{C}$ is isomorphic to $\id_{\mathcal{C}}$ \cite{MR4042867}. The above formula and its variant also give a categorical analogue of Radford's $S^4$-formula \cite{MR2097289} and a Frobenius type property of tensor functors between finite tensor categories \cite{MR3569179}.

For the above reasons, we think of the Nakayama functor as a fundamental tool in the theory of finite tensor categories and their modules.
It is thus natural to ask how the Nakayama functor is given for a particular class of finite abelian categories.
In this paper, we provide a formula the Nakayama functor of the category of modules over a monad on a finite abelian category under a technical assumption on the underlying functor of the monad.

To describe the main result in detail, we let $\mathbf{T} = (T, \mu, \eta)$ be a monad on a finite abelian category $\mathcal{M}$ such that the underlying functor $T : \mathcal{M} \to \mathcal{M}$ is linear and right exact.
Then, as is well-known, the category $\lmod{\mathbf{T}}$ of $\mathbf{T}$-modules is a finite abelian category such that the forgetful functor to $\mathcal{M}$ preserves and reflects exact sequences.
Now we suppose that $T$ has a left adjoint. Then there is a natural transformation $b : T^{\ladj} \circ T \to T^{\ladj}$ making $T^{\ladj}$ a `right module' over $\mathbf{T}$.
Considering $T^{\ladj}$ as an analogue of the bimodule $A^{*}$ in \eqref{eq:intro-Nakayama-A-mod}, we define
\begin{equation*}
  \mathbf{T}^{\ladj} \otimes_{\mathbf{T}} \mathbf{M}
  := \mathrm{coequalizer} \Big(
  \begin{tikzcd}[column sep = 64pt]
    T^{\ladj} T(M)
    \arrow[r, shift left=3pt, "{b_M}"]
    \arrow[r, shift right=3pt, "{T^{\ladj}(a_M)}"']
    & T^{\ladj}(M)
  \end{tikzcd} \Big) \in \mathcal{M}
\end{equation*}
for $\mathbf{M} = (M, a_M : T(M) \to M) \in \lmod{\mathbf{T}}$.

The assignment $\mathbf{M} \mapsto \mathbf{T}^{\ladj} \otimes_{\mathbf{T}} \mathbf{M}$ is not an endofunctor on $\lmod{\mathbf{T}}$ since, unlike the case of \eqref{eq:intro-Nakayama-A-mod}, the functor $T^{\ladj}$ is not a `bimodule' over $T$. Such asymmetry has already been noticed in the study of algebras and modules in a monoidal category. Namely, let $A$ be an algebra in $\mathcal{C}$. Then, according to \cite[\S2.4.5]{MR4254952}, the left dual object of an $A$-bimodule in $\mathcal{C}$ is not an $A$-bimodule in $\mathcal{C}$ in general, but in fact an $A^{\vee\vee}$-$A$-bimodule in $\mathcal{C}$.
We now suppose that a double left adjoint $T^{\lladj} := (T^{\ladj})^{\ladj}$ exists. Then $T^{\lladj}$ is naturally a monad on $\mathcal{M}$, which we denote by $\mathbf{T}^{\lladj}$.
In the same way as \cite[\S2.4.5]{MR4254952}, we see that $T^{\ladj}$ is a `$\mathbf{T}^{\lladj}$-$\mathbf{T}$-bimodule.' Furthermore, the assignment $\mathbf{M} \mapsto \mathbf{T}^{\ladj} \otimes_{\mathbf{T}} \mathbf{M}$ gives rise to a linear functor
\begin{equation}
  \label{eq:intro-T-la-tensor}
  \mathbf{T}^{\ladj} \otimes_{\mathbf{T}} (-) :
  \lmod{\mathbf{T}} \to \lmod{\mathbf{T}^{\lladj}}.
\end{equation}
The Nakayama functor of $\mathcal{M}$ lifts to the functor
\begin{equation}
  \label{eq:intro-Nakayama-lifts}
  \lmod{\mathbf{T}^{\lladj}} \to \lmod{\mathbf{T}},
  \quad (M, a_M : T^{\lladj}(M) \to M) \mapsto (\Nak_{\mathcal{M}}(M), \widetilde{a}_M),
\end{equation}
where $\widetilde{a}_M$ is the composition
\begin{equation*}
  T (\Nak_{\mathcal{M}}(M))
  \xrightarrow{\quad \text{\eqref{eq:intro-cano-iso} with $G = T^{\lladj}$} \quad}
  \Nak_{\mathcal{M}} (T^{\lladj}(M))
  \xrightarrow{\quad \Nak_{\mathcal{M}}(a_M) \quad}
  \Nak_{\mathcal{M}}(M).
\end{equation*}
Now the main result of this paper is stated as follows:

\begin{theorem}[$=$ Theorem \ref{thm:Nakayama-monads}]
  \label{thm:intro-main-thm}
  Let $\mathcal{M}$ be a finite abelian category, and let $\mathbf{T}$ be a linear monad on $\mathcal{M}$ whose underlying functor admits a double left adjoint and a double right adjoint. Then the Nakayama functor of $\lmod{\mathbf{T}}$ is given by the composition of the functors \eqref{eq:intro-T-la-tensor} and \eqref{eq:intro-Nakayama-lifts}.
\end{theorem}

To define the functors \eqref{eq:intro-T-la-tensor} and \eqref{eq:intro-Nakayama-lifts}, it suffices to assume that both $T^{\lladj}$ and $T^{\radj}$ exist.
Let $\Forg : \lmod{\mathbf{T}} \to \mathcal{M}$ denote the forgetful functor.
In the proof of the above theorem, we consider the isomorphism
\begin{equation}
  \label{eq:intro-proof-main-thm}
  \Nak_{\lmod{\mathbf{T}}} \circ \Forg^{\ladj} \cong \Forg^{\radj} \circ \Nak_{\mathcal{M}}
\end{equation}
obtained by letting $G = \Forg^{\ladj}$ in \eqref{eq:intro-cano-iso}.
The existence of $T^{\rradj}$ is required to ensure that a double right adjoint of $\Forg^{\ladj}$ is right exact. The assumption of the above theorem might be weaken by detouring the use of \eqref{eq:intro-proof-main-thm}.

In the theory of finite tensor categories and their modules, there are many examples of monads to which Theorem \ref{thm:intro-main-thm} can be applied.
Let $\mathcal{C}$ be a finite tensor category, let $A$ be an algebra in $\mathcal{C}$, and let $\mathcal{M}$ be a finite left $\mathcal{C}$-module category with action $\catactl : \mathcal{C} \times \mathcal{M} \to \mathcal{M}$ (see \S\ref{subsec:FTC-modules}). Then the monad $\mathbf{T} := A \catactl (-)$ on $\mathcal{M}$ induced by the algebra $A$ fulfills the assumption of Theorem~\ref{thm:intro-main-thm}. The theorem implies that the Nakayama functor of the category of modules over this monad is written as
\begin{equation}
  \label{eq:Intro-monad-for-modules}
  \mathbf{M} \mapsto \Nak_{\mathcal{M}}(A^{\vee} \otimes_A \mathbf{M})
\end{equation}
(see \S\ref{subsec:modules-module-cats} for the precise meaning).
This formula is used to, for example, compute the Nakayama functor of the category of modules over a Hopf algebra in a braided finite tensor category (see \S\ref{subsec:braid-Hopf}). The formula is also applied to compute the Nakayama functor of the center of a finite bimodule category.

\subsection{Organization of this paper}
We explain the organization of this paper as well as results yet to be described in the above.

In Section~\ref{sec:prelim}, we introduce some notations and recall basic results on finite tensor categories and their modules.

In Section \ref{sec:Nakayama-and-double-adj}, we review necessary results on the Nakayama functor.
We explain the construction of the canonical isomorphism \eqref{eq:intro-cano-iso} and discuss how it looks like in various settings.
Most of the contents of this section are found in \cite{MR4042867}.
The aim of this section is, rather, to write down those results explicitly for the use of later sections.

In Section \ref{sec:for-monads}, we state and prove the main result of this paper. For this purpose, we first introduce the notion of bimodules over monads and the tensor product over a monad. Let $\mathcal{M}$ be a finite abelian category, and let $\mathbf{T}$ be a monad on $\mathcal{M}$ satisfying the assumption of Theorem~\ref{thm:intro-main-thm}. The forgetful functor $\Forg: \lmod{\mathbf{T}} \to \mathcal{M}$ has a left adjoint, called the free $\mathbf{T}$-module functor. An easy but important observation is that the functor $\lmod{\mathbf{T}} \to \mathcal{M}$ given by $\mathbf{M} \mapsto \mathbf{T}^{\ladj} \otimes_{\mathbf{T}} \mathbf{M}$ is a left adjoint of $\Forg^{\ladj}$ (Lemma~\ref{lem:left-adjoint-of-free}). Thus, by \eqref{eq:intro-cano-iso}, we have
\begin{equation*}
  \Forg \Nak_{\lmod{\mathbf{T}}}(\mathbf{M})
  \cong \Nak_{\mathcal{M}} \Forg^{\lladj}(\mathbf{M})
  \cong \Nak_{\mathcal{M}} (\mathbf{T}^{\ladj} \otimes_{\mathbf{T}} \mathbf{M})
\end{equation*}
for $\mathbf{M} \in \lmod{\mathbf{T}}$. This means that the underlying object of $\Nak_{\lmod{\mathbf{T}}}(\mathbf{M})$ is given as stated in Theorem \ref{thm:intro-main-thm}. To verify that $\Nak_{\lmod{\mathbf{T}}}(\mathbf{M})$ is given as stated as a $\mathbf{T}$-module, we need a bit technical discussion using \eqref{eq:intro-proof-main-thm}, which we omit in Introduction.

In Section \ref{sec:algebras-w-Fb-trace}, we give applications of our results to the category of modules over an algebra in a finite tensor category.
Let $\mathcal{C}$ be a finite multi-tensor category, and let $\mathcal{M}$ be a finite left $\mathcal{C}$-module category.
We denote by ${}_A \mathcal{M}$ the category of modules over the monad induced by an algebra $A$ in $\mathcal{C}$.
By applying the main theorem, we see that the Nakayama functor of ${}_A \mathcal{M}$ is given by \eqref{eq:Intro-monad-for-modules}.

The case where $A$ is Frobenius may be of particular interest.
In Section \ref{sec:algebras-w-Fb-trace}, we actually consider a more general setting that there is an invertible object $I \in \mathcal{C}$ and $A$ is an algebra in $\mathcal{C}$ equipped with a `non-degenerate' morphism $\lambda : A \to I$ (such a morphism $\lambda$ is called an $I$-valued Frobenius trace in this paper).
In this case, the Nakayama functor of ${}_A \mathcal{M}$ is given by $\mathbf{M} \mapsto \Nak_{\mathcal{M}}(I \catactl \mathbf{M})$ on the level of objects of $\mathcal{M}$, but the action of $A$ is twisted by an analogue of the Nakayama automorphism of $A$ as in the case of ordinary Frobenius algebras (Theorem~\ref{thm:Nakayama-for-modules-2}).
This result can be applied to Hopf algebras in a braided finite multi-tensor category (\S\ref{subsec:braid-Hopf}).
Indeed, a cointegral and the `object of integrals' play the role of $\lambda$ and $I$ of the above, respectively.

In Section \ref{sec:the-center}, we discuss the Nakayama functor of the center of a finite bimodule category. The base field is assumed to be perfect in this section.
Let $\mathcal{C}$ be a finite multi-tensor category, and let $\mathcal{C}^{\env} := \mathcal{C} \boxtimes \mathcal{C}^{\rev}$. 
Given a finite $\mathcal{C}$-bimodule category $\mathcal{M}$, the center $\mathcal{Z}(\mathcal{M})$ is defined as an analogue of the center of a bimodule over a ring.
There is an algebra $\canalg \in \mathcal{C}^{\env}$ called the canonical algebra in \cite{MR3242743}.
The algebra $\canalg$ has an $(\alpha \boxtimes \unitobj)$-valued Frobenius trace, where $\alpha = \Nak_{\mathcal{C}}(\unitobj)$.
It was shown in \cite{2017arXiv170709691S} that $\mathcal{Z}(\mathcal{M})$ is isomorphic to ${}_{\canalg}\mathcal{M}$, where $\mathcal{M}$ is regarded as a finite left $\mathcal{C}^{\env}$-module category (see \S\ref{subsec:canonical-algebra}).
By applying the arguments of Section~\ref{sec:algebras-w-Fb-trace} to ${}_{\canalg}\mathcal{M}$, we obtain an explicit formula of $\Nak_{\mathcal{Z}(\mathcal{M})}$ (Theorem~\ref{thm:Nakayama-center}).

Some applications of Theorem~\ref{thm:Nakayama-center} are exhibited in \S\ref{subsec:applications}. We introduce some of them:
Let $\mathcal{C}$ be a finite-multi tensor category.
One of important examples of the center construction is the Drinfeld center of $\mathcal{C}$. Our result implies a known fact that $\mathcal{Z}(\mathcal{C})$ is unimodular \cite{MR2097289}.
Another important example is the dual of a finite multi-tensor category.
Let $\mathcal{M}$ be a finite left $\mathcal{C}$-module category. Then the category $\mathcal{E}$ of linear right exact endofunctors on $\mathcal{M}$ is a finite $\mathcal{C}$-bimodule category in a natural way. The center of $\mathcal{E}$ is nothing but the dual of $\mathcal{C}$ with respect to $\mathcal{M}$.
Hence we reproduce a formula of the Nakayama functor of the dual finite tensor category recently obtained in \cite{2022arXiv220707031F} in a different method.

In Section \ref{sec:Hopf-examples}, we explain how our results are applied to some categories appearing in the Hopf algebra theory.
Let $H$ be a finite-dimensional Hopf algebra. In this section, we work in the category $\mathcal{C}$ of finite-dimensional left $H$-comodules.
Given two algebras $A$ and $B$ in $\mathcal{C}$, the category ${}_A \mathcal{C}_B$ of $A$-$B$-bimodules in $\mathcal{C}$ is defined (an object of this category is called a relative Hopf bimodule).
Since ${}_A\mathcal{C}_B$ is the category of modules over the monad $A \otimes (-) \otimes B$, its Nakayama functor can be computed by our result.
For practical applications, we consider algebras in $\mathcal{C}$ admitting non-degenerate grouplike-cointegral in the sense of \cite{2018arXiv181007114K}. When $A$ and $B$ are such algebras, then the Nakayama functor of ${}_A \mathcal{C}_B$ takes a simple form (Theorem~\ref{thm:relative-Hopf-bimod-Nakayama}).
By slightly extending the notion of unimodularity of a finite tensor category, we say that ${}_A\mathcal{C}_A$ is unimodular if the Nakayama functor of ${}_A \mathcal{C}_A$ fixes $A \in {}_A \mathcal{C}_A$ up to isomorphism.
Theorem~\ref{thm:relative-Hopf-bimod-Nakayama} yields handy criteria for ${}_A \mathcal{C}_A$ to be unimodular (Corollaries~\ref{cor:A-bimod-unimodular-1} and \ref{cor:A-bimod-unimodular-2}).

Some concrete examples are given in \S\ref{subsec:examples-unimodularity}.
For an algebra $A \in \mathcal{C}$, the category $\lmod{A}$ is in fact a left module category over the finite tensor category $\mathcal{D} := \lmod{H}$.
According to \cite{MR2331768}, the category ${}_A \mathcal{C}_A$ is monoidally equivalent to the dual of $\mathcal{D}$ with respect to $\mathcal{M}$. Thus our computation can also be thought of as examples of determining the unimodularity of the dual tensor category.

Our criteria for unimodularity (Corollaries~\ref{cor:A-bimod-unimodular-1} and \ref{cor:A-bimod-unimodular-2}) are applicable only for algebras in $\mathcal{C}$ admitting non-degenerate grouplike-cointegrals.
The presence of such a cointegral reduces the computation drastically, however, we shall note that it is sometimes absent. We give examples of coideal subalgebras without non-zero grouplike-cointegrals in \S\ref{subsec:no-g-cointegrals}. Such cases should also be treated, but it is beyond the scope of this paper.

\subsection{Acknowledgment}

The author thanks Taiki Shibata for discussion.
The author is supported by JSPS KAKENHI Grant Number JP20K03520.

\section{Preliminaries}
\label{sec:prelim}

\subsection{Basic notations}

Our basic reference on category theory is Mac Lane \cite{MR1712872}.
Given a category $\mathcal{C}$, we denote its opposite category by $\mathcal{C}^{\op}$.
An object $X \in \mathcal{C}$ is written as $X^{\op}$ when it is viewed as an object of $\mathcal{C}^{\op}$.
A similar notation will be used for functors.

Given a functor $F$, we denote by $F^{\ladj}$ and $F^{\radj}$ a left and a right adjoint of $F$, respectively, provided that they exist.
We say that $F$ admits a double left adjoint if a left adjoint of $F$ exists and admits a left adjoint.
The meaning of a `triple' left adjoint should be clear.
A similar phrase will be used for right adjoints.

Throughout this paper, we work over a field $\bfk$ (a technical assumption will be imposed to $\bfk$ in Section \ref{sec:the-center}).
Given a vector space $X$, we denote by $X^*$ the dual space of $X$.
By an algebra, we mean an associative and unital algebra over $\bfk$. Given two algebras $A$ and $B$, we denote by $\lmod{A}$, $\rmod{B}$ and $\bimod{A}{B}$ the category of finite-dimensional left $A$-modules, right $B$-modules and $A$-$B$-modules, respectively.
We write $\lmod{\bfk}$ as $\Vect$.

\subsection{Finite abelian categories}
\label{subsec:fin-ab-cat}

A finite abelian category is a linear category that is equivalent to $\lmod{A}$ for some finite-dimensional algebra $A$. We note that a linear category $\mathcal{M}$ is a finite abelian category if and only if $\mathcal{M}^{\op}$ is, since the opposite category of $\lmod{A}$ is equivalent to $\lmod{A^{\op}}$ by the duality.

Given two finite abelian categories $\mathcal{M}$ and $\mathcal{N}$, we denote by $\REX(\mathcal{M}, \mathcal{N})$ the category of linear right exact functors from $\mathcal{M}$ to $\mathcal{N}$.
Let $A$ and $B$ be finite-dimensional algebras.
The Eilenberg-Watts theorem states that the functor
\begin{equation}
  \label{eq:EW-equiv}
  \bimod{B}{A} \to \REX(\lmod{A}, \lmod{B}),
  \quad M \mapsto M \otimes_A(-)
\end{equation}
is an equivalence of linear categories. By this equivalence, we see that $\REX(\mathcal{M}, \mathcal{N})$ is a finite abelian category.

The equivalence \eqref{eq:EW-equiv} also shows that a linear functor $F : \mathcal{M} \to \mathcal{N}$ has a right adjoint if and only if $F$ is right exact. Applying the same argument to $F^{\op}: \mathcal{M}^{\op} \to \mathcal{N}^{\op}$, we see that $F$ has a left adjoint if and only if $F$ is left exact.

Let $\mathcal{M}$ be a finite abelian category, and let $M \in \mathcal{M}$ be an object. Since the linear functor $\Hom_{\mathcal{M}}(M, -): \mathcal{M} \to \Vect$ is left exact, it has a left adjoint.
We denote it by $(-) \copow M$.
Thus, by definition, there is a natural isomorphism
\begin{equation}
  \label{eq:def-copower}
  \Hom_{\mathcal{M}}(X \copow M, N) \cong \Hom_{\bfk}(X, \Hom_{\mathcal{M}}(M, N))
\end{equation}
for $X \in \Vect$ and $N \in \mathcal{M}$. The assignment $(X, M) \mapsto X \copow M$ extends to a bilinear functor from $\Vect \times \mathcal{M}$ to $\mathcal{M}$, which we call the {\em copower}, in such a way that \eqref{eq:def-copower} is also natural in $M$.
Using the terminology to be introduced in \S\ref{subsec:FTC-modules}, the category $\mathcal{M}$ is a finite module category over $\Vect$.

\subsection{Finite tensor categories and their modules}
\label{subsec:FTC-modules}

For basics on monoidal categories, we refer the reader to \cite{MR1712872} and \cite{MR3242743}.
We assume that all monoidal categories are strict in view of Mac Lane's strictness theorem. Given a monoidal category $\mathcal{C}$, we usually denote by $\otimes$ and $\unitobj$ the monoidal product and the unit object of $\mathcal{C}$, respectively.
We denote by $\mathcal{C}^{\rev}$ the monoidal category obtained from $\mathcal{C}$ by reversing the order of the monoidal product.

We follow \cite[\S2.10]{MR3242743} for terminology for dual objects in a monoidal category.
Let $\mathcal{C}$ be a rigid monoidal category, that is, a monoidal category of which every object has a left and a right dual object.
We usually denote a left dual object of $X \in \mathcal{C}$ by
$(X^{\vee}, \eval_X : X^{\vee} \otimes X \to \unitobj, \coev_X : \unitobj \to X \otimes X^{\vee})$.
The assignment $X \mapsto X^{\vee}$ gives rise to a contravariant monoidal equivalence $\mathcal{C} \to \mathcal{C}^{\rev}$, which we call the left duality functor.
A quasi-inverse of $(-)^{\vee}$, which we denote by ${}^{\vee}(-)$, is given by taking a right dual object.
By replacing $\mathcal{C}$ with an equivalent one and choosing dual objects in an appropriate way, we may assume that $(-)^{\vee}$ and ${}^{\vee}(-)$ are mutually inverse strict monoidal functors.

\begin{definition}[\cite{MR3242743}]
  A {\em finite multi-tensor category} is a finite abelian category equipped with a structure of a rigid monoidal category whose monoidal product is bilinear.
  A {\em finite tensor category} is a finite multi-tensor category whose unit object is simple.
\end{definition}

Given a monoidal category $\mathcal{C}$, a {\em left $\mathcal{C}$-module category} \cite{MR3242743} is a category $\mathcal{M}$ endowed with a functor $\catactl: \mathcal{C} \times \mathcal{M} \to \mathcal{M}$, called the {\em action} of $\mathcal{C}$ on $\mathcal{M}$, and natural isomorphisms
\begin{equation}
  \label{eq:mod-cat-assoc}
  (X \otimes Y) \catactl M \cong X \catactl (Y \catactl M)
  \quad \text{and} \quad
  \unitobj \catactl M \cong M
  \quad (X, Y \in \mathcal{C}, M \in \mathcal{M})
\end{equation}
satisfying certain axioms similar to those for monoidal categories.
A right module category and a bimodule category are defined analogously.

Let $\mathcal{M}$ and $\mathcal{N}$ be left $\mathcal{C}$-module categories.
A {\em lax left $\mathcal{C}$-module functor} \cite{MR3934626} from $\mathcal{M}$ to $\mathcal{N}$ is a pair $(F, \xi)$ consisting of a functor $F: \mathcal{M} \to \mathcal{N}$ and a natural transformation $\xi_{X,M} : X \catactl F(M) \to F(X \catactl M)$ ($X \in \mathcal{C}$, $M \in \mathcal{M}$) obeying certain axioms similar to those for monoidal functors.
A lax left $\mathcal{C}$-module functor $(F, \xi)$ is said to be {\em strong} if the structure morphism $s$ is invertible.
When $\mathcal{C}$ is rigid, every lax left $\mathcal{C}$-module functor is strong \cite[Lemma 2.10]{MR3934626} and thus the adjective `lax' is usually omitted.

Left $\mathcal{C}$-module categories, lax left $\mathcal{C}$-module functors and their morphisms form a 2-category.
An equivalence of left $\mathcal{C}$-module categories is defined as an equivalence in this 2-category.
A left $\mathcal{C}$-module category is said to be strict if the natural isomorphisms \eqref{eq:mod-cat-assoc} are the identities.
One can prove an analogue of Mac Lane's strictness theorem for module categories \cite[Remark 7.2.4]{MR3242743}. Hence we may, and do, assume that all module categories are strict in this paper.

\begin{definition}
  Let $\mathcal{C}$ be a finite multi-tensor category.
  A {\em finite left $\mathcal{C}$-module category} is a finite abelian category $\mathcal{M}$ equipped with a structure of a left $\mathcal{C}$-module category such that the action of $\mathcal{C}$ on $\mathcal{M}$ is linear and right exact in each variables. A finite right $\mathcal{C}$-module category and a finite $\mathcal{C}$-bimodule category are defined in an analogous way.
\end{definition}

Despite that we only assume its right exactness, the action of $\mathcal{C}$ on a finite left $\mathcal{C}$-module category is {\em exact} in each variable \cite[Corollary 2.26]{MR3934626}.

\subsection{Nakayama functor}
\label{subsec:Nakayama}

Let $\mathcal{M}$ be a finite abelian category.
In \cite{MR4042867}, the Nakayama functor of $\mathcal{M}$ is defined to be the endofunctor $\Nak_{\mathcal{M}}$ on $\mathcal{M}$ given by
\begin{equation*}
  \Nak_{\mathcal{M}}(M) = \int^{X \in \mathcal{M}} \Hom_{\mathcal{M}}(M, X)^* \copow X
\end{equation*}
for $M \in \mathcal{M}$, where the integral means coends \cite{MR1712872}.
When the category $\mathcal{M}$ is clear from the context, we write $\Nak_{\mathcal{M}}$ as $\Nak$.

A finite abelian category $\mathcal{M}$ is said to be {\em Frobenius} if the class of projective objects of $\mathcal{M}$ coincides with the class of injective objects of $\mathcal{M}$.
We say that $\mathcal{M}$ is {\em symmetric Frobenius} if $\mathcal{M} \approx \lmod{A}$ for some symmetric Frobenius algebra $A$.
According to \cite[Proposition 3.24]{MR4042867}, $\mathcal{M}$ is Frobenius ({\em respectively}, symmetric Frobenius) if and only if the functor $\Nak_{\mathcal{M}}$ is an equivalence ({\it respectively}, $\Nak_{\mathcal{M}}$ is isomorphic to $\id_{\mathcal{M}}$).

One of key features of the Nakayama functor is the following canonical isomorphism:
Let $G: \mathcal{M} \to \mathcal{N}$ be a linear functor between finite abelian categories.
If $G$ admits a triple right adjoint, then there is a natural isomorphism
\begin{equation}
  \label{eq:Nakayama-cano-iso}
  \can_G(M): G^{\rradj}\Nak(M) \to \Nak G(M)
  \quad (M \in \mathcal{M})
\end{equation}
that is `natural' in $G$ and `coherent' in a certain sense \cite[Theorem 3.18]{MR4042867}. The construction of~\eqref{eq:Nakayama-cano-iso} will be recalled in Section \ref{sec:Nakayama-and-double-adj}.

\subsection{Radford isomorphism}
\label{subsec:radford-iso}

Let $\mathcal{C}$ be a finite multi-tensor category, and let $\mathcal{M}$ be a finite left $\mathcal{C}$-module category with action $\catactl : \mathcal{C} \times \mathcal{M} \to \mathcal{M}$.
Then the functor
\begin{equation*}
  \act_{\mathcal{M}}: \mathcal{C} \to \REX(\mathcal{M}) := \REX(\mathcal{M}, \mathcal{M}),
  \quad \act_{\mathcal{M}}(X)(M) = X \catactl M
\end{equation*}
is a linear strong monoidal functor, where $\REX(\mathcal{M})$ is viewed as a monoidal category by the composition of functors. Since a strong monoidal functor preserves duals, there is an adjunction $\act_{\mathcal{M}}(X) \dashv \act_{\mathcal{M}}({}^{\vee\!}X)$ for each object $X \in \mathcal{C}$.
The functor $\Nak_{\mathcal{M}}$ is a `twisted' left $\mathcal{C}$-module functor by the natural isomorphism
\begin{equation}
  \label{eq:cat-action-left-Nakayama}
  {}^{\vee\vee\!}X \catactl \Nak_{\mathcal{M}}(M)
  \cong \Nak_{\mathcal{M}}(X \catactl M)
  \quad (X \in \mathcal{C}, M \in \mathcal{M})
\end{equation}
obtained by letting $G = \act_{\mathcal{M}}(X)$ in \eqref{eq:Nakayama-cano-iso}.

An analogous result holds for finite right module categories: Let $\mathcal{C}$ be as above, and let $\mathcal{M}$ be a finite right $\mathcal{C}$-module category with action $\catactr: \mathcal{M} \times \mathcal{C} \to \mathcal{M}$. Then $\Nak_{\mathcal{M}}$ has a twisted right $\mathcal{C}$-module structure
\begin{equation}
  \label{eq:cat-action-right-Nakayama}
  \Nak_{\mathcal{M}}(M) \catactr X^{\vee\vee}
  \cong \Nak_{\mathcal{M}}(M \catactr X)
  \quad (X \in \mathcal{C}, M \in \mathcal{M}).
\end{equation}

The category $\mathcal{C}$ itself is a finite $\mathcal{C}$-bimodule category by the tensor product of $\mathcal{C}$. Hence, as noted in \cite{MR4042867}, there are natural isomorphisms
\begin{equation}
  \label{eq:FTC-Nakayama-formula}
  {}^{\vee\vee\!}X \otimes \Nak_{\mathcal{C}}(\unitobj)
  \mathop{\cong}^{\eqref{eq:cat-action-left-Nakayama}}
  \Nak_{\mathcal{C}}(X \otimes \unitobj)
  = \Nak_{\mathcal{C}}(X)
  = \Nak_{\mathcal{C}}(\unitobj \otimes X)
  \mathop{\cong}^{\eqref{eq:cat-action-right-Nakayama}}
  \Nak_{\mathcal{C}}(\unitobj) \otimes X^{\vee\vee}
\end{equation}
for $X \in \mathcal{C}$.
The object $\Nak_{\mathcal{C}}(\unitobj)$ is of particular importance.
If $\mathcal{C} = \lmod{H}$ for some finite-dimensional (quasi-)Hopf algebra $H$, then $\Nak_{\mathcal{C}}(\unitobj)$ is given by the modular function on $H$. Accordingly, we introduce the following terminology:

\begin{definition}
  Let $\mathcal{C}$ be a finite abelian category equipped with a structure of a monoidal category. We call $\alpha_{\mathcal{C}} := \Nak_{\mathcal{C}}(\unitobj)$ the {\em modular object} of $\mathcal{C}$. We say that $\mathcal{C}$ is {\em unimodular} if $\Nak_{\mathcal{C}}(\unitobj)$ is isomorphic to the unit object.
\end{definition}

Radford's $S^4$-formula for finite-dimensional Hopf algebras has been extended to finite multi-tensor category as a formula of the quadruple dual \cite{MR2097289}.
As pointed out in \cite{MR4042867}, the $S^4$-formula of \cite{MR2097289} follows from a basic property of the Nakayama functor. To give a detail, we first introduce:

\begin{definition}
  For a finite multi-tensor category $\mathcal{C}$, we define
  \begin{equation*}
    \Radford_X := \left(
      \alpha_{\mathcal{C}} \otimes X^{\vee\vee}
      \xrightarrow[\cong]{\quad \eqref{eq:cat-action-right-Nakayama} \quad}
      \Nak_{\mathcal{C}}(X)
      \xrightarrow[\cong]{\quad \eqref{eq:cat-action-left-Nakayama} \quad}
      {}^{\vee\vee\!}X \otimes \alpha_{\mathcal{C}}
    \right)
    \quad (X \in \mathcal{C})
  \end{equation*}
  and call $\Radford$ the {\em Radford isomorphism} of $\mathcal{C}$.
\end{definition}

Since a finite multi-tensor category is Frobenius \cite{MR3242743}, the functor $\Nak_{\mathcal{C}}$ is an equivalence. By the formula \eqref{eq:FTC-Nakayama-formula} of the Nakayama functor, the modular object $\alpha_{\mathcal{C}}$ is invertible. Hence the Radford isomorphism induces a natural isomorphism
\begin{equation*}
  X^{\vee\vee\vee\vee}
  \xrightarrow{\quad \eval^{-1} \otimes \id \quad}
  \alpha_{\mathcal{C}}^{\vee} \otimes \alpha_{\mathcal{C}} \otimes X^{\vee\vee\vee\vee}
  \xrightarrow{\quad \id \otimes \Radford_{X^{\vee\vee}}^{} \quad}
  \alpha_{\mathcal{C}}^{\vee} \otimes X \otimes \alpha_{\mathcal{C}}
\end{equation*}
for $X \in \mathcal{C}$. As has been explained in \cite[Section 4]{2017arXiv170709691S}, this isomorphism coincides with the Radford $S^4$-formula given in \cite{MR2097289}.

\begin{remark}
  \label{rem:Nakayama-twisted-module-structure}
  Let $\mathcal{C}$ be a finite multi-tensor category.
  In view of \eqref{eq:FTC-Nakayama-formula}, we may assume that $\Nak_{\mathcal{C}}$ is given by $\Nak_{\mathcal{C}}(X) = {}^{\vee\vee\!}X \otimes \alpha_{\mathcal{C}}$ for $X \in \mathcal{C}$. Under this choice of the Nakayama functor (and our assumption that the double dual functor of $\mathcal{C}$ is strict monoidal), the twisted left $\mathcal{C}$-module structure \eqref{eq:cat-action-left-Nakayama} of $\Nak_{\mathcal{C}}$ is the identity morphism. The twisted right $\mathcal{C}$-module structure \eqref{eq:cat-action-right-Nakayama} is given by
  \begin{equation*}
    \id_{{}^{\vee\vee\!}M} \otimes \Radford_X:
    \Nak_{\mathcal{C}}(M) \otimes X^{\vee\vee}
    \to \Nak_{\mathcal{C}}(M \otimes X)
  \end{equation*}
  for $M, X \in \mathcal{C}$.
\end{remark}

\section{Nakayama functors and double adjoints}
\label{sec:Nakayama-and-double-adj}

\subsection{Construction of the canonical isomorphism}

In this section, we collect useful formulas of Nakayama functors from \cite{MR4042867} and discuss how the canonical isomorphism \eqref{eq:Nakayama-cano-iso} looks like in various settings.
We first recall the following lemma for (co)ends:

\begin{lemma}
  \label{lem:adjoints-ends}
  Let $T: \mathcal{B}^{\op} \times \mathcal{A} \to \mathcal{V}$ and $F: \mathcal{A} \to \mathcal{B}$ be functors, where $\mathcal{A}$, $\mathcal{B}$ and $\mathcal{V}$ are categories.
  If $F$ has a left adjoint, then we have
  \begin{equation}
    \label{eq:adjoint-end-iso}
    \int_{X \in \mathcal{A}} T(F(X), X)
    \cong \int_{Y \in \mathcal{B}} T(Y, F^{\ladj}(Y)),
  \end{equation}
  meaning that the end of the left hand side exists if and only if so does the right hand side and, if they exist, they are canonically isomorphic.
  If $F$ has a right adjoint, then we have
  \begin{equation}
    \label{eq:adjoint-coend-iso}
    \int^{X \in \mathcal{A}} T(F(X), X)
    \cong \int^{Y \in \mathcal{B}} T(Y, F^{\radj}(Y))
  \end{equation}
  with a similar meaning as the case of ends.
\end{lemma}

The isomorphism \eqref{eq:adjoint-coend-iso} is given in \cite[Lemma 3.8]{MR2869176}, and the isomorphism \eqref{eq:adjoint-end-iso} is obtained by the dual argument.
For reader's convenience, we include how the isomorphism \eqref{eq:adjoint-coend-iso} is established. Let $C$ and $D$ be the left and the right hand side of \eqref{eq:adjoint-coend-iso}, respectively, and let
\begin{equation*}
  i_X : T(F(X),X) \to C
  \quad (X \in \mathcal{A})
  \quad \text{and} \quad
  j_Y : T(Y, F^{\radj}(Y)) \to D
  \quad (Y \in \mathcal{B})
\end{equation*}
be the universal dinatural transformations for these coends.
Then the isomorphism of the above lemma, which we denote by $\phi : C \to D$, and its inverse are characterized as unique morphisms in $\mathcal{V}$ such that the equations
\begin{equation*}
  \phi \circ i_X = j_{F(X)} \circ T(\id_{F(X)}, \eta_{X})
  \quad \text{and} \quad
  \phi^{-1} \circ j_{Y} = i_{F^{\radj}(Y)} \circ T(\varepsilon_Y, \id_{F^{\radj}(Y)})
\end{equation*}
hold for all objects $X \in \mathcal{A}$ and $Y \in \mathcal{B}$, where $\eta$ and $\varepsilon$ are the unit and the counit of the adjunction $F \dashv F^{\radj}$.

Let $\mathcal{M}$ and $\mathcal{N}$ be finite abelian categories, and let $G: \mathcal{M} \to \mathcal{N}$ be a linear functor admitting a triple right adjoint. The canonical isomorphism
\begin{equation*}
  \tag{\ref{eq:Nakayama-cano-iso}}
  \can_G(M) : G^{\rradj} \Nak(M) \to \Nak G(M)
  \quad (M \in \mathcal{M})
\end{equation*}
mentioned in \S\ref{subsec:Nakayama} is given by the composition
\begin{align*}
  G^{\rradj} \Nak(M)
  & \textstyle = G^{\rradj}(\int^{X \in \mathcal{M}} \Hom_{\mathcal{M}}(M, X)^* \copow X) \\
  & \textstyle \cong \int^{X \in \mathcal{M}} \Hom_{\mathcal{M}}(M, X)^* \copow G^{\rradj}(X) \\
  & \textstyle \cong \int^{X \in \mathcal{N}} \Hom_{\mathcal{M}}(M, G^{\radj}(X))^* \copow X \\
  & \textstyle \cong \int^{X \in \mathcal{N}} \Hom_{\mathcal{M}}(G(M), X)^* \copow X
    = \Nak G(M),
\end{align*}
where the first isomorphism follows from that $G^{\rradj}$ preserves copowers and colimits as it has a right adjoint, the second one is given by Lemma~\ref{lem:adjoints-ends} and the last one is the adjunction isomorphism.

There are natural isomorphisms
\begin{equation}
  \label{eq:Nakayama-Hom-1}
  \begin{aligned}
    \Hom_{\mathcal{M}}(\Nak(M), M')
    & \textstyle \cong \int_{X \in \mathcal{M}} \Hom_{\mathcal{M}}(\Hom_{\mathcal{M}}(X, M')^* \copow M, X) \\
    & \textstyle \cong \int_{X \in \mathcal{M}} \Hom_{\bfk}(\Hom_{\mathcal{M}}(X, M')^*, \Hom_{\mathcal{M}}(M, X)) \\
    & \textstyle \cong \int_{X \in \mathcal{M}} \Hom_{\mathcal{M}}(X, M') \otimes_{\bfk} \Hom_{\mathcal{M}}(M, X) \\
  \end{aligned}
\end{equation}
for $M, M' \in \mathcal{M}$.
By the construction of the isomorphism \eqref{eq:Nakayama-cano-iso}, it is straightforward to verify the following lemma:

\begin{lemma}
  \label{lem:Nakayama-adj-diagram}
  For $G$ as above, the diagram of Figure \ref{fig:Nakayama-diagram-1} commutes for all objects $M \in \mathcal{M}$ and $N \in \mathcal{N}$ (the Hom functor is denoted by $[-,-]$ to save spaces in the diagram).
\end{lemma}

\begin{figure}
  \centering
  \makebox[\textwidth][c]{
    \begin{tikzcd}[ampersand replacement=\&, row sep = 5pt, column sep = 48pt]
      [\Nak G(M), N]
      \arrow[r, "\text{\eqref{eq:Nakayama-Hom-1}}"]
      \arrow[dd, "{[\can_G(M), N]}"']
      \& \textstyle \int_{X \in \mathcal{N}} [G(M), X] \otimes_{\bfk} [X, N]
      \arrow[d, "\text{adjunction}"] \\[15pt]
      \& \textstyle \int_{X \in \mathcal{N}} [M, G^{\radj}(X)]
      \otimes_{\bfk} [X, N]
      \arrow[dd, "\text{Lemma~\ref{lem:adjoints-ends}}"] \\[0pt]
      [G^{\rradj}\Nak(M), N]
      \arrow[dd, "\text{adjunction}"'] \\[0pt]
      \& \textstyle \int_{X \in \mathcal{M}} [M, X] \otimes_{\bfk} [G^{\rradj}(X), N]
      \arrow[d, "\text{adjunction}"] \\[15pt]
      [\Nak(M), G^{\rrradj}(N)]
      \arrow[r, "\text{\eqref{eq:Nakayama-Hom-1}}"]
      \& \textstyle \int_{X \in \mathcal{M}} [M, X] \otimes_{\bfk} [X, G^{\rrradj}(N)]
    \end{tikzcd}}
  \caption{}
  \label{fig:Nakayama-diagram-1}
\end{figure}

\subsection{Nakayama functor of the opposite category}
\label{subsec:Nakayama-opposite}

We discuss the Nakayama functor of the opposite category of a finite abelian category.
Let $\mathcal{M}$ be a finite abelian category. There is an endofunctor
\begin{equation*}
  \overline{\Nak}_{\mathcal{M}}: \mathcal{M} \to \mathcal{M},
  \quad \overline{\Nak}_{\mathcal{M}}(M) = \int_{X \in \mathcal{M}} \Hom_{\mathcal{M}}(X, M) \copow X
\end{equation*}
called the left exact analogue of the Nakayama functor in \cite[Definition 3.14]{MR4042867}.
The functor $\overline{\Nak}_{\mathcal{M}}$ is nothing but the Nakayama functor of $\mathcal{M}^{\op}$ viewed as an endofunctor on $\mathcal{M}$. Namely, we have
\begin{equation}
  \label{eq:Nakayama-oppo}
  \Nak_{\mathcal{M}^{\op}} = (\overline{\Nak}_{\mathcal{M}})^{\op},
\end{equation}
as noted in \cite[Proposition 3.20]{MR4042867}.

Since the functor $\Hom_{\mathcal{M}}(M, -) : \mathcal{M} \to \Vect$ preserves ends and copowers, there is a natural isomorphism
\begin{equation}
  \label{eq:Nakayama-Hom-2}
  \Hom_{\mathcal{M}}(M, \overline{\Nak}_{\mathcal{M}}(M')
  \cong \int_{X \in \mathcal{M}} \Hom_{\mathcal{M}}(M, X) \otimes_{\bfk} \Hom_{\mathcal{M}}(X, M')
\end{equation}
for $M, M' \in \mathcal{M}$. Thus we obtain a natural isomorphism
\begin{equation}
  \label{eq:Nakayama-Hom-3}
  \Hom_{\mathcal{M}}(\Nak_{\mathcal{M}}(M), M')
  \cong
  \Hom_{\mathcal{M}}(M, \overline{\Nak}_{\mathcal{M}}(M'))
  \quad (M, M' \in \mathcal{M})
\end{equation}
by composing \eqref{eq:Nakayama-Hom-1} and \eqref{eq:Nakayama-Hom-2}.
This means that $\overline{\Nak}_{\mathcal{M}}$ is right adjoint to $\Nak_{\mathcal{M}}$ \cite[Lemma 3.16]{MR4042867}.
In what follows, we choose $\overline{\Nak}_{\mathcal{M}}$ as a right adjoint of $\Nak_{\mathcal{M}}$ together with the adjunction isomorphism \eqref{eq:Nakayama-Hom-3}.

Now let $F: \mathcal{N} \to \mathcal{M}$ be a linear functor between finite abelian categories $\mathcal{M}$ and $\mathcal{N}$. We assume that $F$ admits a triple left adjoint so that $F^{\op} : \mathcal{N}^{\op} \to \mathcal{M}^{\op}$ admits a triple right adjoint. Then there is a canonical isomorphism
\begin{equation*}
  \can_{F^{\op}}:
  (F^{\op})^{\rradj} \circ \Nak_{\mathcal{N}^{\op}} \to
  \Nak_{\mathcal{M}^{\op}} \circ F^{\op}
\end{equation*}
by \eqref{eq:Nakayama-cano-iso} with $G = F^{\op}$.
We note that $\can_{F^{\op}}$ is an isomorphism of functors from $\mathcal{N}^{\op}$ to $\mathcal{M}^{\op}$. Thus $\can_{F^{\op}}$ is actually a family
\begin{equation*}
  \can_{F^{\op}} = \{ \can_{F^{\op}}(N) : \Nak^{\radj} F(N) \to F^{\lladj}\Nak^{\radj}(N) \}_{N \in \mathcal{N}}
\end{equation*}
of isomorphisms in $\mathcal{M}$.

\begin{lemma}
  \label{lem:Nakayama-iso-opposite}
  Let $F$ be as above, and let $G = F^{\llladj}$ be a triple left adjoint of $F$.
  Then, as a natural transformation between functors from $\mathcal{N}$ to $\mathcal{M}$, the canonical isomorphism $\can_{F^{\op}}$ is given by the composition
  \begin{equation*}
    \can_{F^{\op}} = \left(
      \Nak^{\radj} F
      \xrightarrow{\quad \cong \quad} (G^{\rradj} \Nak)^{\radj}
      \xrightarrow{\quad (\can_{G}^{-1})^{\radj} \quad} (\Nak G)^{\radj}
      \xrightarrow{\quad \cong \quad} F^{\lladj} \Nak^{\radj}
    \right),
  \end{equation*}
  where the first and the third arrows represent the canonical isomorphism $(S \circ T)^{\radj} \cong T^{\radj} \circ S^{\radj}$ for composable functors $S$ and $T$ admitting right adjoints.
\end{lemma}
\begin{proof}
  We fix objects $M \in \mathcal{M}$ and $N \in \mathcal{N}$, and consider the diagram of Figure \ref{fig:Nakayama-diagram-2} (where the Hom functor is denoted by $[-,-]$ to save spaces).
  The left rectangle is the commutative diagram of Figure~\ref{fig:Nakayama-diagram-1} with $G = F^{\llladj}$.
  The right rectangle is also commutative by Lemma \ref{lem:Nakayama-adj-diagram} applied to $F^{\op}$.
  Hence the diagram of Figure \ref{fig:Nakayama-diagram-2} is commutative.
  It shrinks to the following commutative diagram:
  \begin{equation*}
    \begin{tikzcd}[ampersand replacement=\&, row sep = 20pt, column sep = 48pt]
      \Hom_{\mathcal{N}}(\Nak G(M), N)
      \ar[d, "{\Hom_{\mathcal{N}}(\can_{G}(M), N)}"']
      \ar[r, "\eqref{eq:Nakayama-Hom-3}"]
      \& \Hom_{\mathcal{N}}(G(M), \Nak^{\radj} (N))
      \ar[d, "\text{adjunction}"] \\
      \Hom_{\mathcal{N}}(F^{\ladj}\Nak (M), N)
      \ar[d, "\text{adjunction}"']
      \& \Hom_{\mathcal{M}}(M, F^{\lladj}\Nak^{\radj}(N))
      \ar[d, "{\Hom_{\mathcal{N}}(M, \can_{F^{\op}}(N)^{-1})}"] \\
      \Hom_{\mathcal{M}}(M, \Nak^{\radj} F(N))
      \ar[r, "\eqref{eq:Nakayama-Hom-3}"]
      \& \Hom_{\mathcal{M}}(M, \Nak^{\radj} F(N))
    \end{tikzcd}
  \end{equation*}
  By letting $M = F^{\lladj}\Nak^{\radj}(N)$ and chasing $\id_{M} \in \Hom_{\mathcal{M}}(M, F^{\lladj}\Nak^{\radj}(N))$ in this diagram, we see that the inverse of $\can_{F^{\op}}$ is given by
  \begin{equation*}
    \can_{F^{\op}}^{-1} = \left(
      F^{\lladj} \Nak^{\radj}
      \xrightarrow{\quad \cong \quad} (\Nak G)^{\radj}
      \xrightarrow{\quad (\can_{G})^{\radj} \quad} (F^{\ladj} \Nak)^{\radj}
      \xrightarrow{\quad \cong \quad} \Nak^{\radj} F
    \right).
  \end{equation*}
  Now the desired formula is obvious. The proof is done.
\end{proof}

\begin{figure}
  \centering
  \begin{tikzcd}[ampersand replacement=\&, row sep = 10pt, column sep = 32pt]
    \bullet \arrow[r, "\text{\eqref{eq:Nakayama-Hom-1}}"]
    \arrow[dd]
    \& \int_{X \in \mathcal{M}} [F^{\llladj}(M),X]
    \otimes_{\bfk} [X, N]
    \arrow[r, "\text{\eqref{eq:Nakayama-Hom-2}}"]
    \arrow[d, "\text{adjunction}"]
    \& {}[F^{\llladj}(M), \Nak^{\radj}(N)]
    \arrow[dd, "\text{adjunction}"]
    \\[15pt]
    \& \int_{X \in \mathcal{N}} [M, F^{\lladj}(X)]
    \otimes_{\bfk} [X, N]
    \arrow[dd, "\text{Lemma~\ref{lem:adjoints-ends}}"]
    \\[0pt]
    \bullet \arrow[dd]
    \arrow[r, phantom, "\scriptsize \text{(Figure \ref{fig:Nakayama-diagram-1} with $G = F^{\llladj}$)}"]
    \& \mbox{} \& {} [M, F^{\lladj}\Nak^{\radj}(N)]
    \arrow[dd, "{[M, \can_{F^{\op}}(N)^{-1})]}"]
    \\[0pt]
    \& \int_{X \in \mathcal{M}} [M,X] \otimes_{\bfk} [F^{\ladj}(X), N]
    \arrow[d, "\text{adjunction}"]
    \\[15pt]
    \bullet \arrow[r, "\text{\eqref{eq:Nakayama-Hom-1}}"]
    \& \int_{X \in \mathcal{M}} [M,X] \otimes_{\bfk} [X, F(N)]
    \arrow[r, "\text{\eqref{eq:Nakayama-Hom-2}}"]
    \& {} [M, \Nak^{\radj} F(N)]
  \end{tikzcd}
  \caption{}
  \label{fig:Nakayama-diagram-2}
\end{figure}

\subsection{Nakayama functor of the Deligne tensor product}

Here we give technical remarks on the Nakayama functor of the Deligne tensor product.
Given finite abelian categories $\mathcal{M}$ and $\mathcal{N}$, we denote by $\mathcal{M} \boxtimes \mathcal{N}$ their Deligne tensor product \cite{MR3242743}. 
It was pointed out in \cite[Subsection 3.4]{MR4042867} that a coend over $\mathcal{M} \boxtimes \mathcal{N}$ can be computed as a double coend $\int^{X \in \mathcal{M}, Y \in \mathcal{N}}$ under suitable assumptions on the integrand. As a consequence, we see that there is an isomorphism
\begin{equation}
  \label{eq:Nakayama-Deligne-tensor}
  \theta: \Nak_{\mathcal{M}} \boxtimes \Nak_{\mathcal{N}}
  \to \Nak_{\mathcal{M} \boxtimes \mathcal{N}}
\end{equation}
given as follows: Given a finite abelian category $\mathcal{A}$, we denote by
\begin{equation*}
  i_{M,X} : \Hom_{\mathcal{A}}(M, X)^* \copow X \to \Nak_{\mathcal{A}}(M)
  \quad (M, X \in \mathcal{A})
\end{equation*}
the universal dinatural transformation for the coend $\Nak_{\mathcal{A}}(M)$.
We fix objects $M \in \mathcal{M}$ and $N \in \mathcal{N}$.
By the Fubini theorem for coends and the exactness of $\boxtimes$, the object $\Nak_{\mathcal{M}}(M) \boxtimes \Nak_{\mathcal{N}}(N)$ is a coend
\begin{equation*}
  \int^{X \in \mathcal{M}, Y \in \mathcal{N}}
  (\Hom_{\mathcal{M}}(X,M)^* \copow X)
  \boxtimes (\Hom_{\mathcal{M}}(Y,N)^* \copow Y)
\end{equation*}
with universal dinatural transformation $i_{M,X} \boxtimes i_{N,Y}$.
For a $\boxtimes$-decomposable object $T = M \boxtimes N$, the isomorphism $\theta_T$ is defined to be the unique morphism in $\mathcal{M} \boxtimes \mathcal{N}$ such that the diagram
\begin{equation*}
  \begin{tikzcd}[column sep = 48pt, row sep = 24pt]
    (\Hom_{\mathcal{M}}(M,X)^* \copow X)
    \boxtimes (\Hom_{\mathcal{M}}(N,Y)^* \copow Y)
    \arrow[r, "{i_{X,M} \boxtimes i_{Y,N}}"]
    \arrow[d, "{\cong}"']
    & \Nak_{\mathcal{M}}(M) \boxtimes \Nak_{\mathcal{N}}(N)
    \arrow[d, "\theta_T"] \\
    \Hom_{\mathcal{M}}(M \boxtimes N, X \boxtimes Y)^* \copow (X \boxtimes Y)
    \arrow[r, "{i_{X \boxtimes Y, M \boxtimes N}}"]
    & \Nak_{\mathcal{M} \boxtimes \mathcal{N}}(T)
  \end{tikzcd}
\end{equation*}
commutes for all $X \in \mathcal{M}$ and $Y \in \mathcal{N}$.
We extend $\theta_T$ to all objects $T \in \mathcal{M} \boxtimes \mathcal{N}$ by the right exactness of Nakayama functors and a resolution of $T$ by $\boxtimes$-decomposable objects.

Let $F: \mathcal{M}_1 \to \mathcal{M}_2$ and $G: \mathcal{N}_1 \to \mathcal{N}_2$ be linear functors, where $\mathcal{M}_1$, $\mathcal{M}_2$, $\mathcal{N}_1$ and $\mathcal{N}_2$ are finite abelian categories.
We assume that both $F$ and $G$ admit triple right adjoint and choose $F^{\rradj} \boxtimes G^{\rradj}$ as a double right adjoint of $F \boxtimes G$.
In view of the above discussion, we may also choose $\Nak_{\mathcal{M}_i} \boxtimes \Nak_{\mathcal{N}_i}$ ($i = 1, 2$) as the Nakayama functor of $\mathcal{M}_i \boxtimes \mathcal{N}_i$. Then, as one may expect, the canonical isomorphism \eqref{eq:Nakayama-cano-iso} for $F \boxtimes G$ is decomposed as follows:

\begin{lemma}
  \label{lem:Nakayama-iso-Deligne-tensor}
  For $F$ and $G$ as above, we have
  \begin{equation}
    \label{eq:Nakayama-iso-Deligne-tensor-0}
    \can_{F \boxtimes G}
    = \can_F \boxtimes \can_G :
    F^{\rradj} \Nak \boxtimes G^{\rradj} \Nak
    \to \Nak F \boxtimes \Nak G.
  \end{equation}
\end{lemma}
\begin{proof}
  Since the source and the target of \eqref{eq:Nakayama-iso-Deligne-tensor-0} are linear right exact functors, it suffices to show that the equation
  \begin{equation}
    \label{eq:Nakayama-iso-Deligne-tensor-2}
    \can_{F \boxtimes G}(M \boxtimes N) = \can_F(M) \boxtimes \can_G(N)
  \end{equation}
  holds for all objects $M \in \mathcal{M}_1$ and $N \in \mathcal{N}_1$.
  For simplicity of notation, we set $\mathcal{L}_i = \mathcal{M}_i \boxtimes N_i$ ($i = 1, 2$) and introduce the functor
  \begin{equation*}
    T: \mathcal{L}_1^{\op} \times \mathcal{L}_1 \to \mathcal{L}_2,
    \quad (K^{\op}, L) \mapsto (F^{\rradj} \boxtimes G^{\rradj})(\Hom_{\mathcal{L}_1}(M \boxtimes N, K)^* \copow L).
  \end{equation*}
  The source of \eqref{eq:Nakayama-iso-Deligne-tensor-2} is a coend of $T$. We denote by $j_L$ ($L \in \mathcal{L}_1$) the universal dinatural transformation of this coend. For a $\boxtimes$-decomposable object $L = X \boxtimes Y$ of $\mathcal{L}_1$, the morphism $j_{L}$ is given by the composition of the isomorphism
  \begin{equation*}
    T(L,L) \cong F^{\rradj}(\Hom_{\mathcal{M}_1}(M, X)^* \copow X)
    \boxtimes G^{\rradj}(\Hom_{\mathcal{N}_1}(N, Y)^* \copow Y)
  \end{equation*}
  and $F^{\rradj}(i_{M,X}) \boxtimes G^{\rradj}(i_{N,Y})$. By this observation and the definition of the canonical isomorphism \eqref{eq:Nakayama-cano-iso}, it is straightforward to verify that the equation
  \begin{equation}
    \label{eq:Nakayama-iso-Deligne-tensor-proof-1}
    \can_{F \boxtimes G}(M \boxtimes N) \circ j_{L} = (\can_F(M) \boxtimes \can_G(N)) \circ j_{L}
  \end{equation}
  holds for all $\boxtimes$-decomposable object $L \in \mathcal{L}_1$.

  To complete the proof, we shall show that the equation \eqref{eq:Nakayama-iso-Deligne-tensor-proof-1} holds for all objects $L \in \mathcal{L}_1$. Let $L \in \mathcal{L}_1$ be an arbitrary object, and let $q: X \boxtimes Y \to L$ be an epimorphism in $\mathcal{L}_1$ from a $\boxtimes$-decomposable object. Then we have
  \begin{gather*}
    (\can_F(M) \boxtimes \can_G(N)) \circ j_{L} \circ T(\id_L, q)
    = (\can_F(M) \boxtimes \can_G(N)) \circ j_{X \boxtimes Y} \circ T(q, \id_{X \boxtimes Y}) \\
    = \can_{F \boxtimes G}(M \boxtimes N) \circ j_{X \boxtimes Y} \circ T(q, \id_{X \boxtimes Y})
    = \can_{F \boxtimes G}(M \boxtimes N) \circ j_{L} \circ T(\id_L, q)
  \end{gather*}
  by the dinaturality of $j$.
  Since $T$ is right exact in each variable, $T(\id_L, q)$ is an epimorphism in $\mathcal{L}_2$. Thus the equation \eqref{eq:Nakayama-iso-Deligne-tensor-proof-1} follows. The proof is done.
\end{proof}

\subsection{Nakayama functor of the category of right exact functors}

For two finite abelian categories $\mathcal{M}$ and $\mathcal{N}$, there is an equivalence
\begin{equation}
  \label{eq:abstract-EW-equiv}
  \EW_{\mathcal{M}, \mathcal{N}}: \mathcal{M}^{\op} \boxtimes \mathcal{N} \to \REX(\mathcal{M}, \mathcal{N}),
  \quad M^{\op} \boxtimes N \mapsto \Hom_{\mathcal{M}}(-, M)^* \copow N
\end{equation}
of linear categories \cite[Lemma 2.3]{MR3569179}, which may be thought of as a `Morita invariant' version of \eqref{eq:EW-equiv}.
Given linear right exact functors  $F: \mathcal{M}_2 \to \mathcal{M}_1$ and $G: \mathcal{N}_1 \to \mathcal{N}_2$ between finite abelian categories $\mathcal{M}_1$, $\mathcal{M}_2$, $\mathcal{N}_1$ and $\mathcal{N}_2$, we define
\begin{equation}
  \label{eq:Rex-F-G-def}
  \REX(F, G): \REX(\mathcal{M}_1, \mathcal{N}_1) \to \REX(\mathcal{M}_2, \mathcal{N}_2),
  \quad X \mapsto G \circ X \circ F.
\end{equation}
The equivalence~\eqref{eq:abstract-EW-equiv} is `natural' in the sense that the diagram
\begin{equation*}
  \begin{tikzcd}[column sep = 64pt]
    (\mathcal{M}_1)^{\op} \boxtimes \mathcal{N}_1
    \arrow[r, "{\EW_{\mathcal{M}_1, \mathcal{N}_1}}"]
    \arrow[d, "(F^{\radj})^{\op} \boxtimes G"']
    & \REX(\mathcal{M}_1, \mathcal{N}_1)
    \arrow[d, "{\REX(F, G)}"] \\
    (\mathcal{M}_2)^{\op} \boxtimes \mathcal{N}_2
    \arrow[r, "{\EW_{\mathcal{M}_2, \mathcal{N}_2}}"]
    & \REX(\mathcal{M}_2, \mathcal{N}_2)
  \end{tikzcd}
\end{equation*}
commutes up to isomorphism.

\begin{lemma}
  \label{lem:Nakayama-REX}
  For finite abelian categories $\mathcal{M}$ and $\mathcal{N}$, we have
  \begin{equation*}
    \Nak_{\REX(\mathcal{M}, \mathcal{N})} = \REX(\Nak_{\mathcal{M}}, \Nak_{\mathcal{N}}).
  \end{equation*}
\end{lemma}
\begin{proof}
  Despite that this formula has been given in \cite[Lemma 3.21]{MR4042867}, we include a proof to explain how we view the functor $\REX(\Nak_{\mathcal{M}}, \Nak_{\mathcal{N}})$ as a Nakayama functor of $\mathcal{F} := \REX(\mathcal{M}, \mathcal{N})$. By the `naturality' of the equivalence \eqref{eq:abstract-EW-equiv} mentioned in the above, we have isomorphisms
  \begin{gather*}
    \REX(\Nak_{\mathcal{M}}, \Nak_{\mathcal{N}}) \circ \EW
    \cong \EW \circ ((\Nak_{\mathcal{M}}^{\radj})^{\op} \boxtimes \Nak_{\mathcal{N}}) \\
    \mathop{=}^{\eqref{eq:Nakayama-oppo}}
    \EW \circ (\Nak_{\mathcal{M}^{\op}} \boxtimes \Nak_{\mathcal{N}})
    \mathop{\cong}^{\eqref{eq:Nakayama-Deligne-tensor}} \EW \circ \Nak_{\mathcal{M}^{\op} \boxtimes \mathcal{N}}
    \cong \Nak_{\mathcal{F}} \circ \EW,
  \end{gather*}
  where $\EW = \EW_{\mathcal{M}, \mathcal{N}}$.
  Hence $\Nak_{\mathcal{F}} \cong \REX(\Nak_{\mathcal{M}}, \Nak_{\mathcal{N}})$. The proof is done.
\end{proof}

Now let $F$ and $G$ be as in \eqref{eq:Rex-F-G-def}.
We assume that $F$ has a triple left adjoint and $G$ has a triple right adjoint.
Then the functor $T := \REX(F, G)$ has a triple right adjoint. Indeed, we have a chain of adjunctions
\begin{equation*}
  \REX(F, G)
  \dashv \REX(F^{\ladj}, G^{\radj})
  \dashv \REX(F^{\lladj}, G^{\rradj})
  \dashv \REX(F^{\llladj}, G^{\rrradj}).
\end{equation*}
Thus there is a canonical isomorphism
\begin{equation*}
  \can_T : T^{\rradj} \circ \Nak_{\mathcal{F}_1} \to \Nak_{\mathcal{F}_2} \circ T,
\end{equation*}
where $\mathcal{F}_i = \REX(\mathcal{M}_i, \mathcal{N}_i)$ ($i = 1, 2$).
We choose $\REX(\Nak_{\mathcal{M}_i}, \Nak_{\mathcal{N}_i})$ as a Nakayama functor of $\mathcal{F}_i$. We also choose $\REX(F^{\lladj}, G^{\rradj})$ as a double right adjoint of $T$. Under our choice of adjoints and Nakayama functors, $\can_T$ is a family
\begin{equation*}
  \can_T = \{ 
  \can_T(X) : G^{\rradj} \circ \Nak_{\mathcal{N}_1} \circ X \circ \Nak_{\mathcal{M}_1} \circ F^{\lladj}
  \to \Nak_{\mathcal{N}_2} \circ G \circ X \circ F \circ \Nak_{\mathcal{M}_2} \}_{X \in \mathcal{F}_1}
\end{equation*}
of morphisms in $\mathcal{F}_2$. We give an explicit formula of $\can_T$ as follows:

\begin{lemma}
  Notations are as above. For $X \in \mathcal{F}_1$, we have
  \begin{equation*}
    \can_T(X) =  \can_{G} \circ \id_X \circ (\can_{F^{\lladj}})^{-1},
  \end{equation*}
  where $\circ$ means the horizontal composition of natural transformations.
\end{lemma}
\begin{proof}
  We choose $(\Nak_{\mathcal{M}_i}^{\radj})^{\op} \boxtimes \Nak_{\mathcal{N}_i}$ as a Nakayama functor of $\mathcal{T}_i := \mathcal{M}_i^{\op} \boxtimes \mathcal{N}_i$ as in the previous subsection. By Lemmas~\ref{lem:Nakayama-iso-opposite} and \ref{lem:Nakayama-iso-Deligne-tensor}, the canonical isomorphism
  \begin{equation*}
    \can_{(F^{\radj})^{\op} \boxtimes G}:
    \Nak_{\mathcal{T}_2} \circ ((F^{\radj})^{\op} \boxtimes G)
    \to ((F^{\radj})^{\op} \boxtimes G)^{\rradj} \circ \Nak_{\mathcal{T}_1}
  \end{equation*}
  is given by $\can_{(F^{\radj})^{\op} \boxtimes G} = (\can'_F)^{\op} \boxtimes \can_G$, where $\can'_F$ is the composition
  \begin{equation*}
    F^{\radj} \Nak_{\mathcal{M}}^{\radj}
    \xrightarrow{\quad \cong \quad}
    (\Nak_{\mathcal{M}} F)^{\radj}
    \xrightarrow{\quad (\can_{F^{\lladj}}^{-1})^{\radj} \quad}
    (F^{\lladj} \Nak_{\mathcal{N}})^{\radj}
    \xrightarrow{\quad \cong \quad}
    \Nak_{\mathcal{N}}^{\radj} (F^{\lladj})^{\radj}.
  \end{equation*}
  The proof is completed by translating this formula through the Eilenberg-Watts equivalence \eqref{eq:abstract-EW-equiv}.
\end{proof}

\section{Nakayama functor for monads}
\label{sec:for-monads}

\subsection{Tensor products over monads}

A {\em monad} \cite{MR1712872} on a category $\mathcal{M}$ is a triple $\mathbf{T} = (T, \mu, \eta)$ consisting of an endofunctor $T$ on $\mathcal{M}$ and natural transformations $\mu: TT \to T$ and $\eta: \id_{\mathcal{M}} \to T$ satisfying the associative and unit laws.
Let $\mathbf{T} = (T, \mu, \eta)$ be a monad on $\mathcal{M}$.
A {\em $\mathbf{T}$-module in $\mathcal{M}$} ($=$ a $\mathbf{T}$-algebra \cite{MR1712872}) is a pair $\mathbf{M} = (M, a)$ consisting of an object $M \in \mathcal{M}$ and a morphism $a: T(M) \to M$ in $\mathcal{M}$, called the action, satisfying the associative and the unit laws. Modules over $\mathbf{T}$ form a category, which we denote by $\lmod{\mathbf{T}}$.

If $\mathbf{T}$ is a linear right exact monad on a finite abelian category $\mathcal{M}$, then $\lmod{\mathbf{T}}$ is a finite abelian category such that the forgetful functor $\Forg: \lmod{\mathbf{T}} \to \mathcal{M}$ preserves and reflects exact sequences. The goal of this section is to prove Theorem~\ref{thm:Nakayama-monads}, which provides a formula of $\Nak_{\lmod{\mathbf{T}}}$ under the assumption that
\begin{equation}
  \label{eq:nice-monad}
  \text{$T$ admits a double left adjoint and a double right adjoint}.
\end{equation}

To state and prove the theorem, we introduce {\em the tensor product over a monad} as a slight generalization of the tensor product over a ring.
This idea comes from monads in a bicategory and related notions.
We write down necessary definitions in the form specialized to the 2-category of categories.
Given a functor $F$ and an object $M$ of the source of $F$, we often, but not always, write $F(M)$ as $F \otimes M$ in this section.
In response, the symbol $\otimes$ will also be used to express the composition of functors and the horizontal composition of natural transformations.
Let $\mathbf{S} = (S, \mu^S, \eta^S)$ and $\mathbf{T} = (T, \mu^T, \eta^T)$ be monads on categories $\mathcal{L}$ and $\mathcal{M}$, respectively.
Then a {\em left $\mathbf{S}$-module in $[\mathcal{M}, \mathcal{L}]$} is a pair $(F, a)$ consisting of a functor $F: \mathcal{M} \to \mathcal{L}$ and a natural transformation $a: S \otimes F \to F$ satisfying
\begin{gather*}
  a \circ (\id_S \otimes a) = a \circ (\mu^S \otimes \id_F)
  \quad \text{and} \quad
  a \circ (\eta^S \otimes \id_F) = \id_F.
\end{gather*}
A {\em right $\mathbf{T}$-module in $[\mathcal{M}, \mathcal{L}]$} is defined analogously. An {\em $\mathbf{S}$-$\mathbf{T}$-bimodule} in in $[\mathcal{M}, \mathcal{L}]$ is a triple $\mathbf{F} = (F, a^{\ell}, a^r)$ consisting of a functor $F: \mathcal{M} \to \mathcal{L}$ and natural transformations $a^{\ell}: S \otimes F \to F$ and $a^r: F \otimes T \to F$ such that $(F, a^{\ell})$ is a left $\mathbf{S}$-module in $[\mathcal{M}, \mathcal{L}]$, $(F, a^r)$ is a right $\mathbf{T}$-module in $[\mathcal{M}, \mathcal{L}]$, and the following equation is satisfied:
\begin{equation*}
  a^{r} \circ (a^{\ell} \otimes \id_T) = a^{\ell} \circ (\id_S \otimes a^r).
\end{equation*}

Now we suppose that the category $\mathcal{L}$ admits coequalizers.
For a right $\mathbf{T}$-module $\mathbf{F} = (F, a^r_F)$ in $[\mathcal{M}, \mathcal{L}]$ and a $\mathbf{T}$-module $\mathbf{M} = (M, a_M)$ in $\mathcal{M}$, we define
\begin{equation*}
  \mathbf{F} \otimes_{\mathbf{T}} \mathbf{M}
  = \mathrm{coequalizer} \Big(
  \begin{tikzcd}[column sep = 64pt]
    F \otimes T \otimes M
    \arrow[r, shift left=3pt, "{a^r_F \otimes \id_M}"]
    \arrow[r, shift right=3pt, "{\id_F \otimes a_M}"']
    & F \otimes M
  \end{tikzcd} \Big)
\end{equation*}
and call it the {\em tensor product} of $\mathbf{F}$ and $\mathbf{M}$ over $\mathbf{T}$.
If $\mathbf{F} = (F, a^{\ell}_F, a^r_F)$ is an $\mathbf{S}$-$\mathbf{T}$-module $[\mathcal{L}, \mathcal{M}]$ and $S$ preserves coequalizers, then $\mathbf{F} \otimes_{\mathbf{T}} \mathbf{M}$ is an $\mathbf{S}$-module by the action induced by $a^{\ell}_F \otimes \id_F \otimes \id_M$ and this construction gives rise to a functor
\begin{equation*}
  \mathbf{F} \otimes_{\mathbf{T}} (-) : \lmod{\mathbf{T}} \to \lmod{\mathbf{S}},
  \quad \mathbf{M} \mapsto \mathbf{F} \otimes_{\mathbf{T}} \mathbf{M}.
\end{equation*}
If, furthermore, $F$ preserves coequalizers, then the functor $\mathbf{F} \otimes_{\mathbf{T}} (-)$ also preserves coequalizers.

\subsection{Adjoints of (bi)modules over monads}

As the second important ingredient to state the main theorem, we discuss a natural module structure of an adjoint of a (bi)module over monads.
Throughout this subsection, $\mathbf{S} = (S, \mu^S, \eta^S)$ and $\mathbf{T} = (T, \mu^T, \eta^T)$ are monads on categories $\mathcal{L}$ and $\mathcal{M}$, respectively.
We first discuss a left adjoint of a left $\mathbf{S}$-module.
Let $\mathcal{N}$ be a category.
We note that if $F: \mathcal{M} \to \mathcal{L}$ is a functor admitting a left adjoint, then there is a canonical bijection
\begin{equation}
  \label{eq:nat-bij-adj-1}
  \Nat(X, F \otimes Y) \cong \Nat(F^{\ladj} \otimes X, Y)
\end{equation}
for functors $X: \mathcal{N} \to \mathcal{L}$ and $\mathcal{N} \to \mathcal{M}$. There is also a canonical bijection
\begin{equation}
  \label{eq:nat-bij-adj-2}
  \Nat(X \otimes F, Y) \cong \Nat(X, Y \otimes F^{\ladj})
\end{equation}
for functors $X : \mathcal{M} \to \mathcal{N}$ and $Y: \mathcal{L} \to \mathcal{N}$.

\begin{lemma}
  \label{lem:adjoint-module-1}
  Let $\mathbf{F} = (F, a^{\ell}_F)$ be a left $\mathbf{S}$-module in $[\mathcal{M}, \mathcal{L}]$ whose underlying functor $F$ admits a left adjoint. Then the left adjoint $F^{\ladj}$ is a right $\mathbf{S}$-module in $[\mathcal{L}, \mathcal{M}]$ by the natural transformation corresponding to $a_F^{\ell}$ via
  \begin{equation*}
    \Nat(S \otimes F, F)
    \xrightarrow{\quad \eqref{eq:nat-bij-adj-2} \quad}
    \Nat(S, F \otimes F^{\ladj})
    \xrightarrow{\quad \eqref{eq:nat-bij-adj-1} \quad}
    \Nat(F^{\ladj} \otimes S, F^{\ladj}).
  \end{equation*}
\end{lemma}

We denote by $\mathbf{F}^{\ladj}$ the right $\mathbf{S}$-module in $[\mathcal{M}, \mathcal{L}]$ of the above lemma.
We suppose that the functor $T$ has a double left adjoint. Then the functor $T^{\lladj}$ is naturally a monad on $\mathcal{M}$, which we denote by $\mathbf{T}^{\lladj}$.
As has been observed at least for algebras and modules in a monoidal category \cite[\S2.4.5]{MR4254952}, a left adjoint of an $\mathbf{S}$-$\mathbf{T}$-module in $[\mathcal{M}, \mathcal{L}]$ has no natural structure of a $\mathbf{T}$-$\mathbf{S}$-bimodule in general, however, it becomes a $\mathbf{T}^{\lladj}$-$\mathbf{S}$-bimodule in the following way:

\begin{lemma}
  \label{lem:adjoint-module-2}
  Suppose that $T$ has a double left adjoint.
  Let $\mathbf{F} = (F, a^r_F)$ be a right $\mathbf{T}$-module in $[\mathcal{M}, \mathcal{L}]$ such that $F$ admits a left adjoint.
  Then $F^{\ladj}$ is a left $\mathbf{T}^{\lladj}$-module in $[\mathcal{L}, \mathcal{M}]$ by the natural transformation corresponding to $a^r_F$ via
  \begin{equation*}
    \Nat(F \otimes T, F)
    \xrightarrow{\quad (-)^{\ladj} \quad}
    \Nat(F^{\ladj}, T^{\ladj} \otimes F^{\ladj})
    \xrightarrow{\quad \eqref{eq:nat-bij-adj-1} \quad}
    \Nat(T^{\lladj} \otimes F^{\ladj}, F^{\ladj}).
  \end{equation*}
  If $\mathbf{T}$ is an $\mathbf{S}$-$\mathbf{T}$-bimodule, then $F^{\ladj}$ is a $\mathbf{T}^{\lladj}$-$\mathbf{S}$-bimodule by this left action and the right action given by the previous lemma.
\end{lemma}

Lemmas \ref{lem:adjoint-module-1} and \ref{lem:adjoint-module-2} are proved straightforwardly.
Instead of giving the detailed proof of these lemmas, we provide diagrammatic expressions of the actions of monads on the functor $F^{\ladj}$. Thanks to our convention of denoting the composition of functors by $\otimes$, we may divert the standard graphical calculus in a monoidal category to express natural transformations.
We use the diagrams $\begin{array}{c} \begin{tikzpicture}[baseline=0pt, x=.75pc, y=.75pc] \draw (0,1) -- (0,-1); \draw[->] (-1,1) to [out=-90, in=180, looseness=1.5] (0,0); \end{tikzpicture} \end{array}$ and $\begin{array}{c} \begin{tikzpicture}[baseline=0pt, x=.75pc, y=.75pc] \draw (0,1) -- (0,-1); \draw[->] (1,1) to [out=-90, in=0, looseness=1.5] (0,0); \end{tikzpicture} \end{array}$ to represent a left and a right action of a monad on a functor. We also use a cap and a cup to represent the unit and the counit for an adjunction, respectively.
Then the right action of $\mathbf{S}$ and the left action of $\mathbf{T}^{\lladj}$ on $\mathbf{F}^{\ladj}$ given in the above lemmas are expressed as follows:
\begin{equation*}
  \begin{tikzpicture}[baseline = 0pt, x = 1pc, y = 1pc]
    \draw (0,2) node [above] {$F^{\ladj}$}
    -- (0,-2) node [below] {$F^{\ladj}$};
    \draw [->] (1.5, 2) node [above] {$S$}
    -- ++(0, -.5) to [out=-90, in=0, looseness=1.5] (0,0);
  \end{tikzpicture}
  :=
  \begin{tikzpicture}[baseline = 0pt, x = 1pc, y = 1pc]
    \draw (-3,2) node [above] {$F^{\ladj}$}
    -- ++(0, -2) to [out=-90, in=-90, looseness=2] ++(3,0) coordinate (P)
    -- ++(0,1)
    to [out=90, in=90, looseness=2] ++(2,0)
    -- ++(0, -3) node [below] {$F^{\ladj}$};
    \draw [->] (-1.5,2) node [above] {$S$}
    -- ++(0, -.5) to [out=-90, in=180, looseness=1.5] (P);
  \end{tikzpicture}
  \qquad
  \begin{tikzpicture}[baseline = 0pt, x = 1pc, y = 1pc]
    \draw (0,2) node [above] {$F^{\ladj}$}
    -- (0,-2) node [below] {$F^{\ladj}$};
    \draw [->] (-1.5, 2) node [above] {$T^{\lladj}$}
    -- ++(0, -.5) to [out=-90, in=180, looseness=1.5] (0,0);
  \end{tikzpicture}
  :=
  \begin{tikzpicture}[baseline = 0pt, x = 1pc, y = 1pc]
    \draw (-3,2) node [above] {$F^{\ladj}$}
    -- ++(0, -3) to [out=-90, in=-90, looseness=2] ++(1,0)
    -- ++(0, .5) coordinate (P)
    -- ++(0,1)
    to [out=90, in=90, looseness=2] ++(3,0)
    -- ++(0, -2.5) node [below] {$F^{\ladj}$};
    \draw [->] (-4.5, 2) node [above] {$T^{\lladj}$}
    -- ++(0,-2.5)
    to [out=-90, in=-90, looseness=2] ($(P)+(2,0)$)
    -- ++(0,1)
    to [out=90, in=90, looseness=2] ++(-1,0)
    to [out=-90, in=0, looseness=1.5] (P);
  \end{tikzpicture}
\end{equation*}

Similarly, if $\mathbf{F}$ is a right $\mathbf{T}$-module in $[\mathcal{M}, \mathcal{L}]$ whose underlying functor $F$ admits a right adjoint, then $F^{\radj}$ has a natural structure of a left $\mathbf{T}$-module in $[\mathcal{L}, \mathcal{M}]$.
As in the case of left adjoints, we denote this right $\mathbf{T}$-module by $\mathbf{F}^{\radj}$.
If, furthermore, $\mathbf{F}$ is an $\mathbf{S}$-$\mathbf{T}$-bimodule in $[\mathcal{M}, \mathcal{L}]$ and $S$ admits a double right adjoint, then $\mathbf{F}^{\radj}$ is a $\mathbf{T}$-$\mathbf{S}^{\rradj}$-bimodule in $[\mathcal{L}, \mathcal{M}]$.
Graphically, the actions on $\mathbf{F}^{\radj}$ are expressed by the left-right inverted ones of the above diagrams expressing the actions on $\mathbf{F}^{\ladj}$.

\subsection{Adjunctions for monads}

Let $\mathcal{M}$ be a category admitting coequalizers, and let $\mathbf{T} = (T, \mu, \eta)$ be a monad on $\mathcal{M}$.
There is the forgetful functor
\begin{equation*}
  \Forg : \lmod{\mathbf{T}} \to \mathcal{M},
  \quad (M, a_M) \mapsto M.
\end{equation*}
As is well-known, $\Forg$ has a left adjoint
\begin{equation*}
  \Free : \mathcal{M} \to \lmod{\mathbf{T}},
  \quad X \mapsto (T(X), \mu_X)
\end{equation*}
called the {\em free $\mathbf{T}$-module functor}.
Below we describe a left adjoint of $\Free$ and a right adjoint of $\Forg$ by using adjoints of bimodules.

\begin{lemma}
  \label{lem:left-adjoint-of-free}
  Suppose that $T$ admits a left adjoint.
  Then the functor
  \begin{equation*}
    \Free^{\ladj} : \lmod{\mathbf{T}} \to \mathcal{M},
    \quad \mathbf{M} \mapsto \mathbf{T}^{\ladj} \otimes_{\mathbf{T}} \mathbf{M}
  \end{equation*}
  is left adjoint to $\Free$.
\end{lemma}
\begin{proof}
  Let $\mathbf{M} = (M, a_M) \in \lmod{\mathbf{T}}$ and $X \in \mathcal{M}$ be objects. One can check that a morphism $f : M \to T \otimes X$ in $\mathcal{M}$ satisfies the equation
  \begin{equation}
    \label{eq:left-adj-of-free-eq-1}
    f \circ a_M = \mu_X \circ T(f)
  \end{equation}
  if and only if the corresponding morphism $g : T^{\ladj} \otimes M \to X$ in $\mathcal{M}$ satisfies
  \begin{equation}
    \label{eq:left-adj-of-free-eq-2}
    g \circ (a^r \otimes \id_X) = g \circ (\id_{T^{\ladj}} \otimes a_M),
  \end{equation}
  where $a^r : T^{\ladj} \otimes T \to T^{\ladj}$ is the right action of $\mathbf{T}$ on $T^{\ladj}$. Hence we have natural isomorphisms
  \begin{align*}
    \Hom_{\lmod{\mathbf{T}}}(\mathbf{M}, \Free(X))
    & = \{ f \in \Hom_{\mathcal{M}}(M, T \otimes X)
      \mid \text{$f$ satisfies \eqref{eq:left-adj-of-free-eq-1}} \} \\
    & \cong \{ g \in \Hom_{\mathcal{M}}(T^{\ladj} \otimes M, X)
      \mid \text{$g$ satisfies \eqref{eq:left-adj-of-free-eq-2}} \} \\
    & \cong \Hom_{\mathcal{M}}(\mathbf{T}^{\ladj} \otimes_{\mathbf{T}} \mathbf{M}, X),
  \end{align*}
  where the first isomorphism is the adjunction isomorphism and the second one follows from the universal property of the coequalizer defining the tensor product over $\mathbf{T}$. The proof is done.
\end{proof}

Given a left $T$-module $\mathbf{F} = (F, a_F)$ in $[\mathcal{M}, \mathcal{M}]$ and $X \in \mathcal{M}$, we define
\begin{equation*}
  \mathbf{F} \otimes X := (F \otimes X, a_F \otimes \id_X) \in \lmod{\mathbf{T}}.
\end{equation*}

\begin{lemma}
  \label{lem:right-adjoint-of-forg}
  Suppose that $T$ admits a right adjoint. Then the functor
  \begin{equation*}
    \Forg^{\radj} : \mathcal{M} \to \lmod{\mathbf{T}},
    \quad X \mapsto \mathbf{T}^{\radj} \otimes X \quad (X \in \mathcal{M})
  \end{equation*}
  is right adjoint to $\Forg$.
\end{lemma}
\begin{proof}
  The functor $C := T^{\radj}$ has a natural structure of a comonad on $\mathcal{M}$, and the category of $T$-modules is identified with the category of $C$-comodules. Under this identification, a right adjoint of $\Forg$ is given by the free $C$-comodule functor
  \begin{equation*}
    F_C: \mathcal{M} \to \text{(the category of $C$-comodules)},
    \quad X \mapsto (C \otimes X, \Delta \otimes \id_X),
  \end{equation*}
  where $\Delta : C \to C \otimes C$ is the comultiplication of $C$.
  The $\mathbf{T}$-module $\mathbf{T}^{\radj} \otimes X$ in concern is nothing but the $C$-comodule $F_C(X)$ viewed as a $\mathbf{T}$-module. Hence a right adjoint of $\Forg$ is given as stated.
\end{proof}

\subsection{Nakayama functor for monads}

Let $\mathbf{T} = (T, \mu^T, \eta^T)$ be a linear monad on a finite abelian category $\mathcal{M}$ satisfying \eqref{eq:nice-monad}.
As in the previous subsection, we denote by $\Free$ and $\Forg$ the free and the forgetful functor for the monad $\mathbf{T}$.
Since $T$ is right exact, the category $\lmod{\mathbf{T}}$ is a finite abelian category such that $\Forg$ preserves and reflects exact sequences.
By the coherence property of the canonical isomorphism \eqref{eq:Nakayama-cano-iso}, the Nakayama functor $\Nak_{\mathcal{M}}$ lifts to the functor
\begin{equation}
  \label{eq:Nakayama-lift}
  \lmod{\mathbf{T}^{\lladj}} \to \lmod{\mathbf{T}},
  \quad (M, a_M) \mapsto (\Nak_{\mathcal{M}}(M), \widetilde{a}_M),
\end{equation}
where $\widetilde{a}_M$ is the composition
\begin{equation*}
  T \otimes \Nak_{\mathcal{M}}(M)
  \xrightarrow{\quad \can_{T^{\lladj}}(M) \quad}
  \Nak_{\mathcal{M}}(T^{\lladj} \otimes M)
  \xrightarrow{\quad \Nak_{\mathcal{M}}(a_M) \quad}
  \Nak_{\mathcal{M}}(M).
\end{equation*}
Now the main result of this section is stated as follows:

\begin{theorem}
  \label{thm:Nakayama-monads}
  The Nakayama functor of $\lmod{\mathbf{T}}$ is given by the composition
  \begin{equation}
    \label{eq:Nakayama-monads}
    \lmod{\mathbf{T}}
    \xrightarrow{\quad \mathbf{T}^{\ladj} \otimes_{\mathbf{T}} (-) \quad}
    \lmod{\mathbf{T}^{\lladj}}
    \xrightarrow{\quad \eqref{eq:Nakayama-lift} \quad}
    \lmod{\mathbf{T}}.
  \end{equation}
\end{theorem}

We recall that $\Free$ has a left adjoint given by Lemma~\ref{lem:left-adjoint-of-free}.
Since $\Free$ is left adjoint to $\Forg$, there is a natural isomorphism
\begin{equation*}
  \psi_X : \Free^{\ladj} \Free(X)
  = \Free^{\ladj} \Forg^{\ladj}(X)
  \xrightarrow{\quad \cong \quad}
  (\Forg \Free)^{\ladj}(X) = T^{\ladj} \otimes X
  \quad (X \in \mathcal{M})
\end{equation*}
by the uniqueness of adjoints.
To prove the above theorem, we give the following technical remark:

\begin{lemma}
  The isomorphism $\psi_X$ is a morphism
  \begin{equation*}
    \psi_X : \mathbf{T}^{\ladj} \otimes_{\mathbf{T}} (\mathbf{T} \otimes X)
    \to \mathbf{T}^{\ladj} \otimes X
  \end{equation*}
  of $\mathbf{T}^{\lladj}$-modules in $\mathcal{M}$.
\end{lemma}
\begin{proof}
  Let $a^{\ell} : T^{\lladj} \otimes T^{\ladj} \to T^{\ladj}$ and $a^r : T^{\ladj} \otimes T \to T^{\ladj}$ be the left and the right action of $\mathbf{T}^{\lladj}$ and $\mathbf{T}$ on $T^{\ladj}$, respectively. We note that the action $\overline{a}{}^{\ell}$ of $\mathbf{T}^{\lladj}$ on the source of $\psi_X$ is determined by the equation
  \begin{equation*}
    \overline{a}{}^{\ell} \circ (\id_{T^{\lladj}} \otimes \pi_X)
    = \pi_X \circ (a^{\ell} \otimes \id_X),
  \end{equation*}
  where $\pi_X : T^{\ladj} \otimes T \otimes X \to \mathbf{T}^{\ladj} \otimes_{\mathbf{T}} (\mathbf{T} \otimes X)$ is the canonical epimorphism.

  There are natural isomorphisms
  \begin{align}
    \label{eq:Nakayama-monads-psi-1}
    \Hom_{\mathcal{M}}(T^{\ladj} \otimes X, Y)
    & \cong \Hom_{\mathcal{M}}(X, T \otimes Y) \\
    \label{eq:Nakayama-monads-psi-2}
    & \cong \Hom_{\lmod{\mathbf{T}}}(\Free(X), \Free(Y)) \\
    \label{eq:Nakayama-monads-psi-3}
    & \cong \Hom_{\mathcal{M}}(\Free^{\ladj} \Free(X), Y)
  \end{align}
  for $X, Y \in \mathcal{M}$.
  The isomorphism $\psi_X$ is the image of $\id_{T^{\ladj} \otimes X}$ under the above chain of isomorphisms with $Y = T^{\ladj} \otimes X$.
  Let $\eta$ and $\varepsilon$ be the unit and the counit for the adjunction $T^{\ladj} \dashv T$. The element $\id_{T^{\ladj} \otimes X}$ is sent as follows:
  \begin{equation*}
    \id_{T^{\ladj} \otimes X}
    \xmapsto{\ \eqref{eq:Nakayama-monads-psi-1} \ }
    \eta \otimes \id_X
    \xmapsto{\ \eqref{eq:Nakayama-monads-psi-2} \ }
    (\mu^T \otimes \id_{T^{\ladj}} \otimes \id_X)
    \circ (\id_T \otimes \eta \otimes \id_X).
  \end{equation*}
  By the proof of Lemma \ref{lem:left-adjoint-of-free}, the isomorphism \eqref{eq:Nakayama-monads-psi-3} sends an element $f$ to the unique morphism $g$ in $\mathcal{M}$ such that the equation
  \begin{equation*}
    g \circ \pi_X = (\varepsilon \otimes \id_T \otimes \id_X) \circ (\id_{T^{\ladj}} \otimes f)
  \end{equation*}
  holds. In conclusion, the isomorphism $\psi_X$ is the unique morphism in $\mathcal{M}$ such that the equation $\psi_X \circ \pi = a^r \otimes \id_X$ holds. By this equation and the fact that $(T^{\ladj}, a^{\ell}, a^r)$ is a bimodule, we have the equation
  \begin{align*}
    (a^{\ell} \otimes \id_X) \circ (\id_{T^{\lladj}} \otimes \psi_X \pi_X)
    = (\overline{a}{}^{\ell} \otimes \id_X)
    \circ (\id_{T^{\lladj}} \otimes \psi_X \pi_X),
  \end{align*}
  which implies that $\psi_X$ is a morphism of left $\mathbf{T}^{\lladj}$-modules in $\mathcal{M}$.
\end{proof}

\begin{proof}[Proof of Theorem \ref{thm:Nakayama-monads}]
  By Lemma~\ref{lem:right-adjoint-of-forg}, $\Forg$ has a right adjoint.
  Since $\Forg$ is a double right adjoint of $\Free^{\ladj}$, we have a canonical isomorphism
  \begin{equation*}
    \xi_{\mathbf{M}} := \can_{\Free^{\ladj}}(\mathbf{M}) : \Forg \Nak_{\lmod{\mathbf{T}}}(\mathbf{M})
    \xrightarrow{\quad \eqref{eq:Nakayama-cano-iso} \quad}
    \Nak_{\mathcal{M}} \Free^{\ladj}(\mathbf{M})
    \quad (\mathbf{M} \in \lmod{\mathbf{T}}).
  \end{equation*}
  By Lemma \ref{lem:left-adjoint-of-free}, the functor $\Nak_{\mathcal{M}}\Free^{\ladj}$ is the composition of \eqref{eq:Nakayama-monads} and $\Forg$.
  Thus, to complete the proof, it suffices to show that $\xi_{\mathbf{M}}$ is in fact a morphism
  \begin{equation}
    \label{eq:Nakayama-monads-proof-1}
    \xi_{\mathbf{M}} : \Nak_{\lmod{\mathbf{T}}}(\mathbf{M}) \to
    \Nak_{\mathcal{M}}(\mathbf{T}^{\ladj} \otimes_{\mathbf{T}} \mathbf{M})
  \end{equation}
  of $\mathbf{T}$-modules in $\mathcal{M}$ for all $\mathbf{M} \in \lmod{\mathbf{T}}$.

  We first verify that \eqref{eq:Nakayama-monads-proof-1} is a morphism of $\mathbf{T}$-modules in the case where $\mathbf{M}$ is a free $\mathbf{T}$-module.
  By Lemma~\ref{lem:right-adjoint-of-forg} and the assumption that $T$ has a double right adjoint, $\Forg^{\radj}$ is right exact. This implies that $\Free$ has a triple right adjoint. Hence there is the following natural isomorphism of $\mathbf{T}$-modules:
  \begin{equation}
    \label{eq:Nakayama-monads-proof-2}
    \can_{\Free}(X)^{-1} :
    \Nak_{\lmod{\mathbf{T}}} \Free(X)
    \xrightarrow{\quad \eqref{eq:Nakayama-cano-iso} \quad}
    \mathbf{T}^{\radj} \otimes \Nak_{\mathcal{M}}(X)
    \quad (X \in \mathcal{M}).
  \end{equation}
  By the naturality and the coherence property, we see that
  \begin{equation}
    \label{eq:Nakayama-monads-proof-3}
    \can_{T^{\radj}}(X) : \mathbf{T}^{\radj} \otimes \Nak_{\mathcal{M}}(X)
    \xrightarrow{\quad \eqref{eq:Nakayama-cano-iso} \quad}
    \Nak_{\mathcal{M}}(\mathbf{T}^{\ladj} \otimes X)
    \quad (X \in \mathcal{M})
  \end{equation}
  is an isomorphism of $\mathbf{T}$-modules.
  By composing \eqref{eq:Nakayama-monads-proof-2},
  \eqref{eq:Nakayama-monads-proof-3} and $\Nak_{\mathcal{M}}(\psi_X^{-1})$, we obtain the following natural isomorphism of $\mathbf{T}$-modules:
  \begin{equation*}
    \zeta_X : \Nak_{\lmod{\mathbf{T}}} \Free(X)
    \to \Nak_{\mathcal{M}}(\mathbf{T}^{\ladj} \otimes_{\mathbf{T}} \Free(X))
    \quad (X \in \mathcal{M}).
  \end{equation*}
  We recall that $\psi_X$ is defined by the uniqueness of adjoints.
  By the naturality and the coherence property of the canonical isomorphism \eqref{eq:Nakayama-cano-iso}, we see that the following diagram is commutative:
  \begin{equation*}
    \begin{tikzcd}[column sep = 72pt]
      \Forg \Forg^{\radj} \Nak_{\mathcal{M}}(X)
      \arrow[r, "{\Forg(\can_{\Free}(X))}"]
      \arrow[d, equal]
      & \Forg \Nak_{\mathcal{M}} \Free(X)
      \arrow[r, "{\can_{\Free^{\ladj}}(X)}", "(=\,\xi_{\Free(X)})"']
      & \Nak_{\mathcal{M}}(\Free^{\ladj} \Free(X))
      \arrow[d, "{\Nak_{\mathcal{M}}(\psi_X)}"] \\
      T^{\radj} \otimes \Nak_{\mathcal{M}}(X)
      \arrow[rr, "{\can_{T^{\ladj}}(X)}"]
      & & \Nak_{\mathcal{M}}(T^{\ladj} \otimes X).
    \end{tikzcd}
  \end{equation*}
  Hence $\xi_{\Free(X)} = \zeta_X$ for all $X \in \mathcal{M}$.
  By the construction, $\zeta_X$ is an isomorphism of $\mathbf{T}$-modules.
  Thus $\xi_{\Free(X)}$ is also an isomorphism of $\mathbf{T}$-modules.

  Now let $\mathbf{M}$ be an arbitrary $\mathbf{T}$-module in $\mathcal{M}$. Then there is an exact sequence in $\lmod{\mathbf{T}}$ of the form $\Free(X) \to \Free(Y) \to \mathbf{M} \to 0$ for some $X, Y \in \mathcal{M}$. We consider the following diagram:
  \begin{equation*}
    \begin{tikzcd}
      \Nak_{\lmod{\mathbf{T}}}\Free(X) \arrow[r]
      \arrow[d, "{\xi_{\Free(X)}}"]
      & \Nak_{\lmod{\mathbf{T}}}\Free(Y) \arrow[r]
      \arrow[d, "{\xi_{\Free(Y)}}"]
      & \Nak_{\lmod{\mathbf{T}}}(\mathbf{M}) \arrow[r]
      \arrow[d, dashed] & 0 \\
      \Nak_{\mathcal{M}}(\mathbf{T}^{\ladj} \otimes_{\mathbf{T}} \Free(X)) \arrow[r]
      & \Nak_{\mathcal{M}}(\mathbf{T}^{\ladj} \otimes_{\mathbf{T}} \Free(Y)) \arrow[r]
      & \Nak_{\mathcal{M}}(\mathbf{T}^{\ladj} \otimes_{\mathbf{T}} \mathbf{M})
      \arrow[r] & 0
    \end{tikzcd}
  \end{equation*}
  Since $\Nak_{\lmod{\mathbf{T}}}$ and \eqref{eq:Nakayama-monads} are right exact functors, the rows are exact. Thus $\xi_{\mathbf{M}}$ is characterized as a unique morphism in $\mathcal{M}$ such that the above diagram commutes when the dashed arrow is filled with it.
  Since both $\xi_{\Free(X)}$ and $\xi_{\Free(Y)}$ are morphisms of $\mathbf{T}$-modules, so is $\xi_{\mathbf{M}}$.
  The proof is done.
\end{proof}

\section{Algebras with Frobenius traces}
\label{sec:algebras-w-Fb-trace}

\subsection{Modules in module categories}
\label{subsec:modules-module-cats}

Let $\mathcal{C}$ be a finite multi-tensor category, and let $A$ and $B$ be algebras in $\mathcal{C}$.
If $\mathcal{M}$ is a finite left $\mathcal{C}$-module category with action $\catactl: \mathcal{C} \times \mathcal{M} \to \mathcal{M}$, then the endofunctor $A \catactl (-)$ on $\mathcal{M}$ has a natural structure of a monad on $\mathcal{M}$. We define ${}_A \mathcal{M}$ to be the category of modules over this monad and call an object of ${}_A \mathcal{M}$ an left $A$-module in $\mathcal{M}$. In a similar way, the category $\mathcal{M}_B$ of right $B$-modules in $\mathcal{M}$ and the category ${}_A \mathcal{M}_B$ of $A$-$B$-bimodules in $\mathcal{M}$ are defined when $\mathcal{M}$ is a finite right $\mathcal{C}$-module category and a finite $\mathcal{C}$-bimodule category, respectively.

Since $\mathcal{C}$ itself is a finite $\mathcal{C}$-bimodule category, the category ${}_A \mathcal{C}_B$ of $A$-$B$-bimodules in $\mathcal{C}$ is defined. Let $\mathcal{M}$ be a finite left $\mathcal{C}$-module category. Given $X \in {}_A \mathcal{C}_B$ and $M \in {}_B \mathcal{M}$, their tensor product over $B$ is defined by
\begin{equation*}
  X \otimes_B M
  := \mathrm{coequalizer} \Big(
  \begin{tikzcd}[column sep = 64pt]
    X \catactl B \catactl M
    \arrow[r, shift left=3pt, "{\id_X \catactl a^{\ell}}"]
    \arrow[r, shift right=3pt, "{a^{r} \catactl \id_M}"']
    & X \catactl M
  \end{tikzcd} \Big) \in {}_A \mathcal{M},
\end{equation*}
where $a^{\ell} : B \catactl M \to M$ and $a^{r} : X \otimes B \to X$ are the left and the right action of $B$ on $M$ and $X$, respectively. This construction defines a functor $\otimes_B: {}_A \mathcal{C}_B \times {}_B \mathcal{M} \to {}_A \mathcal{M}$ that is linear and right exact in each variable. When $\mathcal{M}$ is a finite right $\mathcal{C}$-module category, there is a functor $\otimes_A: \mathcal{M}_A \times {}_A \mathcal{C}_B \to \mathcal{M}_B$ defined in a similar way (these functors are a special case of the tensor product of modules over monads introduced in the previous section).
Theorem \ref{thm:Nakayama-monads} gives a formula for the Nakayama functor of categories of the form ${}_A \mathcal{M}$, $\mathcal{M}_B$ or ${}_A \mathcal{M}_B$ as follows:

\begin{theorem}
  \label{thm:Nakayama-for-modules}
  Let $A$ and $B$ be algebras in a finite multi-tensor category $\mathcal{C}$.
  \begin{enumerate}
  \item For a finite left $\mathcal{C}$-module category $\mathcal{M}$, we have
    \begin{equation*}
      \Nak_{{}_{A}\mathcal{M}}(\mathbf{M})
      = \Nak_{\mathcal{M}}(A^{\vee} \otimes_A \mathbf{M})
      \quad (\mathbf{M} \in {}_A \mathcal{M}).
    \end{equation*}
  \item For a finite right $\mathcal{C}$-module category $\mathcal{M}$, we have
    \begin{equation*}
      \Nak_{\mathcal{M}_B}(\mathbf{M})
      = \Nak_{\mathcal{M}}(\mathbf{M} \otimes_B {}^{\vee\!}B)
      \quad (\mathbf{M} \in \mathcal{M}_B).
    \end{equation*}
  \item For a finite $\mathcal{C}$-bimodule category $\mathcal{M}$, we have
    \begin{equation*}
      \Nak_{{}_{A}\mathcal{M}_{B}}(\mathbf{M})
      = \Nak_{\mathcal{M}}(A^{\vee} \otimes_A \mathbf{M} \otimes_B {}^{\vee\!}B)
      \quad (\mathbf{M} \in {}_A \mathcal{M}_B).
    \end{equation*}
  \end{enumerate}
\end{theorem}

In this section, we further investigate the Nakayama functor of the category of modules in the case where the algebra is a kind of Frobenius algebra.

\subsection{Algebras with a Frobenius trace}

Let $\mathcal{C}$ be a rigid monoidal category.
The notion of a Frobenius algebra in $\mathcal{C}$ is defined as a generalization of a Frobenius algebra (see, {\it e.g.}, \cite{MR2500035}).
If $A$ is a Frobenius algebra in $\mathcal{C}$, then $A \cong {}^{\vee}\!A$ and $A \cong A^{\vee}$ as left and right $A$-modules in $\mathcal{C}$, respectively. Hence the Nakayama functor of the category of $A$-modules would look a simpler form than that given in Theorem \ref{thm:Nakayama-for-modules}. In view of further applications, we consider a more general setting:

\begin{definition}
  \label{def:I-Frob-trace}
  Given an invertible object $I \in \mathcal{C}$, an {\em $I$-valued Frobenius trace} on an algebra $A$ in $\mathcal{C}$ is a morphism $\lambda: A \to I$ in $\mathcal{C}$ such that the morphism
  \begin{equation*}
    \phi :=
    (\eval_I \otimes \id_{A^{\vee}})
    \circ (\id_{I^{\vee}} \otimes \lambda m \otimes \id_{A^{\vee}})
    \circ (\id_{I^{\vee}} \otimes \coev_A):
    I^{\vee} \otimes A \to A^{\vee}
  \end{equation*}
  is invertible, where $m : A \otimes A \to \unitobj$ is the multiplication of $A$.  
\end{definition}

For example, we can say that a Frobenius algebra in $\mathcal{C}$ is an algebra $A$ in $\mathcal{C}$ endowed with a $\unitobj$-valued Frobenius trace.
A non-zero cointegral on a Hopf algebra in a braided finite multi-tensor category is an $I$-valued Frobenius trace with $I$ the `object of integrals' of $H$ \cite{MR1685417,MR1759389}.

Now we fix an invertible object $I \in \mathcal{C}$ and an algebra $A$ in $\mathcal{C}$ equipped with an $I$-valued Frobenius trace $\lambda$.
We define $\phi$ as above, and set
\begin{equation*}
  \beta := \eval_I \circ (\id_{I^{\vee}} \otimes \lambda m)
  = \eval_A \circ (\phi \otimes \id_{A}).
\end{equation*}
The following graphical expressions of $\beta$ and $\phi$ may be helpful:
\begin{equation*}
  \beta =
  \begin{tikzpicture}[baseline = 0pt, x = 8pt, y = 8pt]
    \draw (0, 3) coordinate (S1) node [above] {$I^{\vee}$};
    \draw (2.5, 3) coordinate (S2) node [above] {$A$};
    \draw let \p1 = (S2), \p2 = ($(\p1)-(0, 1)$),
    in (\p1) -- (\p2)
    to [out=-90, in=-90, looseness=2] coordinate[midway] (M1) ++(2,0)
    -- ++(0, 1) node [above] {$A$};
    \draw (M1) node {$\bullet$} -- ++(0, -1)
    node (B1) [below, draw, rectangle] {$\lambda$};
    \draw let \p1 = (S1), \p2 = (B1.south)
    in (S1) -- (\x1, \y2) to [out=-90, in=-90, looseness=2] (\p2);
  \end{tikzpicture}
  \qquad \phi =
  \begin{tikzpicture}[baseline = 0pt, x = 8pt, y = 8pt]
    \draw (0, 3) coordinate (S1) node [above] {$I^{\vee}$};
    \draw (2.5, 3) coordinate (S2) node [above] {$A$};
    \draw (7, -3) coordinate (T1) node [below] {$A^{\vee}$};
    \draw let \p1 = (S2), \p2 = ($(\p1)-(0, 1)$), \p3 = (T1)
    in (\p1) -- (\p2)
    to [out=-90, in=-90, looseness=2] coordinate [midway] (M1) ++(2,0)
    to [out=+90, in=+90, looseness=2] (\x3, \y2) -- (T1);
    \draw (M1) node {$\bullet$} -- ++(0, -1)
    node (B1) [below, draw, rectangle] {$\lambda$};
    \draw let \p1 = (S1), \p2 = (B1.south)
    in (S1) -- (\x1, \y2) to [out=-90, in=-90, looseness=2] (\p2);
  \end{tikzpicture}
  \qquad
  \left( m =
    \begin{tikzpicture}[baseline = -2pt, x = 10pt, y = 10pt]
      \draw (-1,1) node [above] {$A$}
      to [out=-90, in=-90, looseness=2]
      coordinate [midway] (M1) ++(2,0)
      node [above] {$A$};
      \draw (M1) node {$\bullet$} -- (0,-1) node [below] {$A$};
    \end{tikzpicture}
  \right)
\end{equation*}
It is easy to see that $\phi$ is an isomorphism of right $A$-modules.
The map $\lambda \mapsto \phi$ gives a bijection between the set of $I$-valued Frobenius traces on $A$ and the set of isomorphisms $I^{\vee} \otimes A \to A^{\vee}$ of right $A$-modules in $\mathcal{C}$.

\begin{definition}
  The {\em Nakayama isomorphism} of $A$ associated to $\lambda$ is
  \begin{equation}
    \label{eq:I-Frob-alg-Nakayama-iso}
    \nu := (\phi^{-1} \otimes \id_{I^{\vee\vee}}) \circ \phi^{\vee} : A^{\vee\vee} \to I^{\vee} \otimes A \otimes I^{\vee\vee}.
  \end{equation}
\end{definition}

When $\mathcal{C}$ is pivotal and $I = \unitobj$, the composition of $\nu$ and the pivotal structure of $\mathcal{C}$ is equal to the Nakayama automorphism of $A$ introduced in \cite[Section 5]{MR2500035} (since $\beta$, $\phi$ and ${}^{\vee\!}\phi$ correspond to $\kappa$, $\Phi_{\kappa, \mathrm{r}}$ and  $\Phi_{\kappa, \mathrm{l}}$ of \cite{MR2500035}, respectively).
Thus we may view $\nu$ as an analogue of the Nakayama automorphism of a Frobenius algebra, although it is no more an `automorphism' of $A$.

The Nakayama automorphism of a Frobenius algebra is defined by an equation like `$\lambda(ba) = \lambda(\nu(a) b)$.' The following lemma can be thought of as an analogue of such an equation.

\begin{lemma}
  \label{lem:gFb-algebra-Nakayama}
  With notation as above, we have
  \begin{equation}
    \label{eq:gFb-algebra-Nakayama}
    \begin{tikzpicture}[baseline = 0pt, x = 8pt, y = 8pt]
      \node (B1) at (0, -2) [below, draw, rectangle] {\makebox[2em]{$\beta$}};
      \draw let \p1 = ([shift={(-1, 0)}]B1.north)
      in (\p1) -- (\x1, 3) node [above] {$I^{\vee}$};
      \draw let \p1 = ([shift={(0, 0)}]B1.north)
      in (\p1) to [looseness=1.5, out=90, in=-90] ($(\x1, 3)+(1.5,0)$)
      node [above] {$A$};
      \draw let \p1 = ([shift={(+1, 0)}]B1.north),
      \p2 = ($(\p1) + (2,0)$), \p3 = (B1.south)
      in (\p1) to [looseness=2, out=90, in=90] (\p2)
      -- ($(\x2, \y3) + (0, 1)$)
      to [looseness=1.5, out=-90, in=-90] ++(-7,0) coordinate (T);
      \draw let \p1 = (T)
      in (T) to [looseness=1.5, out=90, in=-90] ($(\x1, 3)+(0, 0)$)
      node [above] {$A^{\vee\vee}$};
    \end{tikzpicture} =
    \begin{tikzpicture}[baseline = 0pt, x = 8pt, y = 8pt]
      \node (B1) at (1.5, -2) [below, draw, rectangle] {\makebox[2em]{$\beta$}};
      \coordinate (S1) at (0, 3); \node at (S1) [above] {$A^{\vee\vee}$};
      \coordinate (S2) at (2.5, 3); \node at (S2) [above] {$I^{\vee}$};
      \coordinate (S3) at (5.0, 3); \node at (S3) [above] {$A$};
      \draw (S1) -- ++(0, -.75)
      node [below, draw, rectangle] (B2) {\makebox[1em]{$\nu$}};
      \draw let \p1 = ([shift={(.75, 0)}]B2.south), \p2=(S2)
      in (\p1) to [looseness=2, out=-90, in=-90] (\x2, \y1) -- (S2);
      \draw let \p1 = ([shift={(1, 0)}]B1.north), \p2 = (S3)
      in (\p1) to [looseness=1.5, out=90, in=-90] (\p2);
      \draw let \p1 = ([shift={(0, 0)}]B2.south),
      \p2 = ([shift={(0, 0)}]B1.north)
      in (\p1) to [looseness=1.5, out=-90, in=90] (\p2);
      \draw let \p1 = ([shift={(-.75, 0)}]B2.south),
      \p2 = ([shift={(-1, 0)}]B1.north)
      in (\p1) to [looseness=1.5, out=-90, in=90] (\p2);
    \end{tikzpicture}
  \end{equation}
\end{lemma}
\begin{proof}
  The equation is obvious if one rewrites $\beta$ and $\nu$ by $\phi$.
\end{proof}

For simplicity of notation, we write $A^I = I^{\vee} \otimes A \otimes I^{\vee\vee}$ and
\begin{equation*}
  \rho = (\id_{I^{\vee\vee}} \otimes m)
  \circ (\id_{I^{\vee\vee}} \otimes \id_A \otimes \eval_{I^{\vee}} \otimes \id_A)
  :A^I \otimes I^{\vee} \otimes A \to I^{\vee} \otimes A.
\end{equation*}
The object $A^I$ is an algebra in $\mathcal{C}$ with multiplication $\rho \otimes \id_{I^{\vee\vee}}$.
By using the above lemma, one can show that $\nu : A^{\vee\vee} \to A^I$ is an isomorphism of algebras in $\mathcal{C}$ in a similar way as \cite[Section 5]{MR2500035}.
The object $I^{\vee} \otimes A$ is an $A^I$-$A$-bimodule in $\mathcal{C}$ by the left action $\rho$ and the right action $\id_{I^{\vee}} \otimes m$.

\begin{lemma}
  \label{lem:phi-bimodule-morphism}
  The isomorphism $\phi : I^{\vee} \otimes A \to A^{\vee}$ in $\mathcal{C}$ is actually an isomorphism of $A^{\vee\vee}$-$A$-bimodules in $\mathcal{C}$ if we view the source of $\phi$ as a left $A^{\vee\vee}$-module through the Nakayama isomorphism $\nu : A^I \to A^{\vee\vee}$.
\end{lemma}
\begin{proof}
  Let $\rho$ be as above, and let $m' : A^{\vee\vee} \otimes A^{\vee} \to A^{\vee}$ be the left action of $A^{\vee\vee}$ on $A^{\vee}$. This lemma claims that the following equation holds:
  \begin{equation}
    \label{eq:I-Frob-alg-Nakayama-iso-2}
    \phi \circ \rho \circ (\nu \otimes \id_{I^{\vee}} \otimes \id_A)
    = m' \circ (\id_{A^{\vee\vee}} \otimes \phi).
  \end{equation}
  This is verified as in Figure~\ref{fig:proof-phi} by the graphical calculus.
  Here, the first and the fourth equality follow from the right $A$-linearity of $\phi$,
  the second and the fifth from the definition of the right action of $A^{\vee}$ on $A$,
  the third from the definition of $\phi^{\vee}$,
  and the last from the definition of the left action of $A^{\vee\vee}$ on $A^{\vee}$.
\end{proof}

\begin{remark}
  Let $m''$ be the right-hand side of \eqref{eq:I-Frob-alg-Nakayama-iso-2}. Then we have
  \begin{equation*}
    \nu' = (\phi^{-1} \circ m'' \circ (\id_{A^{\vee\vee}} \otimes \id_{I^{\vee}} \otimes u)) \otimes \id_{I^{\vee}}) \circ (\id_{A^{\vee\vee}} \otimes \coev_{I^{\vee}}) = \nu,
  \end{equation*}
  where $u : \unitobj \to A$ is the unit of $A$. This means that the above lemma can also be used as the definition of the Nakayama isomorphism of $A$.
\end{remark}

\begin{figure}
  \begin{equation*}
    \begin{tikzpicture}[baseline = 0pt, x = 10pt, y = 8pt]
      \draw (0,5) coordinate (S1) node [above] {$A^{\vee\vee}$};
      \draw (2.25,5) coordinate (S2) node [above] {$I^{\vee}$};
      \draw (4,5) coordinate (S3) node [above] {$A$};
      \draw (S1) -- ++(0, -1)
      node (B1) [below, draw, rectangle] {\makebox[1.2em]{$\nu$}};
      \draw let \p1 = ([shift={(.5, 0)}]B1.south), \p2 = (S2)
      in (\p1) to [out=-90, in=-90, looseness=2] (\x2, \y1) -- (\p2);
      \draw let \p1 = (B1.south), \p2 = (S3)
      in (\p1) to [out=-90, in=-90, looseness=2]
      coordinate [midway] (M1)
      (\x2, \y1) -- (\p2);
      \draw (M1) node {$\bullet$} -- ++ (0, -1.5) coordinate (M2);
      \node (B2) at ($(M2)-(.5,0)$) [below, draw, rectangle] {\makebox[1.2em]{$\phi$}};
      \draw let \p1 = ([shift={(-.5,0)}]B1.south), \p2 = ([shift={(-.5,0)}]B2.north)
      in (\p1) -- ++(0, -1) to [out=-90, in=90, looseness=1] (\p2);
      \draw let \p1 = (B2.south)
      in (\p1) -- (\x1, -5) node [below] {$A^{\vee}$};
    \end{tikzpicture}
    = \quad
    \begin{tikzpicture}[baseline = 0pt, x = 10pt, y = 8pt]
      \draw (0,5) coordinate (S1) node [above] {$A^{\vee\vee}$};
      \draw (2.25,5) coordinate (S2) node [above] {$I^{\vee}$};
      \draw (4,5) coordinate (S3) node [above] {$A$};
      \draw (S1) -- ++(0, -1) node (B1) [below, draw, rectangle] {\makebox[1.2em]{$\nu$}};
      \draw let \p1 = ([shift={(.5, 0)}]B1.south), \p2 = (S2)
      in (\p1) to [out=-90, in=-90, looseness=2] (\x2, \y1) -- (\p2);
      \draw let \p1 = ([shift={(-.5,0)}]B1.south)
      in (\p1) -- ++(0, -2) coordinate (M1);
      \node (B2) at ($(M1)+(.25,0)$) [below, draw, rectangle] {\makebox[1.2em]{$\phi$}};
      \draw let \p1 = (B1.south), \p2 = (B2.north) in (\p1) -- (\x1, \y2);
      \draw let \p1 = (B2.south) in (\p1) -- (\x1, -5) coordinate (T1);
      \node at (T1) [below] {$A^{\vee}$};
      \draw [->] let \p1 = (S3), \p2 = ($(T1)+(0, 2)$)
      in (\p1) -- ($(\x1, \y2) + (0, 4)$)
      to [out=-90, in=0, looseness=1] (\p2);
    \end{tikzpicture}
    =
    \begin{tikzpicture}[baseline = 0pt, x = 10pt, y = 8pt]
      \draw (0,5) coordinate (S1) node [above] {$A^{\vee\vee}$};
      \draw (2.25,5) coordinate (S2) node [above] {$I^{\vee}$};
      \draw (4,5) coordinate (S3) node [above] {$A$};
      \draw (S1) -- ++(0, -1) node (B1) [below, draw, rectangle] {\makebox[1.2em]{$\phi^{\vee}$}};
      \draw let \p1 = ([shift={(.5, 0)}]B1.south), \p2 = (S2)
      in (\p1) to [out=-90, in=-90, looseness=2] (\x2, \y1) -- (\p2);
      \draw (S3) -- ++(0, -3)
      to [out=-90, in=-90, looseness=2] coordinate [midway] (M1) ++(1.5,0)
      to [out=+90, in=+90, looseness=2] ++(1.5,0) coordinate (M2);
      \draw let \p1 = (M2) in (\p1) -- (\x1, -5) node [below] {$A^{\vee}$};
      \draw let \p1 = ([shift={(-.5, 0)}]B1.south), \p2 = (M1)
      in (\p1) -- ++(0,-2) to [out=-90, in=-90, looseness=2] ($(\x2, \y1) - (0,2)$)
      -- (\p2) node {$\bullet$};
    \end{tikzpicture}
    =
    \begin{tikzpicture}[baseline = 0pt, x = 10pt, y = 8pt]
      \draw (0,5) coordinate (S1) node [above] {$A^{\vee\vee}$};
      \draw (2.25,5) coordinate (S2) node [above] {$I^{\vee}$};
      \draw (4,5) coordinate (S3) node [above] {$A$};
      \draw (S3) -- ++(0, -3)
      to [out=-90, in=-90, looseness=2] coordinate [midway] (M1) ++(1.5,0)
      to [out=+90, in=+90, looseness=2] ++(1.5,0) coordinate (M2);
      \draw let \p1 = (M2) in (\p1) -- (\x1, -5) node [below] {$A^{\vee}$};
      \node at (M1) {$\bullet$};
      \draw (M1) -- ++(0,-1.5) coordinate (M3);
      \node (B1) at ($(M3)+(-.5, 0)$) [below, draw, rectangle] {\makebox[1.2em]{$\phi$}};
      \draw let \p1 = (S2), \p2 = ([shift={(-.5,0)}]B1.north)
      in (\p1) to [out=-90, in=90, looseness=1] (\p2);
      \draw let \p1 = (S1), \p2 = (B1.south)
      in (\p1) -- (\x1, \y2) to [out=-90, in=-90, looseness=2] (\p2);
    \end{tikzpicture}
  \end{equation*}
  \bigskip
  \begin{equation*}
    = \begin{tikzpicture}[baseline = 0pt, x = 10pt, y = 8pt]
      \draw (0,5) coordinate (S1) node [above] {$A^{\vee\vee}$};
      \draw (2.25,5) coordinate (S2) node [above] {$I^{\vee}$};
      \draw (4,5) coordinate (S3) node [above] {$A$};
      \path (S2) -- coordinate [midway] (M1) (S3);
      \node (B1) at ($(M1)-(0,3)$) [below, draw, rectangle] {\makebox[1.2em]{$\phi$}};
      \draw (S2) to [out=-90, in=90, looseness=1.5] ([shift={(-.5,0)}]B1.north);
      \draw (S3) to [out=-90, in=90, looseness=1.5] ([shift={(+.5,0)}]B1.north);
      \draw let \p1 = (B1.south), \p2 = ($(\p1)-(0,2)$), \p3=(S1)
      in (\p1) -- (\p2) coordinate (M1)
      to [out=-90, in=-90, looseness=2] (\x3, \y2) -- (\p3);
      \coordinate (T1) at (8, -5); \node at (T1) [below] {$A^{\vee}$};
      \draw [->] (T1) -- ++(0,7) to [out=90, in=90, looseness=2] ++(-2,0)
      to [out=-90, in=0, looseness=1.5] ($(M1)+(0,1)$);
    \end{tikzpicture}
    = \begin{tikzpicture}[baseline = 0pt, x = 10pt, y = 8pt]
      \draw (0,5) coordinate (S1) node [above] {$A^{\vee\vee}$};
      \draw (2.25,5) coordinate (S2) node [above] {$I^{\vee}$};
      \draw (4,5) coordinate (S3) node [above] {$A$};
      \path (S2) -- coordinate [midway] (M1) (S3);
      \node (B1) at ($(M1)-(0,2)$) [below, draw, rectangle] {\makebox[1.2em]{$\phi$}};
      \draw (S2) to [out=-90, in=90, looseness=1.5] ([shift={(-.5,0)}]B1.north);
      \draw (S3) to [out=-90, in=90, looseness=1.5] ([shift={(+.5,0)}]B1.north);
      \coordinate (T1) at (10, -5); \node at (T1) [below] {$A^{\vee}$};
      \draw (T1) -- ++(0,5) to [out=90, in=90, looseness=2] ++(-5,0)
      to [out=-90, in=-90, looseness=2] coordinate [midway] (M1) ++(1.75,0)
      to [out=+90, in=+90, looseness=2] ++(1.75,0) coordinate (M2);
      \draw let \p1 = (B1.south), \p2 = (M1)
      in (\p1) -- (\x1, \y2) to [out=-90, in=-90, looseness=2] (\p2) node {$\bullet$};
      \draw let \p1=(M2), \p2=(S1)
      in (M2) to [out=-90, in=-90, looseness=1.8] (\x2, \y1) -- (\p2);
    \end{tikzpicture}
    = \begin{tikzpicture}[baseline = 0pt, x = 10pt, y = 8pt]
      \draw (0,5) coordinate (S1) node [above] {$A^{\vee\vee}$};
      \draw (2.25,5) coordinate (S2) node [above] {$I^{\vee}$};
      \draw (4,5) coordinate (S3) node [above] {$A$};
      \path (S2) -- coordinate [midway] (M1) (S3);
      \node (B1) at ($(M1)-(0,2)$) [below, draw, rectangle] {\makebox[1.2em]{$\phi$}};
      \draw (S2) to [out=-90, in=90, looseness=1.5] ([shift={(-.5,0)}]B1.north);
      \draw (S3) to [out=-90, in=90, looseness=1.5] ([shift={(+.5,0)}]B1.north);
      \draw let \p1 = (B1.south) in (\p1) -- (\x1, -5) node [below] {$A^{\vee}$};
      \draw [->] let \p1 = (S1), \p2 = ([shift={(0, -2)}]B1.south)
      in (\p1) -- ++(0,-3) to [out=-90, in=180] (\p2);
    \end{tikzpicture}
  \end{equation*}
  \caption{}
  \label{fig:proof-phi}
\end{figure}

\subsection{A formula of the Nakayama functor}

Let $\mathcal{C}$ be a finite multi-tensor category, and let $\mathcal{M}$ be a finite left $\mathcal{C}$-module category.
Given a left $A$-module $\mathbf{M} = (M, a_M)$ in $\mathcal{M}$, one can make $I^{\vee} \catactl M$ into a left $A^{\vee\vee}$-module in $\mathcal{M}$ in a similar way as the left $A^{\vee\vee}$-module $I^{\vee} \otimes A$ in $\mathcal{C}$ appeared in Lemma \ref{lem:phi-bimodule-morphism}.
This construction gives rise to a functor
\begin{equation}
  \label{eq:I-Frob-alg-Nakayama-1}
  \begin{gathered}
    I^{\vee} \catactl (-) : {}_A\mathcal{M} \to {}_{A^{\vee\vee}}\mathcal{M},
    \quad (M, a_M) \mapsto (I^{\vee} \catactl M, \widetilde{a}_M) \\
    (\widetilde{a}_M := (\id_{I^{\vee}} \catactl a_M) \circ (\id_{I^{\vee}} \catactl \id_{A} \catactl \eval_{I^{\vee}} \catactl \id_M) \circ (\nu \catactl \id_M)).
  \end{gathered}
\end{equation}
Now we express the Nakayama functor of ${}_A \mathcal{M}$ by using this functor as follows:

\begin{theorem}
  \label{thm:Nakayama-for-modules-2}
  Let $A$ and $\mathcal{M}$ be as above.
  Then the Nakayama functor of ${}_A \mathcal{M}$ is given by the composition
  \begin{equation*}
    {}_A \mathcal{M}
    \xrightarrow{\quad \eqref{eq:I-Frob-alg-Nakayama-1} \quad} {}_{A^{\vee\vee}}\mathcal{M}
    \xrightarrow{\quad \eqref{eq:Nakayama-lift} \quad} {}_A \mathcal{M}.
  \end{equation*}
\end{theorem}
\begin{proof}
  By Lemma~\ref{lem:phi-bimodule-morphism}, we have natural isomorphisms
  \begin{equation*}
    A^{\vee} \otimes_A \mathbf{M}
    \cong (I^{\vee} \otimes A) \otimes_A \mathbf{M}
    \cong I^{\vee} \catactl \mathbf{M}
  \end{equation*}
  for $\mathbf{M} \in {}_A \mathcal{M}$. Now this theorem follows from Theorem \ref{thm:Nakayama-for-modules}.
\end{proof}

We note some immediate consequences of this theorem.
In the below, $\mathcal{C}$ is a finite multi-tensor category, $I$ is an invertible object of $\mathcal{C}$, and $\mathcal{M}$ is a finite left $\mathcal{C}$-module category.

\begin{corollary}
  Suppose that the module category $\mathcal{M}$ is Frobenius.
  Let $A$ be an algebra in $\mathcal{C}$.
  If $A$ admits an $I$-valued Frobenius trace, then ${}_A \mathcal{M}$ is Frobenius.
\end{corollary}

The converse of this corollary does not hold:
For an algebra $A$ in $\mathcal{C} = \Vect$, the category ${}_A \mathcal{C}$ ($= \lmod{A}$) is Frobenius if and only if $A$ is self-injective, however, there are self-injective algebras which are not Frobenius \cite{MR2894798}.

\begin{corollary}
  Suppose that $\mathcal{C}$ is pivotal.
  Let $A$ be a Frobenius algebra in $\mathcal{C}$, and let $\mho : A \to A$ be the Nakayama automorphism of $A$ defined in \cite{MR2500035}. Then the Nakayama functor of ${}_A \mathcal{C}$ is isomorphic to the functor
  \begin{equation*}
    {}_A \mathcal{C} \to {}_A \mathcal{C},
    \quad (M, a_M) \mapsto (M \otimes \alpha_{\mathcal{C}}, a_M \mho \otimes \id_{\alpha}),
  \end{equation*}
  where $\alpha_{\mathcal{C}} = \Nak_{\mathcal{C}}(\unitobj)$.
  Thus, when $\mathcal{C}$ is unimodular, then the Nakayama functor of ${}_A\mathcal{C}$ is given by twisting the action of $A$ by the Nakayama automorphism $\mho$, as in the case of ordinary Frobenius algebras.
\end{corollary}
\begin{proof}
  We identify an object $X \in \mathcal{C}$ with $X^{\vee\vee}$ by the pivotal structure of $\mathcal{C}$.
  Then, as noted in the above, $\mho$ is identified with the isomorphism $\nu$ defined by \eqref{eq:I-Frob-alg-Nakayama-iso}.
  By the formula \eqref{eq:FTC-Nakayama-formula} of the Nakayama functor of $\mathcal{C}$, we have $\Nak_{\mathcal{C}}(X) = X \otimes \alpha_{\mathcal{C}}$ for $X \in \mathcal{C}$. Now the result follows from Theorem \ref{thm:Nakayama-for-modules-2}.
\end{proof}

If $\mathcal{C}$ is pivotal and $A$ is a symmetric Frobenius algebra in the sense of \cite{MR2500035}, then we may assume that the Nakayama automorphism of $A$ is the identity. Hence, by the above corollary, we have:

\begin{corollary}
  Suppose that $\mathcal{C}$ is pivotal and unimodular. Let $A$ be a symmetric Frobenius algebra in $\mathcal{C}$. Then ${}_A \mathcal{C}$ is symmetric Frobenius.
\end{corollary}

The unimodularity assumption cannot be dropped from the above corollary. Indeed, the trivial algebra $A := \unitobj$ is a symmetric Frobenius algebra in $\mathcal{C}$, while ${}_A \mathcal{C}$ ($\cong \mathcal{C}$) is symmetric Frobenius only if $\mathcal{C}$ is unimodular.

\subsection{Braided Hopf algebras}
\label{subsec:braid-Hopf}

Let $\mathcal{B}$ be a braided finite multi-tensor category with braiding $\sigma$, and let $H$ be a Hopf algebra in $\mathcal{B}$ with comultiplication $\Delta$, counit $\varepsilon$ and antipode $S$. Then the category ${}_H \mathcal{B}$ is a finite multi-tensor category. Here we determine the modular object of ${}_H \mathcal{B}$ by using the results of this section.

We recall from \cite{MR1685417,MR1759389} basic notions and results on (co)integrals.
There is an invertible object $I := \mathrm{Int}(H) \in \mathcal{B}$ called the {\em object of integrals} in \cite{MR1759389}.
A {\em right integral} in $H$ is a morphism $\Lambda : H \to I$ in $\mathcal{B}$ satisfying $m \circ (\Lambda \otimes \id_H) = \Lambda \otimes \varepsilon$, where $m$ is the multiplication of $H$.
A {\em right cointegral} on $H$ is a morphism $\lambda : H \to I$ in $\mathcal{B}$ satisfying $(\lambda \otimes \id_H) \circ \Delta = \lambda \otimes u$, where $u$ is the unit of $H$.
It is known that a non-zero right (co)integral exists.
Furthermore, a non-zero right cointegral on $H$ is, in our terminology, an $I$-valued Frobenius trace on the algebra $H$.

We fix a non-zero right cointegral $\lambda: H \to I$ on $H$.
Below we compute the Nakayama isomorphism $\nu$ associated to $\lambda$.
For this purpose, we introduce some notations:
For $X \in \mathcal{B}$, we define the morphism $\psi_X$ in $\mathcal{B}$ by
\begin{equation*}
  \psi_X = (\eval_{X} \otimes \id_{X^{\vee\vee}}) \circ (\sigma_{X^{\vee}, X}^{-1} \otimes \id_{X^{\vee\vee}}) \circ (\id_X \otimes \coev_{X^{\vee}}) : X \to X^{\vee\vee}.
\end{equation*}
This morphism and its inverse are graphically expressed as follows:
\begin{equation*}
  \tikzset{cross/.style={preaction={-,draw=white,line width=6pt}}}%
  \psi_X =
  \begin{tikzpicture}[x = 8pt, y = 8pt, baseline = 0pt]
    \coordinate (S1) at (0,2); \node at (S1) [above] {$X$};
    \path let \p1 = (S1) in (\p1) -- ++(0,-.5)
    to [out=-90, in=90, looseness=1] ++(2,-2.5)
    to [out=-90, in=-90, looseness=2] ++(-2,0) coordinate (T1);
    \draw (T1) to [out=90, in=-90, looseness=1] ++(2,2.5) coordinate (T2);
    \draw let \p2 = ($(T2)+(2,0)$)
    in (T2) to [out=90, in=90, looseness=2] (\p2)
    -- (\x2, -2) node [below] {$X^{\vee\vee}$};
    \draw [cross] let \p1 = (S1) in (\p1) -- ++(0,-.5)
    to [out=-90, in=90, looseness=1] ++(2,-2.5)
    to [out=-90, in=-90, looseness=2] ++(-2,0) coordinate (T1);
  \end{tikzpicture}
  \!\!\!\!\!, \quad \psi_{X}^{-1} = \!\!\!
  \begin{tikzpicture}[x = 8pt, y = 8pt, baseline = 0pt]
    \coordinate (S1) at (0,-2); \node at (S1) [below] {$X$};
    \draw let \p1 = (S1) in (\p1) -- ++(0,.5)
    to [out=90, in=-90, looseness=1] ++(-2,2.5)
    to [out=90, in=90, looseness=2] ++(+2,0) coordinate (T1);
    \draw [cross] (T1)
    to [out=-90, in=90, looseness=1] ++(-2,-2.5) coordinate (T2);
    \draw let \p2 = ($(T2)-(2,0)$)
    in (T2) to [out=-90, in=-90, looseness=2] (\p2)
    -- (\x2, 2) node [above] {$X^{\vee\vee}$};
  \end{tikzpicture}
  \quad \left(\sigma_{X,Y} = \!\!\!\!\!\!
    \begin{array}{c}
      \begin{tikzpicture}[x = 8pt, y = 12pt, baseline = 0pt]
        \draw (0,1) to [out = -90, in = 90] (2,-1);
        \draw [cross] (2,1) to [out = -90, in = 90] (0,-1);
        \node at (0, 1) [above] {$X$};
        \node at (2, 1) [above] {$Y$};
        \node at (0,-1) [below] {$Y$};
        \node at (2,-1) [below] {$X$};
      \end{tikzpicture}
    \end{array}\!\!\!\!\!\!, \ \sigma_{X,Y}^{-1} = \!\!\!\!\!\!
    \begin{array}{c}
      \begin{tikzpicture}[x = 8pt, y = 12pt, baseline = 0pt]
        \draw (2,1) to [out = -90, in = 90] (0,-1);
        \draw [cross] (0,1) to [out = -90, in = 90] (2,-1);
        \node at (0, 1) [above] {$Y$};
        \node at (2, 1) [above] {$X$};
        \node at (0,-1) [below] {$X$};
        \node at (2,-1) [below] {$Y$};
      \end{tikzpicture}
    \end{array}
  \right).
\end{equation*}
The notion of a left integral in $H$ is defined in a similar way as a right one.
We fix a non-zero left integral $\Lambda^{\ell} : I \to H$ in $H$.
The left modular function is the morphism $\alpha_H : H \to \unitobj$ in $\mathcal{C}$ determined by
\begin{equation*}
  \Lambda^{\ell} \otimes \alpha_H = m \circ (\Lambda^{\ell} \otimes \id_H).
\end{equation*}
Given a morphism $f : H \to \unitobj$ in $\mathcal{C}$, we define
\begin{equation*}
  \accentset{\leftharpoonup}{f} := (f \otimes \id_H) \circ \Delta : H \to H.
\end{equation*}

\begin{lemma}
  \label{lem:braided-Nakayama}
  The isomorphism $\nu : H^{\vee\vee} \to H^I$ is given by
  \begin{equation*}
    \nu = (\sigma_{H, I^{\vee}} \otimes \id_{I^{\vee\vee}}) \circ ((S^{2} \circ \accentset{\leftharpoonup}{\alpha}_H \circ \psi_H^{-1}) \otimes \coev_{I^{\vee}}).
  \end{equation*}
\end{lemma}
\begin{proof}
  The monodromy \cite{MR1759389} around an invertible object $K \in \mathcal{B}$ is the natural isomorphism $\Omega(K) : \id_{\mathcal{B}} \to \id_{\mathcal{B}}$ uniquely determined by the equation
  \begin{equation*}
    \id_K \otimes \Omega(K)_X = \sigma_{X,K} \circ \sigma_{K,X}
    \quad (X \in \mathcal{B}).
  \end{equation*}
  For an invertible object $K \in \mathcal{B}$, one has
  \begin{equation}
    \label{eq:braided-Hopf-monodromy}
    \Omega(K^{\vee})_X = \Omega(K)^{-1}_X,
    \quad \Omega(K)_X \otimes \id_K = \sigma_{K,X} \circ \sigma_{X,K}
    \quad (X \in \mathcal{B}).
  \end{equation}
  There is an isomorphism $\mathcal{N} : H \to H$ in $\mathcal{B}$ determined by
  \begin{equation}
    \label{eq:braided-Hopf-Nakayama}
    \lambda \circ m \circ \sigma_{H,H} = \lambda \circ m \circ (\id_H \otimes \mathcal{N}).
  \end{equation}
  An explicit description of $\mathcal{N}$ is given in \cite[Lemma 5.7]{MR3569179} in terms of the antipode, the monodromy and the right modular function of $H$.
  We note that, in \cite[Lemma 5.7]{MR3569179}, the symbol $\alpha_H$ is used to express the right modular function of $H$, which is the inverse of the left one with respect to the convolution product. Thus, in our notation, the inverse of $\mathcal{N}$ is given by
  \begin{equation*}
    \mathcal{N}^{-1} = \Omega(I)_H^{-1} \circ S^{2} \circ \accentset{\leftharpoonup}{\alpha}_H.
  \end{equation*}
  By~\eqref{eq:braided-Hopf-Nakayama}, we have $\lambda \circ m = \lambda \circ m \circ \sigma_{H,H} \circ (\id_H \otimes \mathcal{N}^{-1})$.
  Thus the left hand side of \eqref{eq:gFb-algebra-Nakayama} is computed as follows:
  \begin{equation*}
    \tikzset{cross/.style={preaction={-,draw=white,line width=6pt}}}%
    \begin{tikzpicture}[baseline = 0pt, x = 7pt, y = 8pt]
      \node (B1) at (0, -2) [below, draw, rectangle] {\makebox[2em]{$\beta$}};
      \draw let \p1 = ([shift={(-1, 0)}]B1.north)
      in (\p1) -- (\x1, 3) node [above] {$I^{\vee}$};
      \draw let \p1 = ([shift={(0, 0)}]B1.north)
      in (\p1) to [looseness=1.5, out=90, in=-90] ($(\x1, 3)+(1.5,0)$)
      node [above] {$H$};
      \draw let \p1 = ([shift={(+1, 0)}]B1.north),
      \p2 = ($(\p1) + (2,0)$), \p3 = (B1.south)
      in (\p1) to [looseness=2, out=90, in=90] (\p2)
      -- ($(\x2, \y3) + (0, 1)$)
      to [looseness=1.5, out=-90, in=-90] ++(-7,0) coordinate (T);
      \draw let \p1 = (T)
      in (T) to [looseness=1.5, out=90, in=-90] ($(\x1, 3)+(0, 0)$)
      node [above] {$H^{\vee\vee}$};
    \end{tikzpicture} =
    \begin{tikzpicture}[baseline = 0pt, x = 10pt, y = 8pt]
      \coordinate (S1) at (0, 3); \node at (S1) [above] {$H^{\vee\vee}$};
      \coordinate (S2) at (2, 3); \node at (S2) [above] {$I^{\vee}$};
      \coordinate (S3) at (4, 3); \node at (S3) [above] {$H$};
      \draw (S3) -- ++(0,-1.25)
      to [out=-90, in=90, looseness=1] ++(1.5, -2)
      to [out=-90, in=-90, looseness=2]
      coordinate [midway] (M1) ++(-1.5, 0)
      coordinate (TMP);
      \draw [cross] (TMP) to [out=90, in=-90, looseness=1] ++(1.5, 2)
      coordinate (TMP);
      \node (B1) at (TMP) [above, draw, circle, fill = gray] {};
      \draw let \p1=(B1.north), \p2 = ($(\p1)+(1, .25)$), \p3=($(S1)-(0,5)$)
      in (\p1) -- (\x1, \y2)
      to [out=90, in=90, looseness=2] (\p2)
      -- (\x2, \y3) to [out=-90, in=-90, looseness=2] (\p3) -- (S1);
      \node at (M1) {$\bullet$};
      \draw (M1) -- ++(0, -1) node (B2) [below, draw, rectangle] {\makebox[1em]{$\lambda$}};
      \draw let \p1 = (S2), \p2 = (B2.south)
      in (\p1) -- (\x1, \y2) to [out=-90, in=-90, looseness=1.5] (\p2);
    \end{tikzpicture}
    \ = \!\!
    \begin{tikzpicture}[baseline = 0pt, x = 10pt, y = 8pt]
      \coordinate (S1) at (-3, 3); \node at (S1) [above] {$H^{\vee\vee}$};
      \coordinate (S2) at (2, 3); \node at (S2) [above] {$I^{\vee}$};
      \coordinate (S3) at (4, 3); \node at (S3) [above] {$H$};
      \node (B1) at (0,0) [draw, circle, fill = gray] {};
      \draw let \p1 = (B1.north), \p2 = ($(\p1)+(-1.5,2)$)
      in (\p1) to [out=90, in=-90, looseness=1] (\p2)
      to [out=90, in=90, looseness=2] (\x1, \y2) coordinate (TMP);
      \draw [cross] (TMP) to [out=-90, in=90] ++(-1.5, -2) coordinate (TMP);
      \draw let \p1 = (TMP), \p2 = (S1)
      in (\p1) to [out=-90, in=-90, looseness=2] (\x2, \y1) -- (\p2);
      \draw (S2) to [out=-90, in=90] ++(-2, -8)
      to [out=-90, in=-90, looseness=2] ++(2,0)
      node (B2) [above, draw, rectangle] {\makebox[1em]{$\lambda$}};
      \coordinate (M1) at ($(B2.north)+(0,1)$);
      \draw [cross] (B1.south) to [out=-90, in=180] (M1);
      \draw (B2.north) -- ++(0,1) coordinate (M1);
      \node at (M1) {$\bullet$};
      \draw (M1) to [out=0, in=-90] (S3);
    \end{tikzpicture}
    \quad
    \left(
      \mathcal{N}^{-1} =
      \begin{array}{c}
        \begin{tikzpicture}[baseline = 0pt, x = 8pt, y = 8pt]
          \node (B1) at (0,0) [above, draw, circle, fill = gray] {};
          \draw (B1.north) -- ++(0,+.5);
          \draw (B1.south) -- ++(0,-.5);
        \end{tikzpicture}
      \end{array}
    \right)
  \end{equation*}
  Equation \eqref{eq:gFb-algebra-Nakayama} can actually be used as the definition of $\nu$. By comparing the above result with the right hand side \eqref{eq:gFb-algebra-Nakayama}, we obtain
  \begin{equation*}
    \nu = (\sigma_{I^{\vee}, H}^{-1} \otimes \id_{I^{\vee\vee}}) \circ (\mathcal{N}^{-1} \psi_H^{-1} \otimes \coev_{I^{\vee}}).
  \end{equation*}
  The desired formula is obtained by rewriting the right hand side by \eqref{eq:braided-Hopf-monodromy}.
\end{proof}

By the above lemma, we have
\begin{equation*}
  (\id_{I^{\vee}} \otimes \varepsilon \otimes \id_{I^{\vee\vee}}) \circ \nu
  = \coev_{I^{\vee}} \circ \alpha_H^{\vee\vee}.
\end{equation*}
Thus, by Theorem~\ref{thm:Nakayama-for-modules-2}, we immediately have:

\begin{theorem}
  The modular object of $\mathcal{C} := {}_H \mathcal{B}$ is given by
  \begin{equation*}
    \alpha_{\mathcal{C}} = (K,
    \ \alpha_H \otimes \id_K : H \otimes K \to K)
    \quad (K := I^{\vee} \otimes \alpha_{\mathcal{B}}).
  \end{equation*}
\end{theorem}

This formula has been obtained in \cite{MR3569179} in a different method (though, strictly speaking, the category of right $H$-modules is considered in \cite{MR3569179}). Since ${}_H \mathcal{B}$ is a finite multi-tensor category, the above formula can also be read as a formula of the Nakayama functor of ${}_H \mathcal{B}$.

\section{The center of a finite bimodule category}
\label{sec:the-center}

\subsection{The canonical algebra and its variants}
\label{subsec:canonical-algebra}

In this section, we assume that the base field $\bfk$ is perfect. Thanks to this, the class of finite multi-tensor categories over $\bfk$ is closed under the Deligne tensor product \cite{MR1106898,MR3242743}.

Let $\mathcal{C}$ be a finite multi-tensor category.
A finite $\mathcal{C}$-bimodule category $\mathcal{M}$ can be regarded as a finite left module category over $\mathcal{C}^{\env} := \mathcal{C} \boxtimes \mathcal{C}^{\rev}$ by the action $\catacte$ determined by $(X \boxtimes Y) \catacte M = X \catactl M \catactr Y$ for $X, Y \in \mathcal{C}$ and $M \in \mathcal{M}$.
Let $\iHom : \mathcal{C}^{\op} \times \mathcal{C} \to \mathcal{C}^{\env}$ be the internal Hom functor for the left $\mathcal{C}^{\env}$-module category $\mathcal{C}$ \cite{MR3242743}.
Then $\canalg := \iHom(\unitobj, \unitobj)$ is an algebra in $\mathcal{C}^{\env}$ and the functor
\begin{equation}
  \label{eq:fund-thm-Hopf-bimod}
  \mathds{K}: \mathcal{C} \to (\mathcal{C}^{\env})_{\canalg},
  \quad V \mapsto (V \boxtimes \unitobj) \otimes \canalg
\end{equation}
is an equivalence of left $\mathcal{C}^{\env}$-module categories \cite{MR2097289}.

The algebra $\canalg$ is called the {\em canonical algebra} \cite{MR3242743}. It is known that $\canalg$ is the coend $\canalg = \int^{X \in \mathcal{C}} X \boxtimes {}^{\vee\!}X$ as an object of $\mathcal{C}^{\env}$.
In terms of the universal dinatural transformation
$i_X : X \boxtimes {}^{\vee\!}X \to \canalg$ ($X \in \mathcal{C}$), the multiplication $m : \canalg \otimes \canalg \to \canalg$ and the unit $u : \unitobj \boxtimes \unitobj \to \canalg$ of $\canalg$ are given by $m \circ (i_{X} \otimes i_Y) = i_{X \otimes Y}$ and $u = i_{\unitobj}$ for $X, Y \in \mathcal{C}$, respectively \cite{MR3632104}.

By Lemma~\ref{lem:adjoints-ends} and the succeeding remark, there is a unique isomorphism
\begin{equation*}
  \tau_Y : (\unitobj \boxtimes Y) \otimes \canalg \to (Y \boxtimes \unitobj) \otimes \canalg
\end{equation*}
in $\mathcal{C}^{\env}$ such that the diagram
\begin{equation*}
  \begin{tikzcd}[column sep = 64pt]
    (\unitobj \boxtimes Y) \otimes \canalg
    \arrow[d, "{\tau_Y}"']
    & (\unitobj \boxtimes Y) \otimes (X \boxtimes {}^{\vee\!}X)
    \arrow[d, "{\id_{\unitobj \boxtimes Y} \otimes ((\coev_Y \otimes \id_X) \boxtimes \id_{{}^{\vee\!}X})}"]
    \arrow[l, "{\id \otimes i_X}"'] \\
    (Y \boxtimes \unitobj) \otimes \canalg
    & (Y \boxtimes \unitobj) \otimes ((Y^{\vee} \otimes X) \boxtimes {}^{\vee}(Y^{\vee} \otimes X))
    \arrow[l, "{\id \otimes i_{Y^{\vee} \otimes X}}"']
  \end{tikzcd}
\end{equation*}
commutes for all objects $X \in \mathcal{C}$. According to \cite[Appendix A]{2017arXiv170709691S}, the structure morphism of the $\mathcal{C}^{\env}$-module functor $\mathds{K}$, which we denote by
\begin{equation*}
  \widetilde{\tau}_{M, V} : M \otimes \mathds{K}(V)
  \to \mathds{K}(M \catacte V)
  \quad (M \in \mathcal{C}^{\env}, V \in \mathcal{C}),
\end{equation*}
is given by
\begin{align*}
  (X \boxtimes Y) \otimes \mathds{K}(V) = \mbox{}
  & ((X \otimes V) \boxtimes \unitobj) \otimes (\unitobj \boxtimes Y) \otimes \canalg \\
  \xrightarrow{\quad \id \otimes \tau_Y \quad} \mbox{}
  & ((X \otimes V) \boxtimes \unitobj) \otimes (Y \boxtimes \unitobj) \otimes \canalg
    = \mathds{K}((X \boxtimes Y) \catacte V)
\end{align*}
if $M = X \boxtimes Y$ for some $X, Y \in \mathcal{C}$.
For later use, we note:

\begin{lemma}
  \label{lem:m-and-tau}
  The following equation holds:
  \begin{equation*}
    m \circ (i_{Y^{\vee}} \otimes \id_{\canalg})
    = ((\eval_{Y} \boxtimes \unitobj) \otimes \id_{\canalg})
    \circ (\id_{Y^{\vee} \boxtimes \unitobj} \otimes \tau_{Y})
    \quad (Y \in \mathcal{C}).
  \end{equation*}
\end{lemma}
\begin{proof}
  By the definition of $m$ and $\tau_Y$, we have 
  \begin{gather*}
    ((\eval_{Y} \boxtimes \unitobj) \otimes \id_{\canalg}) \circ (\id_{Y^{\vee} \boxtimes \unitobj} \otimes \tau_{Y}) \circ (\id_{Y^{\vee} \boxtimes Y} \otimes i_X) \\
    = i_{Y^{\vee} \otimes X}
    = m \circ (i_{Y^{\vee}} \otimes \id_{\canalg}) \circ (\id_{Y^{\vee} \boxtimes Y} \otimes i_X)
  \end{gather*}
  for $X, Y \in \mathcal{C}$. The claim now follows from the universal property.
\end{proof}

In this section, the following variants of the canonical algebra are useful:

\begin{definition}
  \label{def:cano-alg-twisted}
  For integers $p$ and $q$, we define
  \begin{equation*}
    \canalg_{2p,2q} := \int^{X \in \mathcal{C}} S^{2p}(X) \boxtimes S^{2q-1}(X),
  \end{equation*}
  where $S: \mathcal{C} \to \mathcal{C}$ is the left duality functor.
\end{definition}

The coend $\canalg_{2p,2q}$ is an algebra in $\mathcal{C}^{\env}$ as the image of the canonical algebra $\canalg$ under the tensor autoequivalence $S^{2p} \boxtimes S^{2q}$ of $\mathcal{C}^{\env}$.
Since $S^{2} \boxtimes S^{-2}$ is the double left dual functor of $\mathcal{C}^{\env}$, we have $\canalg^{\!\vee\vee} = \canalg_{2,-2}$ as algebras in $\mathcal{C}^{\env}$.

\subsection{The canonical algebra and the twisted center}

Let $\mathcal{C}$ be a finite multi-tensor category, and let $\mathcal{M}$ be a finite $\mathcal{C}$-bimodule category. The {\em center} of $\mathcal{M}$, denoted by $\mathcal{Z}(\mathcal{M})$, is the category defined as follows: An object of this category is a pair $(M, \sigma)$ consisting of an object $M \in \mathcal{M}$ and a natural transformation $\sigma(X) : X \catactl M \to M \catactr X$ ($X \in \mathcal{C}$) satisfying
\begin{equation*}
  \sigma(\unitobj) = \id_M \quad \text{and} \quad
  \sigma(X \otimes Y) = (\sigma(X) \catactr \id_Y) \circ (\id_X \catactl \sigma(Y))
\end{equation*}
for all objects $X, Y \in \mathcal{C}$. Given two objects $\mathbf{M} = (M, \sigma)$ and $\mathbf{N} = (N, \tau)$ of $\mathcal{Z}(\mathcal{M})$, a morphism from $\mathbf{M}$ to $\mathbf{N}$ is a morphism $f: M \to N$ in $\mathcal{M}$ satisfying
\begin{equation*}
  \tau(X) \circ (\id_X \catactl f) = (f \catactr \id_X) \circ \sigma(X)
\end{equation*}
for all objects $X \in \mathcal{C}$.
Lemma~\ref{lem:twisted-center-lemma-1} below is well-known for $\mathcal{M} = \mathcal{C}$.
The general case can be proved in the same way.

\begin{lemma}
  \label{lem:twisted-center-lemma-1}
  Let $(M, \sigma)$ be an object of $\mathcal{Z}(\mathcal{M})$.
  Then, for all objects $X \in \mathcal{C}$, the morphism $\sigma(X)$ is an isomorphism with the inverse
  \begin{equation*}
    \sigma(X)^{-1} = (\id_X \catactl \id_M \catactr \eval_X) \circ (\id_X \catactl \sigma(X^{\vee}) \catactr \id_X) \circ (\coev_X \catactl \id_{M} \catactr \id_X).
  \end{equation*}
\end{lemma}

Given a tensor autoequivalence $F$ of $\mathcal{C}$ and a left $\mathcal{C}$-module category $\mathcal{N}$, we denote by ${}_{\langle F \rangle} \mathcal{N}$ the left $\mathcal{C}$-module category obtained from $\mathcal{N}$ by twisting the action of $\mathcal{C}$ by $F$. A similar notation will be used for right module categories and bimodule categories. For two integers $p$ and $q$, we define the {\em twisted center} of $\mathcal{M}$ by
\begin{equation*}
  \mathcal{Z}_{2p, 2q}(\mathcal{M}) := \mathcal{Z}({}_{\langle S^{2p} \rangle}\mathcal{M}_{\langle S^{2q} \rangle}),
\end{equation*}
where $S : \mathcal{C} \to \mathcal{C}$ is the left duality functor $X \mapsto X^{\vee}$.
Now we fix two integers $p$ and $q$ and let $\canalg_{2p,2q} \in \mathcal{C}^{\env}$ be the algebra of Definition~\ref{def:cano-alg-twisted}.
According to \cite[Section 3]{2017arXiv170709691S},

\begin{lemma}
  \label{lem:twisted-center-as-modules}
  The category of $\canalg_{2p,2q}$-modules in $\mathcal{M}$ is isomorphic to $\mathcal{Z}_{2p,2q}(\mathcal{M})$.
\end{lemma}

For reader's convenience, we include how the category isomorphism of this lemma is established. We first note that the action of $\mathcal{C}^{\env}$ on $\mathcal{M}$ preserves coends as it is right exact. Thus we have isomorphisms
\begin{align*}
  \Hom_{\mathcal{C}^{\env}}(\canalg_{2p,2q} \catacte M, M)
  & \textstyle \cong \Hom_{\mathcal{C}^{\env}}(\int^{X \in \mathcal{C}} S^{2p}(X) \catactl M \catactr S^{2q-1}(X), M) \\
  & \textstyle \cong \int_{X \in \mathcal{C}} \Hom_{\mathcal{C}^{\env}}(S^{2p}(X) \catactl M \catactr S^{2q-1}(X), M) \\
  & \textstyle \cong \int_{X \in \mathcal{C}} \Hom_{\mathcal{C}^{\env}}(S^{2p}(X) \catactl M, M \catactr S^{2q}(X)) \\
  & \cong \Nat(S^{2p}(-) \catactl M, M \catactr S^{2q}(-))
\end{align*}
for $M \in \mathcal{M}$. The category isomorphism is obtained by showing that a morphism $\canalg_{2p,2q} \catacte M \to M$ in $\mathcal{M}$ makes $M$ an $\canalg_{2p,2q}$-module in $\mathcal{M}$ if and only if the corresponding natural transformation makes $M$ an object of $\mathcal{Z}_{2p,2q}(\mathcal{M})$.

\subsection{The canonical algebra and the Radford isomorphism}
\label{subsec:cano-alg-Radford-iso}

Let $\mathcal{C}$ be a finite multi-tensor category, and let $\alpha := \Nak_{\mathcal{C}}(\unitobj)$.
Since the equivalence $\mathds{K} : \mathcal{C} \to (\mathcal{C}^{\env})_{\canalg}$ of \eqref{eq:fund-thm-Hopf-bimod} is a left $\mathcal{C}^{\env}$-module functor, it induces an equivalence
\begin{equation}
  \label{eq:fund-thm-Hopf-bimod-2}
  \mathds{K} : {}_{\canalg^{\!\vee\vee}}\mathcal{C} \to {}_{\canalg^{\!\vee\vee}}(\mathcal{C}^{\env})_{\canalg}
\end{equation}
between the categories of left $\canalg^{\!\vee\vee}$-modules.
The source of \eqref{eq:fund-thm-Hopf-bimod-2} can be identified with $\mathcal{Z}_{2,-2}(\mathcal{C})$ by Lemma~\ref{lem:twisted-center-as-modules} and the succeeding remark.
We note that $\canalg^{\!\vee}$ belongs to the target of \eqref{eq:fund-thm-Hopf-bimod-2}.

\begin{definition}
  \label{def:modular-obj-ENO}
  The object $\mathbf{D}_{\mathcal{C}} \in \mathcal{Z}_{2,-2}(\mathcal{C})$ is defined as an object corresponding to the $\canalg^{\!\vee\vee}$-$\canalg$-bimodule $\canalg^{\!\vee}$ via the equivalence \eqref{eq:fund-thm-Hopf-bimod-2}.
\end{definition}

We write $\mathbf{D}_{\mathcal{C}} = (D, \Radford')$.
The {\em distinguished invertible object} of $\mathcal{C}$ \cite{MR2097289} is defined as an object of $\mathcal{C}$ corresponding to the right $\canalg$-module $\canalg^{\!\vee}$ via \eqref{eq:fund-thm-Hopf-bimod}.
Thus $D$ in the above is the distinguished invertible object of $\mathcal{C}$.
It is explained in \cite[Section 4]{2017arXiv170709691S} that $D = \overline{\Nak}_{\mathcal{C}}(\unitobj)$, where $\overline{\Nak}_{\mathcal{C}}$ is the right adjoint of $\Nak_{\mathcal{C}}$ given in \S\ref{lem:Nakayama-adj-diagram}. Furthermore, the natural isomorphism $\Radford'$ is given by
\begin{align*}
  \Radford'_X = ( X^{\vee\vee} \otimes D
  & \xrightarrow{\quad \cong \quad}
    \textstyle \int_{Y \in \mathcal{C}} \Hom_{\mathcal{C}}(Y, \unitobj) \copow (X^{\vee\vee} \otimes Y) \\
  & \xrightarrow{\quad \cong \quad}
    \textstyle \int_{Y \in \mathcal{C}} \Hom_{\mathcal{C}}(X^{\vee} \otimes Y, \unitobj) \copow Y \\
  & \xrightarrow{\quad \cong \quad}
    \textstyle \int_{Y \in \mathcal{C}} \Hom_{\mathcal{C}}(Y, X) \copow Y \\
  & \xrightarrow{\quad \cong \quad}
    \textstyle \int_{Y \in \mathcal{C}} \Hom_{\mathcal{C}}(Y \otimes {}^{\vee\!}X, \unitobj) \copow Y \\
  & \xrightarrow{\quad \cong \quad}
    \textstyle \int_{Y \in \mathcal{C}} \Hom_{\mathcal{C}}(Y, \unitobj) \copow (Y \otimes {}^{\vee\vee\!}X)
    \xrightarrow{\quad \cong \quad} D \otimes {}^{\vee\vee\!}X)
\end{align*}
for $X \in \mathcal{C}$, where Lemma \ref{lem:adjoints-ends} is used at the second and the fourth isomorphisms.

We have introduced the Radford isomorphism in \S\ref{subsec:radford-iso} in terms of the Nakayama functor. There is the following relation between the isomorphisms $\Radford$ and $\Radford'$.
The left dual of $\alpha$ is also an end of the same form as $\overline{\Nak}_{\mathcal{C}}(\unitobj)$. Hence we can identify $\alpha^{\vee}$ with $D$ by the universal property. Under this identification, it is straightforward to see that the equation
\begin{equation}
  \label{eq:Radford!-iso}
  \Radford'_X = (\Radford_{{}^{\vee\!}X}^{})^{\vee} : X^{\vee\vee} \otimes \alpha^{\vee} \to \alpha^{\vee} \otimes {}^{\vee\vee\!}X
  \quad (X \in \mathcal{C})
\end{equation}
holds by the construction of the Radford isomorphism.

\subsection{Nakayama isomorphism of the canonical algebra}

Let $\canalg \in \mathcal{C}^{\env}$ be the canonical algebra.
By the definition of $\mathbf{D}_{\mathcal{C}} \in \mathcal{Z}_{2,-2}(\mathcal{C})$, there is an isomorphism $\phi : \mathds{K}(\mathbf{D}_{\mathcal{C}}) \to \canalg^{\!\vee}$ of $\canalg^{\!\vee\vee}$-$\canalg$-modules in $\mathcal{C}^{\env}$.
As a right $\canalg$-module, the source of $\phi$ is $(\alpha \boxtimes \unitobj)^{\vee} \otimes \canalg$. Thus there is an $(\alpha \boxtimes \unitobj)$-valued Frobenius trace $\lambda$ on $\canalg$ inducing the isomorphism $\phi$.
The Nakayama isomorphism
\begin{equation*}
  \nu = (\phi^{-1} \otimes (\alpha^{\vee\vee} \boxtimes \unitobj)) \circ \phi^{\vee}:
  \canalg^{\!\vee\vee} \to \canalg^{\!\alpha} := (\alpha^{\vee} \boxtimes \unitobj) \otimes \canalg \otimes (\alpha^{\vee\vee} \boxtimes \unitobj)
\end{equation*}
associated to $\lambda$ is given as follows:

\begin{lemma}
  \label{lem:cano-alg-Nakayama}
  For all objects $X \in \mathcal{C}$, we have
  \begin{equation*}
    \begin{aligned}
      \nu \circ i_X^{\vee\vee}
      & = (\id_{\alpha^{\vee} \boxtimes \unitobj} \otimes i_{{}^{\vee\vee\!}X} \otimes \id_{\alpha^{\vee\vee} \boxtimes \unitobj}) \\
      & \qquad \qquad \circ ((\Radford'_X \otimes \id_{\alpha^{\vee\vee}})
      \circ (\id_{X^{\vee\vee}} \otimes \coev_{\alpha^{\vee}}))
      \boxtimes \id_{{}^{\vee\vee\vee\!}X}),
    \end{aligned}
  \end{equation*}
  where $\Radford'_X$ is given by \eqref{eq:Radford!-iso}.
\end{lemma}
\begin{proof}
  We first introduce some morphisms.
  Let $\rho : \canalg^{\!\vee\vee} \catacte \alpha^{\vee} \to \alpha^{\vee}$ denote the left action of $\canalg^{\!\vee\vee}$ on $\alpha^{\vee}$ corresponding to $\Radford'$. Thus we have
  \begin{equation*}
    \rho \circ (i_X^{\vee\vee} \catacte \alpha^{\vee})
    = (\id_{\alpha^{\vee}} \otimes \eval_{{}^{\vee\vee\!}X}) \circ (\Radford'_X \otimes \id_{{}^{\vee\vee\vee\!}X})
    \quad (X \in \mathcal{C}).
  \end{equation*}
  The left action of $\canalg^{\!\vee\vee}$ on $\mathds{K}(\alpha^{\vee})$ is $\rho_1 := \mathds{K}(\rho) \circ \tilde{\tau}_{\canalg, \alpha^{\vee}}$, where $\tilde{\tau}$ is the left $\mathcal{C}^{\env}$-module structure of $\mathds{K}$.
  The object $\mathds{K}(\alpha^{\vee})$ is also a left $\canalg^{\alpha}$-module by the action
  \begin{equation*}
    \rho_2 := \id_{\alpha^{\vee} \boxtimes \unitobj} \otimes
    (m \circ (\id_{\canalg} \otimes \eval_{\alpha^{\vee\vee} \boxtimes \unitobj} \otimes \id_{\canalg})) : \canalg^{\alpha} \otimes \mathds{K}(\alpha^{\vee}) \to \mathds{K}(\alpha^{\vee}).
  \end{equation*}
  For $X \in \mathcal{C}$, we set
  $\nu_X' = (\Radford'_X \otimes \id_{\alpha^{\vee\vee}})
  \circ (\id_{X^{\vee\vee}} \otimes \coev_{\alpha^{\vee}})$.
  By the universal property, there is a unique morphism $\nu' : \canalg^{\!\vee\vee} \to \canalg^{\alpha}$ such that the equation
  \begin{equation*}
    \nu' \circ i_X^{\vee\vee} = (\id_{\alpha^{\vee} \boxtimes \unitobj} \otimes i_{{}^{\vee\vee\!}X} \otimes \id_{\alpha^{\vee\vee} \boxtimes \unitobj}) \circ (\nu_X' \boxtimes \id_{{}^{\vee\vee\vee\!}X}).
  \end{equation*}
  This lemma claims that the equation $\nu = \nu'$ holds. Thus, in view of Lemma~\ref{lem:phi-bimodule-morphism} and the succeeding remark, it suffices to show that the following equation holds:
  \begin{equation*}
    \rho_1 = \rho_2 \circ (\nu' \otimes \id_{\mathds{K}(\alpha^{\vee})})
  \end{equation*}
  For $X \in \mathcal{C}$, we compute $\rho_1 \circ (i_X^{\vee\vee} \otimes \id_{\mathds{K}(\alpha^{\vee})})$
  \begin{align*}
    & = \mathds{K}(\rho) \circ \tilde{\tau}_{\canalg, \alpha^{\vee}}
      \circ (i_X \otimes \id_{\mathds{K}(\alpha^{\vee})}) \\
    & = \mathds{K}(\rho \circ (i_X \catacte \alpha^{\vee})) \circ
      \tilde{\tau}_{X^{\vee\vee} \boxtimes {}^{\vee\vee\vee\!}X, \alpha^{\vee}} \\
    & = \mathds{K}((\id_{\alpha^{\vee}} \otimes \eval_{{}^{\vee\vee\!}X}) \circ (\Radford'_X \otimes \id_{{}^{\vee\vee\vee\!}X}))
      \circ \tilde{\tau}_{X^{\vee\vee} \boxtimes {}^{\vee\vee\vee\!}X, \alpha^{\vee}} \\
    & = \mathds{K}((\id_{\alpha^{\vee}} \otimes \eval_{{}^{\vee\vee\!}X}) \circ (\Radford'_X \otimes \id_{{}^{\vee\vee\vee\!}X}))
      \circ (\id_{(X^{\vee\vee} \otimes \alpha^{\vee}) \boxtimes \unitobj} \otimes \tau_{{}^{\vee\vee\vee\!}X}),
  \end{align*}
  where the second equality follows from the naturality of $\tilde{\tau}$.
  We also have
  \begin{align*}
    & \rho_2 \circ
    (\nu' \otimes \id_{\mathds{K}(\alpha^{\vee})})
      \circ (i_{X}^{\vee\vee} \otimes \id_{\mathds{K}(\alpha^{\vee})}) \\
    & = \rho_2 \circ
      (\id_{\alpha^{\vee} \boxtimes \unitobj} \otimes i_{{}^{\vee\vee\!}X} \otimes \id_{\alpha^{\vee\vee} \boxtimes \unitobj} \otimes \id_{\mathds{K}(\alpha^{\vee})})
      \circ ((\nu_X' \boxtimes \id_{{}^{\vee\vee\vee\!}X}) \otimes \id_{\mathds{K}(\alpha^{\vee})}) \\
    & = (\id_{\alpha^{\vee} \boxtimes \unitobj} \otimes
      (m \circ (i_{{}^{\vee\vee\!}X} \otimes \eval_{\alpha^{\vee\vee} \boxtimes \unitobj} \otimes \id_{\canalg}))
      \circ ((\nu_X' \boxtimes \id_{{}^{\vee\vee\vee\!}X}) \otimes \id_{\mathds{K}(\alpha^{\vee})}) \\
    & = (\id_{\alpha^{\vee} \boxtimes \unitobj} \otimes (m \circ (i_{{}^{\vee\vee\!}X} \otimes \id_{\canalg}))
      \circ ((\Radford_X' \boxtimes \id_{{}^{\vee\vee\vee\!}X}) \otimes \id_{\canalg}) \\
    & = \begin{aligned}[t]
      & ((\id_{\alpha^{\vee}} \otimes \eval_{{}^{\vee\vee\!}X}) \boxtimes \unitobj) \otimes \id_{\canalg})
      \circ (\id_{(\alpha^{\vee} \otimes {}^{\vee\vee\!}X) \boxtimes \unitobj} \otimes \tau_{{}^{\vee\vee\vee\!}X}) \\
      & \qquad \qquad \circ ((\Radford_X' \boxtimes \id_{{}^{\vee\vee\vee\!}X}) \otimes \id_{\canalg})
      \qquad \text{(by Lemma \ref{lem:m-and-tau})}
    \end{aligned} \\
    & = \mathds{K}((\id_{\alpha^{\vee}} \otimes \eval_{{}^{\vee\vee\!}X}) \circ (\Radford_X' \otimes \id_{{}^{\vee\vee\vee\!}X}))
      \circ (\id_{(X^{\vee\vee} \otimes \alpha^{\vee}) \boxtimes \unitobj} \otimes \tau_{{}^{\vee\vee\vee\!}X}).
  \end{align*}
  The proof is done.
\end{proof}

\subsection{Nakayama functor of the center}

Let $\mathcal{C}$ be a finite multi-tensor category, and let $\mathcal{M}$ be a finite $\mathcal{C}$-bimodule category.
We fix integers $p$ and $q$.
Given objects $\mathbf{V} = (V, \sigma_V) \in \mathcal{Z}_{2p,2q}(\mathcal{C})$ and $\mathbf{M} = (M, \sigma_M) \in \mathcal{Z}(\mathcal{M})$, we define
\begin{equation*}
  \mathbf{V} \catactl \mathbf{M} := (V \catactl M, \sigma_{V \catactl M}) \in \mathcal{Z}_{2p,2q}(\mathcal{M}),
\end{equation*}
where $\sigma_{V \catactl M}(X)$ for $X \in \mathcal{C}$ is given by the composition
\begin{equation*}
  \newcommand{\xarr}[1]{\xrightarrow{\makebox[9em]{\scriptsize $#1$}}}
  \begin{aligned}
    S^{2p}(X) \catactl (V \catactl M)
    & \xarr{\sigma_V(X) \catactl \id_M}
    V \catactl S^{2q}(X) \catactl M \\
    & \xarr{\id_V \catactl \sigma_M(S^{2q}(X))}
    V \catactl M \catactr S^{2q}(X).
  \end{aligned}
\end{equation*}
This construction gives rise to the functor
\begin{equation*}
  \catactl: \mathcal{Z}_{2p,2q}(\mathcal{C}) \times \mathcal{Z}(\mathcal{M}) \to \mathcal{Z}_{2p,2q}(\mathcal{M}).
\end{equation*}
In particular, the object $\mathbf{D}_{\mathcal{C}}$ of Definition~\ref{def:modular-obj-ENO} yields the functor
\begin{equation}
  \label{eq:tw-center-functor-D}
  \mathbf{D}_{\mathcal{C}} \catactl (-) : \mathcal{Z}(\mathcal{M}) \to \mathcal{Z}_{2,-2}(\mathcal{M}).
\end{equation}

The Nakayama functor $\Nak_{\mathcal{M}} : \mathcal{M} \to \mathcal{M}$ becomes a $\mathcal{C}$-bimodule functor from $\mathcal{M}$ to ${}_{\langle S^{-2} \rangle}\mathcal{M}_{\langle S^2 \rangle}$ together with the natural isomorphisms~\eqref{eq:cat-action-left-Nakayama} and \eqref{eq:cat-action-right-Nakayama} (as noted in \cite{MR4042867} as the bimodule Radford $S^4$-formula). Hence we have a functor
\begin{equation}
  \label{eq:tw-center-induced-by-Nakayama}
  \mathcal{Z}_{2p,2q}(\mathcal{M}) \to \mathcal{Z}_{2p-2,2q+2}(\mathcal{M}),
  \quad (M, \sigma) \mapsto (\Nak_{\mathcal{M}}(M), \widetilde{\sigma})
\end{equation}
for integers $p$ and $q$, where $\widetilde{\sigma}$ is given by
\begin{equation*}
  \newcommand{\xarr}[1]{\xrightarrow{\makebox[5em]{\scriptsize $#1$}}}
  \begin{aligned}
    \widetilde{\sigma}(X) = (S^{2p-2}(X) \catactl \Nak_{\mathcal{M}}(M)
    & \xarr{\eqref{eq:cat-action-left-Nakayama}}
    \Nak_{\mathcal{M}}(S^{2p}(X) \catactl M) \\
    & \xarr{\Nak_{\mathcal{M}}(\sigma(X))}
    \Nak_{\mathcal{M}}(M \catactr S^{2q}(X)) \\
    & \xarr{\eqref{eq:cat-action-right-Nakayama}}
    \Nak_{\mathcal{M}}(M) \catactr S^{2q + 2}(X))
  \end{aligned}
\end{equation*}
for $X \in \mathcal{C}$.
Now we state and prove the following main result of this section:

\begin{theorem}
  \label{thm:Nakayama-center}
  With notation as above, the Nakayama functor of $\mathcal{Z}(\mathcal{M})$ is given by the composition
  \begin{equation*}
    \mathcal{Z}(\mathcal{M})
    \xrightarrow[\eqref{eq:tw-center-functor-D}]{\quad \mathbf{D}_{\mathcal{C}} \catactl (-) \quad}
    \mathcal{Z}_{2,-2}(\mathcal{M})
    \xrightarrow[\quad \eqref{eq:tw-center-induced-by-Nakayama} \quad]{}
    \mathcal{Z}(\mathcal{M}).
  \end{equation*}
\end{theorem}
\begin{proof}
  Let $\nu$ be the Nakayama isomorphism of $\canalg$ associated to the $(\alpha \boxtimes \unitobj)$-valued Frobenius trace mentioned in the previous subsection.
  If we identify $\mathcal{Z}_{2p,2q}(\mathcal{M})$ with the category of $\canalg_{2p,2q}$-modules in $\mathcal{M}$ by Lemma \ref{lem:twisted-center-as-modules}, then the Nakayama functor of $\mathcal{Z}(\mathcal{M})$ is given by the composition of \eqref{eq:tw-center-induced-by-Nakayama} and the functor
  \begin{equation*}
    \begin{gathered}
      {}_{\canalg}\mathcal{M} \to {}_{\canalg^{\vee\vee}}\mathcal{M},
      \quad (M, a_M) \mapsto (\alpha^{\vee} \catactl M, \widetilde{a}_M) \\
      (\widetilde{a}_M := (\id_{\alpha^{\vee} \boxtimes \unitobj} \catacte a_M) \circ (\id_{\alpha^{\vee} \boxtimes \unitobj} \catacte \id_{A} \catacte \eval_{\alpha^{\vee} \boxtimes \unitobj} \catacte \id_M) \circ (\nu \catacte \id_M)).
    \end{gathered}
  \end{equation*}
  By Lemma~\ref{lem:cano-alg-Nakayama}, this functor is identical to \eqref{eq:tw-center-functor-D}. Thus the Nakayama functor of $\mathcal{Z}(\mathcal{M})$ is given as stated. The proof is done.
\end{proof}

\subsection{Applications}
\label{subsec:applications}

Here we exhibit some corollaries of Theorem \ref{thm:Nakayama-center} and give some related remarks.
The theorem has some applications as many basic constructions in the theory of finite tensor categories are viewed as the center of a particular bimodule category.

\subsubsection{Frobenius property of the center}

Let $\mathcal{C}$ be a finite multi-tensor category, and let $\mathcal{M}$ be a finite bimodule category over $\mathcal{C}$. By the description of $\Nak_{\mathcal{Z}(\mathcal{M})}$ given by Theorem \ref{thm:Nakayama-center}, we have:

\begin{corollary}
  \label{cor:center-Frobenius}
  If $\mathcal{M}$ is Frobenius, then so is its center $\mathcal{Z}(\mathcal{M})$.
\end{corollary}

We note that the center of a bimodule category is not always Frobenius: Take a finite-dimensional algebra $A$ such that $\mathcal{M} := \lmod{A}$ is not Frobenius and regard $\mathcal{M}$ as a finite bimodule category over $\mathcal{C} := \Vect$. Then $\mathcal{Z}(\mathcal{M})$ is identified with $\mathcal{M}$, which is not Frobenius.

\subsubsection{The centralizer of a tensor functor}

Let $\mathcal{C}$ and $\mathcal{D}$ be finite multi-tensor categories, and let $F: \mathcal{C} \to \mathcal{D}$ be a tensor functor.
Then $\mathcal{D}$ is a finite $\mathcal{C}$-bimodule category by the action given by $X \catactl M \catactr Y = F(X) \otimes M \otimes F(Y)$ for $X, Y \in \mathcal{C}$ and $M \in \mathcal{D}$. The centralizer of $F$, which we denote by $\mathcal{Z}(F)$, is defined as the center of the finite $\mathcal{C}$-bimodule category $\mathcal{D}$.
The category $\mathcal{Z}(F)$ is a finite multi-tensor category by the monoidal structure inherited from $\mathcal{D}$.

There is a natural isomorphism $\xi_X : F(X^{\vee}) \to F(X)^{\vee}$ ($X \in \mathcal{C}$) defined by the uniqueness of a left dual object (the {\em duality transformation} \cite[Section 1]{MR2381536}).
For $X \in \mathcal{C}$, we set $\zeta_X = (\xi_X^{\vee})^{-1} \circ \xi_{X^{\vee}}$.
By the construction of the isomorphism \eqref{eq:cat-action-right-Nakayama}, we see that the diagram
\begin{equation*}
  \begin{tikzcd}[column sep = 20pt]
    \Nak_{\mathcal{D}}(M \catactr X)
    \arrow[d, equal]
    \arrow[r, "\eqref{eq:cat-action-right-Nakayama}"]
    & \Nak_{\mathcal{D}}(M) \catactr X^{\vee\vee}
    \arrow[r, equal]
    & \Nak_{\mathcal{D}}(M) \otimes F(X^{\vee\vee})
    \arrow[d, "{\id \otimes \zeta_X}"] \\
    \Nak_{\mathcal{D}}(M \otimes F(X))
    \arrow[rr, "\eqref{eq:cat-action-right-Nakayama}"]
    & & \Nak_{\mathcal{D}}(M) \otimes F(X)^{\vee\vee}
  \end{tikzcd}
\end{equation*}
commutes for all $M \in \mathcal{D}$ and $X \in \mathcal{C}$.
There is an analogous commutative diagram for the left action of $\mathcal{C}$ on $\mathcal{D}$. Theorem \ref{thm:Nakayama-center} now yields the following corollary:

\begin{corollary}
  For $F$ as above, the modular object of $\mathcal{Z}(F)$ is given by
  \begin{equation*}
    \alpha_{\mathcal{Z}(F)} = (F(\alpha_{\mathcal{C}}^{\vee}) \otimes \alpha_{\mathcal{D}}, \gamma),
  \end{equation*}
  where $\gamma$ is the natural isomorphism given by the composition
  \begin{align*}
    & F(X) \otimes F(\alpha_{\mathcal{C}}^{\vee}) \otimes \alpha_{\mathcal{D}} \\
    & \cong F(\alpha_{\mathcal{C}}^{\vee}) \otimes F(X^{\vee\vee\vee\vee}) \otimes \alpha_{\mathcal{D}}
    & & \text{\rm (by the Radford isomorphism in $\mathcal{C}$)} \\
    & \cong F(\alpha_{\mathcal{C}}^{\vee}) \otimes F(X)^{\vee\vee\vee\vee} \otimes \alpha_{\mathcal{D}}
    & & \text{\rm (by the duality transformation)} \\
    & \cong F(\alpha_{\mathcal{C}}^{\vee}) \otimes \alpha_{\mathcal{D}} \otimes F(X)
    & & \text{\rm (by the Radford isomorphism in $\mathcal{D}$)}
  \end{align*}
  for $X \in \mathcal{C}$.
\end{corollary}

Since the Drinfeld center of $\mathcal{C}$ is the centralizer of $\id_{\mathcal{C}}$, we obtain:

\begin{corollary}[{\it cf}. \cite{MR2097289}]
  The Drinfeld center of a finite multi-tensor category is unimodular.
\end{corollary}

\subsubsection{The category of module functors}

Let $\mathcal{C}$ be a finite multi-tensor category, and let $\mathcal{M}$ and $\mathcal{N}$ be finite left $\mathcal{C}$-module categories.

\begin{definition}
  We denote by $\REX_{\mathcal{C}}(\mathcal{M}, \mathcal{N})$ the category of linear right exact left $\mathcal{C}$-module functors from $\mathcal{M}$ to $\mathcal{N}$.
\end{definition}

Let $\act_{\mathcal{M}} : \mathcal{C} \to \REX(\mathcal{M})$ and $\act_{\mathcal{N}} : \mathcal{C} \to \REX(\mathcal{N})$ be the functors induced by the actions of $\mathcal{C}$ on $\mathcal{M}$ and $\mathcal{N}$, respectively.
The category $\mathcal{E} := \REX(\mathcal{M}, \mathcal{N})$ is a finite $\mathcal{C}$-bimodule category by the action given by
\begin{equation*}
  X \catactl F \catactr Y = \act_{\mathcal{N}}(X) \circ F \circ \act_{\mathcal{M}}(Y)
  \quad (X, Y \in \mathcal{C}, F \in \mathcal{E}),
\end{equation*}
and the category $\mathcal{F} := \REX_{\mathcal{C}}(\mathcal{M}, \mathcal{N})$ is precisely the center of $\mathcal{E}$.
Thus, by Lemma~\ref{lem:Nakayama-REX} and Theorem \ref{thm:Nakayama-center}, the Nakayama functor of $\mathcal{F}$ is given as follows:

\begin{corollary}
  \label{cor:Nakayama-rex-module-functors}
  For an object $\mathbf{F} = (F, \xi)$ of the above $\mathcal{F}$, we have
  \begin{equation*}
    \Nak_{\mathcal{F}}(\mathbf{F}) = (\Nak_{\mathcal{N}} \circ \act_{\mathcal{N}}(\alpha_{\mathcal{C}}^{\vee}) \circ F \circ \Nak_{\mathcal{M}}, \widetilde{\xi}),
  \end{equation*}
  where the natural isomorphism $\widetilde{\xi}$ is given by the composition
  \begin{align*}
    & \act_{\mathcal{N}}(X) \circ \Nak_{\mathcal{N}} \circ \act_{\mathcal{N}}(\alpha_{\mathcal{C}}^{\vee}) \circ F \circ \Nak_{\mathcal{M}} \\
    \cong \mbox{}
    & \Nak_{\mathcal{N}} \circ \act_{\mathcal{N}}(X^{\vee\vee}) \circ \act_{\mathcal{N}}(\alpha_{\mathcal{C}}^{\vee}) \circ F \circ \Nak_{\mathcal{M}}
    & & \text{\rm (use \eqref{eq:cat-action-left-Nakayama})} \\
    \cong \mbox{} & \Nak_{\mathcal{N}} \circ \act_{\mathcal{N}}(\alpha_{\mathcal{C}}^{\vee}) \circ \act_{\mathcal{N}}({}^{\vee\vee\!}X) \circ F \circ \Nak_{\mathcal{M}}
    & & \text{\rm (use \eqref{eq:Radford!-iso})} \\
    \cong \mbox{} & \Nak_{\mathcal{N}} \circ \act_{\mathcal{N}}(\alpha_{\mathcal{C}}^{\vee}) \circ F \circ \act_{\mathcal{M}}({}^{\vee\vee\!}X) \circ \Nak_{\mathcal{M}}
    & & \text{\rm (use the structure morphism $\xi$)} \\
    \cong \mbox{} & \Nak_{\mathcal{N}} \circ \act_{\mathcal{N}}(\alpha_{\mathcal{C}}^{\vee}) \circ F \circ \Nak_{\mathcal{M}} \circ \act_{\mathcal{M}}(X)
    & & \text{\rm (use \eqref{eq:cat-action-left-Nakayama})}.
  \end{align*}
  for $X \in \mathcal{C}$.
\end{corollary}

Given a finite left $\mathcal{C}$-module category $\mathcal{M}$, we set $\mathcal{C}_{\mathcal{M}}^* := \REX_{\mathcal{C}}(\mathcal{M}, \mathcal{M})$ and view it as a linear monoidal category by the composition of module functors (the order of the tensor product is reversed in \cite{MR3242743}, but this will not matter since we only mention its unimodularity). By the above corollary, we have:

\begin{corollary}
  $\alpha_{\mathcal{C}_{\mathcal{M}}^*}^{} = \Nak_{\mathcal{M}} \circ \act_{\mathcal{M}}(\alpha_{\mathcal{C}}^{\vee}) \circ \Nak_{\mathcal{M}}$.
\end{corollary}

Suppose that $\mathcal{M}$ is indecomposable and exact in the sense of \cite{MR3242743}. Then the linear monoidal category $\mathcal{C}_{\mathcal{M}}^*$ is in fact a finite tensor category called the {\em dual} of $\mathcal{C}$ with respect to $\mathcal{M}$.
The (dual of) the modular object of $\mathcal{C}_{\mathcal{M}}^*$ has been obtained in \cite{2022arXiv220707031F} with the use of relative Serre functors.
The formula of the above corollary coincides with that of \cite[Proposition 4.13]{2022arXiv220707031F} if we rewrite $\Nak_{\mathcal{M}}$ by the relative Serre functor of $\mathcal{M}$.

\section{Examples from Hopf algebra theory}
\label{sec:Hopf-examples}

\subsection{Radford isomorphism}

In this section, we explain how our results are applied to some categories appearing in the Hopf algebra theory.
Unless otherwise noted, the base field $\bfk$ is arbitrary in this section.
The unadorned tensor symbol $\otimes$ means the tensor product over $\bfk$.
Given a Hopf algebra $H$, we denote by $\Delta$, $\varepsilon$ and $S$ the comultiplication, the counit and the antipode of $H$, respectively.
We use the Sweedler notation, such as $\Delta(h) = h_{(1)} \otimes h_{(2)}$, to express the comultiplication.
Accordingly, the left coaction of $H$ is written as $m \mapsto m_{(-1)} \otimes m_{(0)}$ for an element $m$ of a left $H$-comodule.

Let $H$ be a finite-dimensional Hopf algebra, and let $\mathcal{C} := \lcomod{H}$ be the category of finite-dimensional left $H$-comodules.
As is well-known, $\mathcal{C}$ is a finite tensor category.
Given an object $X \in \mathcal{C}$ and a coalgebra automorphism $f$ of $H$, we denote by $X^{(f)}$ the left $H$-comodule obtained from $X$ by twisting the coaction by $f$.
We note that a left dual object of $X \in \mathcal{C}$ is the vector space $X^*$ equipped with the left coaction $f \mapsto f_{(-1)} \otimes f_{(0)}$ determined by
\begin{equation*}
  \langle f_{(-1)}, x \rangle f_{(0)} = \langle f, x_{(0)} \rangle S^{-1}(x_{(1)})
  \quad (x \in X, f \in X^*),
\end{equation*}
and thus $X^{\vee\vee}$ is identified with $X^{(S^{-2})}$ via the canonical isomorphism of vector spaces. Similarly, the double right dual ${}^{\vee\vee\!}X$ is identified with $X^{(S^2)}$.

We give a description of the Radford isomorphism of $\mathcal{C}$ (see \cite[Section 6]{2021arXiv211008739S} for the case of right comodules).
We denote by $G(H)$ the set of grouplike elements of $H$.
We fix a non-zero right cointegral $\lambda_H : H \to \bfk$. Then there is a unique element $g_H \in G(H)$, called the distinguished grouplike element, such that
\begin{equation}
  \label{eq:Hopf-Nakayama-def-comodulus}
  h_{(1)} \lambda_H(h_{(2)}) = \lambda_H(h) g_H^{}
  \quad (h \in H).
\end{equation}

Since $\lambda_H$ is non-degenerate, one can define $\nu_H : H \to H$ to be the unique linear map such that $\lambda_H(b a) = \lambda_H(\nu_H(a) b)$ for all $a, b \in H$ (namely, $\nu_H$ is the Nakayama automorphism of $H$ with respect to the Frobenius trace $\lambda_H$).
We fix a non-zero left integral $\Lambda^{\ell}$ in $H$ and define the left modular function $\alpha_H : H \to \bfk$ in the same way as \S\ref{subsec:braid-Hopf}. By definition, it is characterized by the equation
\begin{equation*}
  \Lambda^{\ell} \cdot h = \alpha_H(h) \Lambda^{\ell}
  \quad (h \in H).
\end{equation*}
Lemma~\ref{lem:braided-Nakayama} reduces to the well-known formula
\begin{equation}
  \label{eq:Hopf-Nakayama-auto}
  \nu_H(h) = S^{2}(h \leftharpoonup \alpha_H) \quad (h \in H),
\end{equation}
where $h \leftharpoonup f = \langle f, h_{(1)} \rangle h_{(2)}$ for $f \in H^*$ and $h \in H$. In particular, we have
\begin{equation}
  \label{eq:Hopf-Nakayama-and-modular}
  \alpha_H = \varepsilon \circ \nu_H.
\end{equation}

A right cointegral on $H$ is the same thing as a left cointegral on the Hopf algebra $H^{\cop}$ obtained from $H$ by reversing the order of the comultiplication. By equations \eqref{eq:Hopf-Nakayama-def-comodulus}--\eqref{eq:Hopf-Nakayama-and-modular}, we find that our $\alpha_H$ and $g_H$ are $\alpha_{H^{\cop}}$ and $g_{H^{\cop}}$, respectively, in the notation of \cite[Section 6]{2021arXiv211008739S}. By the discussion of \cite[Section 6]{2021arXiv211008739S} applied to $H^{\cop}$, we see that the modular object of $\mathcal{C}$ is $\bfk g_H$ and, under an appropriate identification of double duals, the Radford isomorphism of $\mathcal{C}$ is viewed as the following map:
\begin{equation}
  \label{eq:Radford-iso-H-comod}
  \begin{aligned}
    \Radford_X : \bfk g_H \otimes X^{(S^{-2})}
    & \to X^{(S^2)} \otimes \bfk g_H, \\
    \quad g_H \otimes x
    & \mapsto (x \leftharpoonup \alpha_H) \otimes g_H
    \quad (x \in X).
  \end{aligned}
\end{equation}

\begin{notation}
  The symbols $H$, $\mathcal{C}$, $g_H$ and $\alpha_H$ introduced in the above will be used till the end of \S\ref{subsec:relative-hopf}.
  We set $\overline{\alpha}_H = \alpha_H \circ S$ (this is the inverse of $\alpha_H$ with respect to the convolution product).
\end{notation}

\subsection{Cointegrals on comodule algebras}

An algebra in $\mathcal{C}$ is the same thing as a finite-dimensional left $H$-comodule algebra. According to \cite{2018arXiv181007114K}, we introduce the following terminology:

\begin{definition}
  \label{def:g-cointegral}
  Let $A$ be an algebra in $\mathcal{C}$, and let $g \in G(H)$. A {\em $g$-cointegral} on $A$ is a linear map $\lambda : A \to \bfk$ satisfying the equation $a_{(-1)} \langle \lambda, a_{(0)} \rangle = \lambda(a) g$ for all $a \in A$.
  We say that a $g$-cointegral on $A$ is {\em non-degenerate} if it is a Frobenius trace on $A$.
  A (non-degenerate) {\em grouplike-cointegral} is a linear map $A \to \bfk$ being a (non-degenerate) $g$-cointegral for some $g \in G(H)$.
\end{definition}

Some existence criteria for (non-degenerate) grouplike-cointegrals are given in \cite{2018arXiv181007114K} and \cite[\S4.10]{2019arXiv190400376S}.
Given an element $g \in G(H)$, we denote by $\bfk_g \in \mathcal{C}$ the vector space $\bfk$ equipped with the left $H$-coaction given by $1_{\bfk} \mapsto g \otimes 1_{\bfk}$.
With this notation, a $g$-cointegral is the same thing as a morphism $A \to \bfk_g$ in $\mathcal{C}$.
Thus a non-degenerate $g$-cointegral is nothing but a $\bfk_g$-valued Frobenius trace on $A$ in the sense of Definition~\ref{def:I-Frob-trace}.

Now let $A$ be an algebra in $\mathcal{C}$ and suppose that there are an element $g \in G(H)$ and a non-degenerate $g$-cointegral $\lambda : A \to \bfk$. We define the morphism
\begin{equation*}
  \nu : A^{\vee\vee} \to (\bfk_g)^{\vee} \otimes A \otimes (\bfk_g)^{\vee\vee}
\end{equation*}
in $\mathcal{C}$ by \eqref{eq:I-Frob-alg-Nakayama-iso}. Given an object $X \in \mathcal{C}$, we denote by $X^g$ the vector space $X$ equipped with the left $H$-coaction $x \mapsto g^{-1} x_{(-1)} g \otimes x_{(0)}$ ($x \in X$).
The source and the target of $\nu$ are identified with $A^{(S^{-2})}$ and $A^{g}$, respectively, and hence $\nu$ induces a morphism from $A^{(S^{-2})}$ and $A^{g}$, which, by abuse of notation, we denote by the same symbol $\nu : A^{(S^{-2})} \to A^{g}$.

Lemma \ref{lem:gFb-algebra-Nakayama} implies the equation $\lambda(b a) = \lambda(\nu(a) b)$ ($a, b \in A$). Namely, $\nu$ is the Nakayama automorphism of $A$ with respect to the Frobenius trace $\lambda$. A trivial but important point is that $\nu$ is a morphism in $\mathcal{C}$.

\begin{lemma}
  The Nakayama automorphism $\nu$ of $A$ with respect to the non-degenerate $g$-cointegral $\lambda$ satisfies the equations
  \begin{gather}
    \label{eq:H-comod-alg-Nakayama}
    \nu(a)_{(-1)} \otimes \nu(a)_{(0)}
    = g S^{-2}(a_{(-1)}) g^{-1} \otimes \nu(a_{(0)}), \\
    \label{eq:H-comod-alg-Nakayama-inverse}
    \nu^{-1}(a)_{(-1)} \otimes \nu^{-1}(a)_{(0)} =  g^{-1} S^2(a_{(-1)}) g \otimes a_{(0)}
  \end{gather}
  for all $a \in A$.
\end{lemma}
\begin{proof}
  Since $\nu : A^{(S^{-2})} \to A^g$ is a morphism in $\mathcal{C}$, we have
  \begin{equation*}
    \nu(a)_{(-1)} \otimes g^{-1} \nu(a)_{(0)} g
    = S^{-2}(a_{(-1)}) \otimes \nu(a_{(0)})
  \end{equation*}
  for all $a \in A$. The first equation follows from this. The second one is obtained by replacing $a$ with $\nu^{-1}(a)$ in the first one.
\end{proof}

A left coideal subalgebra of $H$ is a subalgebra of $H$ that is also a left coideal of $H$, that is, $\Delta(A) \subset H \otimes A$ holds. A left coideal subalgebra of $H$ is an algebra in $\mathcal{C}$ with respect to the restriction of the comultiplication of $H$.
The above proposition yields the following formula:

\begin{theorem}
  Let $A$ be a left coideal subalgebra of $H$ and suppose that there is an element $g \in G(H)$ and a non-zero $g$-cointegral $\lambda : A \to \bfk$. Then $\lambda$ is non-degenerate, and thus the Nakayama automorphism $\nu_A$ of $A$ with respect to $\lambda$ is defined. For every integer $k$, we have
  \begin{equation}
    \label{eq:coideal-sub-Nakayama}
    \nu_A^k(a) = g^k S^{-2k}(\varepsilon \nu_A^k \rightharpoonup a) g^{-k}
    \quad (a \in A),
  \end{equation}
  where $f \rightharpoonup a = a_{(-1)} \langle f, a_{(0)} \rangle$ for $f \in A^*$ and $a \in A$. The order of $\nu_A$ is finite.
\end{theorem}

The element $f \rightharpoonup a$ in the above is well-defined because $\Delta(A) \subset H \otimes A$, but we should be careful that it may be no more an element of $A$.

\begin{proof}
  The non-degeneracy of $\lambda$ follows from \cite[Theorem 4.11]{2018arXiv181007114K} (see also \cite[Proposition 4.22]{2019arXiv190400376S}).
  Equation \eqref{eq:coideal-sub-Nakayama} is trivial when $k = 0$.
  By \eqref{eq:H-comod-alg-Nakayama}, we have
  \begin{equation*}
    \beta \rightharpoonup \nu_A(a)
    = g_A^{} S^{-2}(a_{(-1)}) g_A^{-1} \langle \beta, \nu_A(a_{(0)}) \rangle
    = g_A^{} S^{-2}(\beta \nu_A \rightharpoonup a) g_A^{-1}
  \end{equation*}
  for all $a \in A$ and $\beta \in A^*$. By this equation and induction on $k$, we show that \eqref{eq:coideal-sub-Nakayama} holds for all positive integer $k$.
  For $k < 0$, use \eqref{eq:H-comod-alg-Nakayama-inverse} instead of \eqref{eq:H-comod-alg-Nakayama}.

  The finiteness of the order of $\nu_A$ is proved as follows:
  Since the set of algebra maps $A \to \bfk$ is at most finite, there is a positive integer $k$ such that $\varepsilon \nu_A^k = \varepsilon$. For this $k$ and a positive integer $\ell$, we have $\nu_A^{k\ell}(a) = g^{k\ell} S^{-2k\ell}(a) g^{-k\ell}$. Since both $g$ and $S$ are of finite order, we conclude that so is $\nu_A$.
\end{proof}

\begin{remark}
  Let $A$ be as above.
  As in the case of Hopf algebras, a left integral in $A$ is defined to be an element $\Lambda \in A$ such that $a \Lambda = \varepsilon(a)\Lambda$ for all $a \in A$.
  It is known that a coideal subalgebra of a finite-dimensional Hopf algebra is Frobenius \cite{MR2434166}.
  By the general argument for Frobenius algebras noted in \cite[Section 4]{1999math4164H}, the space of left integrals in $A$ is one-dimensional. We fix a non-zero left integral $\Lambda$ in $A$. Then the algebra map $\alpha_A := \varepsilon \circ \nu_A$ is characterized by the property that the equation $\Lambda a = \alpha_A(a) \Lambda$ holds for all $a \in A$.
\end{remark}

\begin{remark}
  The Hopf algebra $H$ itself is a left coideal subalgebra of $H$ and a right cointegral $\lambda : H \to \bfk$ is a $g_H$-cointegral in the sense of Definition~\ref{def:g-cointegral}. Thus we have
  \begin{equation}
    \label{eq:Hopf-Nakayama-auto-2}
    \nu_H(h) = g_H S^{-2}(\alpha_H \rightharpoonup h) g_H^{-1}
    \quad (h \in H)
  \end{equation}
  by the above proposition. Radford's $S^4$-formula
  \begin{equation}
    \label{eq:Radford-S4}
    S^4(h) = g_H^{}(\alpha_H \rightharpoonup h \leftharpoonup \overline{\alpha}_H) g_H^{-1}
    \quad (h \in H)
  \end{equation}
  is obtained by comparing two descriptions \eqref{eq:Hopf-Nakayama-auto} and \eqref{eq:Hopf-Nakayama-auto-2} of the Nakayama automorphism of $H$.
  This formula implies that $g_H^{}$ is central in the group $G(H)$.
\end{remark}

\subsection{Relative Hopf bimodules}
\label{subsec:relative-hopf}

Let $A$ and $B$ be algebras in $\mathcal{C}$.
An object of the category ${}_A \mathcal{C}_B$ is often called a {\em relative Hopf bimodule}.
The Nakayama functor of ${}_A \mathcal{C}_B$ is obtained by Theorem \ref{thm:Nakayama-for-modules}. Here we examine the case where both $A$ and $B$ admit non-degenerate grouplike-cointegrals.

\begin{remark}
  \label{rem:relative-Hopf-bimodules}
  Before we go into the detail, we note an interpretation of the category ${}_A \mathcal{C}_B$ from the viewpoint of the theory of finite tensor categories.
  According to \cite{MR2331768}, the category $\lmod{R}$ for an algebra $R \in \mathcal{C}$ has a natural structure of a finite left module category over $\mathcal{D} := \lmod{H}$. There is an equivalence
  \begin{equation*}
    {}_A\mathcal{C}_B \to \REX_{\mathcal{D}}(\lmod{B}, \lmod{A}),
    \quad M \mapsto M \otimes_B (-)
  \end{equation*}
  of categories. Furthermore, when $A = B$, this equivalence is monoidal (where the monoidal product of ${}_A \mathcal{C}_A$ is given by the tensor product over $A$). Thus ${}_A \mathcal{C}_A$ is equivalent to the dual $\mathcal{D}_{\lmod{A}}^*$ as a linear monoidal category.
  Below we compute the Nakayama functor of ${}_A \mathcal{C}_B$ and the modular object of ${}_A \mathcal{C}_A$ by viewing ${}_A \mathcal{C}_B$ as the category of modules over the monad $A \otimes (-) \otimes B$ on $\mathcal{C}$. The same results can of course be obtained from Corollary \ref{cor:Nakayama-rex-module-functors}.
\end{remark}

Now we assume that for each $X \in \{ A, B \}$, there are an element $g_X \in G(H)$ and a non-degenerate $g_X$-cointegral $\lambda_X$ on $X$.
For $X \in \{ A, B \}$, we denote by $\nu_{X}$ the Nakayama automorphism of $X$ associated to $\lambda_X$.
We introduce the endofunctor $\Nak$ on ${}_A \mathcal{C}_B$ as follows:
As a vector space, $\Nak(M) = M$ for $M \in {}_A \mathcal{C}_B$.
The left coaction of $H$ on $\Nak(M)$ is defined by
\begin{equation}
  m \mapsto g_A^{-1} S^2(m_{(-1)}) g_B^{-1} g_H^{} \otimes m_{(0)},
\end{equation}
where $m \mapsto m_{(-1)} \otimes m_{(0)}$ is the original coaction of $H$ on $M$.
The left action $\triangleright$ of $A$ and the right action $\triangleleft$ of $B$ on $\Nak(M)$ are defined by
\begin{equation}
  a \triangleright m \triangleleft b = \nu_A(a) \cdot m \cdot \nu_B^{-1}(b \leftharpoonup \alpha_H)
  \quad (a \in A, b \in B, m \in M),
\end{equation}
where `$\cdot$' means the original actions of $A$ and $B$ on $M$.

One can verify that $\Nak(M)$ is an object of ${}_A \mathcal{C}_B$ with the help of equations \eqref{eq:H-comod-alg-Nakayama}, \eqref{eq:Hopf-Nakayama-auto-2} and \eqref{eq:Radford-S4}. However, we omit the detail, since the reason why $\Nak(M)$ belongs to ${}_A \mathcal{C}_B$ will be revealed in the proof of the following theorem:

\begin{theorem}
  \label{thm:relative-Hopf-bimod-Nakayama}
  The functor $\Nak$ is the Nakayama functor of ${}_A \mathcal{C}_B$.
\end{theorem}
\begin{proof}
  We identify $X^{\vee\vee}$ and ${}^{\vee\vee\!}X$ with $X^{(S^{-2})}$ and $X^{(S^{2})}$, respectively, if $X$ is an object of $\mathcal{C}$.
  In view of \eqref{eq:FTC-Nakayama-formula}, we choose the functor
  \begin{equation*}
    \Nak_{\mathcal{C}}: \mathcal{C} \to \mathcal{C},
    \quad X \mapsto X^{(S^2)} \otimes \bfk g_H^{}
    \quad (X \in \mathcal{C})
  \end{equation*}
  as the Nakayama functor of $\mathcal{C}$.
  As discussed in Subsection~\ref{subsec:modules-module-cats}, there is an isomorphism $A^{\vee} \cong \bfk g_A^{-1} \otimes A$ of right $A$-modules in $\mathcal{C}$. The same argument applied to $B^{\op} \in \mathcal{C}^{\rev}$ yields an isomorphism ${}^{\vee\!}B \cong B \otimes \bfk g_B^{-1}$ of left $B$-modules in $\mathcal{C}$. Hence, by Theorem \ref{thm:Nakayama-for-modules}, we have natural isomorphisms
  \begin{gather*}
    \Nak_{{}_A\mathcal{C}_B}(M)
    \cong \Nak_{\mathcal{C}}(A^{\vee} \otimes_A M \otimes_B {}^{\vee}B)
    \cong \Nak_{\mathcal{C}}(\bfk g_A^{-1} \otimes M \otimes \bfk g_B^{-1}) \\
    \cong \bfk g_A^{-1} \otimes M^{(S^{2})} \otimes \bfk g_B^{-1} \otimes \bfk g_H^{} \cong \Nak(M) \quad (M \in {}_A \mathcal{C}_B)
  \end{gather*}
  of right $H$-comodules.

  Since the twisted left $\mathcal{C}$-module structure of $\Nak_{\mathcal{C}}$ is the identity map, it is easy to see that the isomorphism $\Nak_{{}_A\mathcal{C}_B}(M) \cong \Nak(M)$ obtained in the above preserves left action of $A$.
  The right action of $B$ on $\Nak_{{}_A\mathcal{C}_B}(M)$ is defined via the twisted $\mathcal{C}$-module structure of $\Nak_{\mathcal{C}}$ ({\it cf}. Remark~\ref{rem:Nakayama-twisted-module-structure}).
  Thus, when an element $b \in B$ acts on $\Nak(M)$, it is first affected by the Radford isomorphism~\eqref{eq:Radford-iso-H-comod}, and then acts on $M$ through the Nakayama automorphism of $B^{\op} \in \mathcal{C}^{\rev}$, that is, $\nu_B^{-1}$. Therefore the action of $B$ on $\Nak(M)$ is given as stated. The proof is done. 
\end{proof}

Now we consider the case where $A = B$.

\begin{theorem}
  \label{thm:A-bimod-modular}
  Let $A$ be an algebra in $\mathcal{C}$ and suppose that there is an element $g_A \in G(H)$ and a non-degenerate $g_A$-cointegral $\lambda_A : A \to \bfk$. Then the modular object of ${}_A \mathcal{C}_A$ is the vector space $\boldsymbol{\alpha} := A$ equipped with the following structure morphisms: The left $H$-coaction on $\boldsymbol{\alpha}$ is given by
  \begin{equation*}
    x \mapsto g_A^{-2} g_H^{} x_{(-1)} \otimes x_{(0)}
    \quad (x \in \boldsymbol{\alpha}),
  \end{equation*}
  where $a \mapsto a_{(-1)} \otimes a_{(0)}$ is the original left coaction of $H$ on $A$.
  The left and the right actions of $A$ on $\boldsymbol{\alpha}$ are given by
  \begin{equation*}
    a \triangleright x \triangleleft b
    = \nu_A^2(a \leftharpoonup \overline{\alpha}_H) \cdot x \cdot b
    \quad (a, b \in A, x \in \boldsymbol{\alpha}),
  \end{equation*}
  where $\nu_A$ is the Nakayama automorphism of $A$ with respect to $\lambda_A$.
\end{theorem}
\begin{proof}
  Let, in general, $\beta : H \to \bfk$ be an algebra map.
  By \eqref{eq:H-comod-alg-Nakayama}, we have
  \begin{equation}
    \label{eq:A-bimod-modular}
    \nu_A(a) \leftharpoonup \beta
    = \langle \beta, g_A S^{-2}(a_{(-1)}) g_A^{-1} \rangle \nu_A(a_{(0)})
    = \nu_A(a \leftharpoonup \beta)
  \end{equation}
  for all $a \in A$.
  Since the unit object of ${}_A \mathcal{C}_A$ is $A$, the modular object of ${}_A \mathcal{C}_A$ is $\Nak(A)$, where $\Nak : {}_A \mathcal{C}_A \to {}_A \mathcal{C}_A$ is the functor introduced in the above. By \eqref{eq:A-bimod-modular}, it is easy to see that the bijective linear map
  \begin{equation*}
    \xi: \boldsymbol{\alpha} \to \Nak(A), \quad
    x \mapsto \nu_A^{-1}(x \leftharpoonup \alpha_H)
    \quad (x \in \boldsymbol{\alpha})
  \end{equation*}
  preserves the left and the right actions of $A$. Let $\delta_{\Nak(A)}$ denote the left coaction of $H$ on $\Nak(A)$. The map $\xi$ also preserves the left coaction of $H$. Indeed, for $x \in \boldsymbol{\alpha}$, we have
  \begin{align*}
    (\id_{H} \otimes \xi^{-1})
    & \delta_{\Nak(A)} \xi(x)
      = \langle \alpha_H, x_{(-1)} \rangle
      (\id_{H} \otimes \xi^{-1})\delta_{\Nak(A)}(\nu_A^{-1}(x_{(0)})) \\
    & = \langle \alpha_H, x_{(-1)} \rangle
      g_A^{-1}S^2(\nu_A^{-1}(x_{(0)})_{(-1)}) g_A^{-1} g_H \otimes \xi^{-1}(\nu_A^{-1}(x_{(0)})_{(0)}) \\
    & = \langle \alpha_H, x_{(-2)} \rangle
      g_A^{-2} S^4(x_{(-1)}) g_H \otimes \xi^{-1}(\nu_A^{-1}(x_{(0)})) \\    
    & = g_A^{-2} S^4(x_{(-1)} \leftharpoonup \alpha_H) g_H^{} \otimes (x_{(0)} \leftharpoonup \overline{\alpha}_H) \\
    & = g_A^{-2} g_H^{} (\alpha_H \rightharpoonup x_{(-1)})  \otimes (x_{(0)} \leftharpoonup \overline{\alpha}_H) \\
    & = g_A^{-2} g_H^{} x_{(-1)} \otimes x_{(0)}
  \end{align*}
  by \eqref{eq:H-comod-alg-Nakayama-inverse} and \eqref{eq:Radford-S4}.
  Thus $\boldsymbol{\alpha}$ is an object of ${}_A \mathcal{C}_A$ that is isomorphic to $\Nak(A)$. The proof is done.
\end{proof}

\begin{corollary}
  \label{cor:A-bimod-unimodular-1}
  Notations are as above. Then ${}_A \mathcal{C}_A$ is unimodular if and only if there is an invertible element $g \in A$ such that the equations
  \begin{gather}
    \label{eq:unimo-condition-1}
    g a = \nu_A^2(a \leftharpoonup \overline{\alpha}_H) g, \\
    \label{eq:unimo-condition-2}
    \delta_{A}(g)
    = g_H^{-1} g_A^{2} \otimes g
  \end{gather}
  hold for all elements $a \in A$.
\end{corollary}
\begin{proof}
  Let $\boldsymbol{\alpha}$ be the modular object of ${}_A \mathcal{C}_A$ given by the previous theorem.
  We shall discuss when $\boldsymbol{\alpha}$ is isomorphic to $A$ as objects of ${}_A \mathcal{C}_A$. Given an element $g \in A$, we define $\phi_g : A \to \boldsymbol{\alpha}$ by $\phi_g(a) = g a$ ($a \in A$). The map $g \mapsto \phi_g$ gives a bijection between the set $A$ and the set of right $A$-linear maps $A \to \boldsymbol{\alpha}$. It is easy to see that $\phi_g$ is invertible if and only if $g$ is, $\phi_g$ preserves left $A$-action if and only if \eqref{eq:unimo-condition-1} holds, and $\phi_g$ preserves the left $H$-coaction if and only if \eqref{eq:unimo-condition-2} holds. The claim is proved by summarizing the discussion so far.
\end{proof}

\begin{corollary}
  \label{cor:A-bimod-unimodular-2}
  Let $A$ be a left coideal subalgebra of $H$.
  Suppose that there are an element $g_A \in G(H)$ and a non-zero $g_A$-cointegral on $A$. Then ${}_A \mathcal{C}_A$ is unimodular if and only if $g_H^{-1} g_A^2 \in A$ and $\varepsilon \circ \nu_A^2 = \alpha_H|_A$.
\end{corollary}
\begin{proof}

  Suppose that there is an element $g \in A^{\times}$ satisfying \eqref{eq:unimo-condition-1} and \eqref{eq:unimo-condition-2}.
  We may assume $\varepsilon(g) = 1$ by renormalizing $g$.
  By applying $\id_{H} \otimes \varepsilon$ to the both sides of \eqref{eq:unimo-condition-2}, we obtain $g_H^{-1} g_A^2 = g \in A$.
  We also have
  \begin{gather*}
    \varepsilon \nu_A^2 \rightharpoonup a
    \mathop{=}^{\eqref{eq:coideal-sub-Nakayama}}
    S^{-4}(g_A^{-2} \nu_A^2(a)g_A^2)
    \mathop{=}^{\eqref{eq:unimo-condition-1}}
    S^{-4}(g_A^{-2} g (a \leftharpoonup \alpha_H) g^{-1} g_A^2) \\
    = S^{-4}(g_H^{-1} (a \leftharpoonup \alpha_H) g_H)
    \mathop{=}^{\eqref{eq:Radford-S4}}
    \alpha_H \rightharpoonup a
  \end{gather*}
  for all $a \in A$, and thus $\varepsilon \nu_A^2 = \alpha_H|_A$.
  The `only if' part is proved.
  The `if' part is proved by showing that $g = g_H^{-1} g_A^2$ satisfies \eqref{eq:unimo-condition-1} and \eqref{eq:unimo-condition-2} straightforwardly.
\end{proof}

\subsection{Examples of determining the unimodularity}
\label{subsec:examples-unimodularity}

From now on till the end of this paper, we assume that $\bfk$ is an algebraically closed field of characteristic zero. We fix an odd integer $N > 1$ and a primitive $N$-th root $q \in \bfk$ of unity, and consider the small quantum group $u_q(\mathfrak{sl}_2)$.
We recall that, as an algebra, it is generated by $E$, $F$ and $K$ subject to the relations
$K^N = 1$, $E^N = F^N = 0$, $K E = q^2 F K$, $K F = q^{-2} F K$ and
$E F - F E = (K - K^{-1})/(q-q^{-1})$.
The Hopf algebra structure of $u_q(\mathfrak{sl}_2)$ is determined by
\begin{equation*}
  \Delta(K) = K \otimes K,
  \quad \Delta(E) = E \otimes K + 1 \otimes E,
  \quad \Delta(F) = F \otimes 1 + K^{-1} \otimes F.
\end{equation*}

The set $\{ E^r F^s K^t \mid r, s, t = 0, 1, \cdots, N - 1 \}$ is a basis of $U := u_q(\mathfrak{sl}_2)$.
A formula of the comultiplication of each of elements of this basis is written explicitly in, {\it e.g.}, \cite[Chapter VII]{MR1321145}. One can check that the linear map
\begin{equation*}
  \lambda_U : U \to \bfk,
  \quad
  E^r F^s K^t \mapsto \delta_{r, N-1} \delta_{s, N-1} \delta_{t, 1}
  \quad (r, s, t = 0, 1, \cdots, N - 1)
\end{equation*}
is a right cointegral on $U$. Furthermore, we have $g_{U}^{} = K^2$.
It is also known that $U$ is unimodular, {\it i.e.}, $\alpha_U = \varepsilon$.

Below, for some algebras $A$ in $\mathcal{C} := \lcomod{U}$, we determine whether ${}_A \mathcal{C}_A$ is unimodular.
In view of Remark~\ref{rem:relative-Hopf-bimodules}, we may say that we will discuss whether the dual of $\lmod{U}$ with respect to $\lmod{A}$ is unimodular.

\begin{example}
  We fix $\xi \in \bfk$ and a divisor $d$ of $N$.
  We consider the algebra $A$ generated by $G$ and $Y$ subject to the relations $G^d = 1$, $G Y = q^{-2m} Y G$ and $Y^N = \xi 1_A$, where $m = N/d$. The algebra $A$ is a left $U$-comodule algebra by the left $U$-coaction determined by $G \mapsto K^m$ and $Y \mapsto F \otimes 1 + K^{-1} \otimes Y$. In a similar way as \cite[Section 5]{2019arXiv190400376S}, we see that the linear map
  \begin{equation*}
    \lambda_A : A \to \bfk,
    \quad Y^r G^s \mapsto \delta_{r,N-1} \delta_{s,0}
    \quad (r = 0, \cdots, N-1; s = 0, \cdots, d - 1)
  \end{equation*}
  is a non-degenerate $g_A$-cointegral on $A$ with $g_A = K$. The Nakayama automorphism $\nu_A$ of $A$ with respect to $\lambda_A$ is given by $\nu_A(G) = q^{2m} G$ and $\nu_A(Y) = Y$.

  Corollary \ref{cor:A-bimod-unimodular-1} is effective to determine whether ${}_A \mathcal{C}_A$ is unimodular. Indeed, if $g$ is an invertible element of $A$ satisfying the condition \eqref{eq:unimo-condition-2} of the corollary, then $g$ must be $1_A$ since $g_U^{-1} g_A^2 = 1_U$. Thus ${}_A \mathcal{C}_A$ is unimodular if and only if $\nu_A^2 = \id_A$, or, equivalently, $d = 1$ (because the order of $q$ is odd).
\end{example}

\begin{example}
  We choose a parameter $\xi \in \bfk$ and consider the subalgebra $A$ of $U$ generated by $Y := F + \xi K^{-1}$. Since $\Delta(Y) = F \otimes 1 + K^{-1} \otimes Y$, the algebra $A$ is in fact a left coideal subalgebra of $U$.
  The linear map
  \begin{equation*}
    \lambda_A : A \to \bfk, \quad Y^r \mapsto \delta_{r, N - 1}
    \quad (r = 0, 1, \cdots, N - 1)
  \end{equation*}
  is a non-degenerate $g_A$-cointegral on $A$ with $g_A = K$.
  By Corollary \ref{cor:A-bimod-unimodular-2}, we see that ${}_A \mathcal{C}_A$ is always unimodular (this is actually the case where $d = 1$ in the previous example).
\end{example}

\begin{example}
  We choose a divisor $m$ of $N$ and consider the subalgebra $A$ of $U$ generated by $G := K^m$. The subalgebra $A$ is a left coideal subalgebra of $U$ (which is, in fact, a Hopf subalgebra of $U$). The linear map
  \begin{equation*}
    \lambda_A : A \to \bfk, \quad G^r \mapsto \delta_{r, 0}
    \quad (r = 0, 1, \cdots, N/m - 1)
  \end{equation*}
  is a non-degenerate $g_A$-cointegral on $A$ with $g_A = 1_A$.
  By Corollary \ref{cor:A-bimod-unimodular-2}, ${}_A \mathcal{C}_A$ is unimodular if and only if $m = 1$.
\end{example}

\subsection{Coideal subalgebras without grouplike-cointegrals}
\label{subsec:no-g-cointegrals}

Corollaries~\ref{cor:A-bimod-unimodular-1} and \ref{cor:A-bimod-unimodular-2} are not applicable to comodule algebras without non-degenerate grouplike-cointegrals.
We close this paper by giving an example of coideal subalgebras without non-zero grouplike-cointegrals.

We fix an odd integer $N > 1$ and a primitive $N$-th root $q \in \bfk$ of unity, and define $H$ to be the algebra generated by $a$, $b$, $c$ and $d$ subject to the following relations:
\begin{gather*}
  b a = q a b, \quad
  c a = q a c, \quad
  d b = q b d, \quad
  d c = q c d, \quad
  b c = c b, \\
  a d - q^{-1} b c = d a - q b c = 1, \quad
  a^N = d^N = 1, \quad b^N = c^N = 0.
\end{gather*}
The algebra $H$ is a Hopf algebra with the comultiplication defined so that $a$, $b$, $c$ and $d$ form a matrix coalgebra.
It is known that $H$ is dual to $U := u_q(\mathfrak{sl}_2)$.
A detailed account is found in, {\it e.g.}, \cite[VII.4]{MR1321145} and \cite[Appendix A]{MR3926231}.
Specifically, there is an isomorphism $\phi : H \to U^*$ of Hopf algebras such that
\begin{equation*}
  \begin{pmatrix}
    \phi(a)(u) & \phi(b)(u) \\
    \phi(c)(u) & \phi(d)(u)
  \end{pmatrix}
  = \rho(u) \quad(u \in U),
\end{equation*}
where $\rho : U \to \mathrm{Mat}_{2 \times 2}(\mathbb{C})$ is the algebra map determined by
\begin{equation*}
  E \mapsto \begin{pmatrix} 0 & 1 \\ 0 & 0 \end{pmatrix},
  \quad F \mapsto \begin{pmatrix} 0 & 0 \\ 1 & 0 \end{pmatrix},
  \quad K \mapsto \begin{pmatrix} q & 0 \\ 0 & q^{-1} \end{pmatrix}.
\end{equation*}
A right $H$-comodule is identified with a left $U$-comodule through the Hopf algebra isomorphism $\phi$. In particular, the action of $U$ on $H$ is computed by
\begin{gather*}
  E \cdot 1 = 0, \quad
  E \cdot a = 0, \quad
  E \cdot b = a, \quad
  E \cdot c = 0, \quad
  E \cdot d = c, \\
  F \cdot 1 = 0, \quad
  F \cdot a = b, \quad
  F \cdot b = 0, \quad
  F \cdot c = d, \quad
  F \cdot d = 0, \\
  K \cdot 1 = 1, \quad
  K \cdot a = q a, \quad
  K \cdot b = q^{-1} b, \quad
  K \cdot c = q c, \quad
  K \cdot d = q^{-1} d
\end{gather*}
and the rule $u \cdot (x y) = (u_{(1)} \cdot x) (u_{(2)} \cdot y)$ for $u \in U$ and $x, y \in H$.

For two integers $r$ and $s$, we put $v_{r s} = a^{r} b^{r - s}$ if $0 \le s \le r \le N$ and $v_{r s} = 0$ otherwise. For every $r = 0, \cdots, N$, the subspace $V_r = \mathrm{span}_{\bfk} \{ v_{r s} \mid s = 0, \cdots, r \}$ of `homogeneous polynomials' in $a$ and $b$ is an $(r+1)$-dimensional $U$-submodule of $H$ ({\it cf}. \cite[IV.7]{MR1321145}). The actions of the generators of $U$ on $V_r$ are given by
\begin{equation*}
  E \cdot v_{rs} = [s] v_{r, s - 1}, \quad
  F \cdot v_{rs} = [r - s] v_{r, s + 1}
  \quad \text{and} \quad
  K \cdot v_{rs} = q^{r - 2s} v_{rs},
\end{equation*}
where $[m] = (q^m - q^{-m}) / (q-q^{-1})$.

The $U$-submodule $A := V_{N}$ of $H$ is closed under the multiplication. Thus it is in fact a right coideal subalgebra of $H$.
The subspace of $A$ spanned by $1_A$ ($= v_{N, 0}$) is a unique non-trivial $U$-submodule of $A$.
In particular, $A$ has no quotient right $H$-comodule of dimension one.
Thus $A$, viewed as a left $H^{\cop}$-comodule algebra, admits no non-zero grouplike-cointegral.

\begin{question}
  Let $H$ be a finite-dimensional Hopf algebra, and let $A$ be an algebra in $\mathcal{C} := \lcomod{H}$, which may not admit a non-degenerate grouplike-cointegral.
  Is there an easy criterion for ${}_A \mathcal{C}_A$ to be unimodular?
\end{question}

\def\cprime{$'$}

\end{document}